\documentclass[12pt,a4paper]{article} 
\usepackage[colorlinks]{hyperref}
\usepackage{amsmath}
\usepackage{amsthm}
\usepackage{amsfonts}
\usepackage{amssymb}
\usepackage{mathabx}
\usepackage[mathscr]{eucal}
\usepackage{bussproofs}
\usepackage{lineno}
\usepackage{quiver}
\usepackage{float}
\usepackage{authblk}
\usepackage{color}
\usepackage{quiver}
\usepackage{fdsymbol}
\usepackage{tikzsymbols}

\theoremstyle{plain} 
\newtheorem{theorem}{Theorem}[section]

\usepackage{tikz}

\newtheorem{lemma}[theorem]{Lemma}

\newtheorem{corollary}[theorem]{Corollary}
\theoremstyle{definition}
\newtheorem{definition}[theorem]{Definition}
\newtheorem{example}[theorem]{Example}
\newtheorem{remark}[theorem]{Remark}

\newcommand{\PA}{\mathrm{PA}}

\newcommand{\PV}{\mathrm{PV}}

\newcommand{\TI}{\mathrm{TI}}

\newcommand{\PRA}{\mathrm{PRA}}

\newcommand{\NP}{\mathrm{NP}}

\newcommand{\TFNP}{\mathrm{TFNP}}
\newcommand{\PRWO}{\mathrm{PRWO}}

\newcommand{\PLS}{\mathrm{PLS}}

\newcommand{\LS}{\mathrm{LS}}

\newcommand{\TSP}{\mathrm{TSP}}
\newcommand{\PInd}{\mathrm{PInd}}
\newcommand{\Ind}{\mathrm{Ind}}
\newcommand{\LInd}{\mathrm{LInd}}
\newcommand{\wPInd}{\mathrm{wPInd}}
\newcommand{\wInd}{\mathrm{wInd}}

\begin{document}

\title{Witnessing Flows in Arithmetic} 

\author[]{Amirhossein Akbar Tabatabai\footnote{Supported by the Czech Academy of Sciences (RVO 67985840). The support by the FWF project P 33548 is also gratefully acknowledged.}}
\affil[]{Institute of Mathematics, Czech Academy of Sciences}

\date{}

\maketitle

\begin{abstract}
One of the elegant achievements in the history of proof theory is the characterization of the provably total recursive functions of an arithmetical theory by its proof-theoretic ordinal as a way to measure the time complexity of the functions. Unfortunately, 
the machinery is not sufficiently fine-grained to be applicable on the weak theories on the one hand and to capture the bounded functions with bounded definitions of strong theories, on the other.
In this paper, we develop such a machinery to address the bounded theorems of both strong and weak theories of arithmetic. In the first part, we provide a refined version of ordinal analysis to capture the feasibly definable and bounded functions that are provably total in $\PA+\bigcup_{\beta \prec \alpha} \TI(\prec_{\beta})$, the extension of Peano arithmetic by transfinite induction up to the ordinals below $\alpha$. Roughly speaking, we identify the functions as the ones that are computable by a sequence of $\PV$-provable polynomial time modifications on an initial polynomial time value, where the computational steps are indexed by the ordinals below $\alpha$, decreasing by the modifications. In the second part, and choosing $l \leq k$, we use similar technique to capture the functions with bounded definitions in the theory $T^k_2$ (resp. $S^k_2$) as the functions computable by exponentially (resp. polynomially) long sequence of $\PV_{k-l
+1}$-provable reductions between $l$-turn games starting with an explicit $\PV_{k-l
+1}$-provable winning strategy for the first game. \\

\noindent \textbf{Keywords}. Total search problems, ordinal analysis, bounded arithmetic, local search programs.

\end{abstract}

\section{Introduction}

One of the elegant achievements in the history of proof theory is the \emph{witnessing} techniques connecting the provability of a formula of a certain form to the existence of a computational entity (algorithm \cite{HandbookRealizability}, function \cite{ProvablyRecursiveFunctions}, term in a type theory \cite{HandbookDialectica}, etc.) that \emph{witnesses} the truth of the formula. These connections identify the power of the theories and they are useful to establish the unprovability of a formula by showing the non-existence of the corresponding witness. As an example, consider the \emph{ordinal analysis} as one of the well-known witnessing techniques that among many other things provides a characterization for the provably total recursive functions of some mathematical theories \cite{Kreisel,BussPA,ProvablyRecursiveFunctions}. (For a comprehensive high-level explanation, see \cite{Realm}). It connects the provability of the totality of a $\Sigma^0_1$-definable function to its time complexity, measured by the proof-theoretic ordinal of the theory. The characterization then leads to some independence results for the formulas in the form $A=\forall x \exists y B(x, y)$, where $B \in \Sigma^0_1$ is a definition of a function with a faster growth rate and hence higher time complexity than what the theory can actually reach \cite{ProvablyRecursiveFunctions}.

There are, however, some settings in which the witnessing techniques and especially the one based on ordinal analysis break down. 
%The reasons can be abound. 
Sometimes, we are only interested in the formulas with no existential quantifiers to witness (e.g. $A=\forall x B(x)$, where $B$ is a quantifier-free formula). Other times, %even if the formula has some existential quantifiers, 
the theory is so weak that even the basics of the witnessing machinery goes beyond the power of the theory. Even working with powerful theories, there can be some problematic situations. For instance, one may be interested in bounded formulas (e.g. $\forall x \exists y \leq t(x) B(x, y)$, where all quantifiers in $B$ are also bounded) provable in Peano Arithmetic, denoted by $\PA$. Here, what the usual witnessing methods provide is rather weak or even useless. For instance, using ordinal analysis for $\PA$, the best thing we can learn in the bounded setting is the existence of an algorithm to compute $y$ using a huge amount of time measured by $\epsilon_0$, the ordinal of the theory. This is much weaker than what we started with, i.e., the provability of the totality of the function with a \emph{bounded} definition. The reason roughly is that the algorithm leads to the existence of the definition $\exists w C(x, w, y)$ for the function, where $w$ encodes the computation and $\PA$ proves $\forall x \exists yw C(x, w, y)$. However, the computation $w$ can be huge and hence unbounded by the terms in the language and in this sense proving the totality of a \emph{bounded} function with a \emph{bounded} definition is stronger than the existence of such an algorithm. 

To solve this type of issues and to address both weak theories and low complexity formulas, many new witnessing techniques were designed, from witnessing the universal provable formulas by short propositional proofs \cite{Cook,paris19810,PavKra,BussPropProofComplexity} to witnessing provable bounded formulas in first-order bounded theories of arithmetic in special cases \cite{BussThesis,PLS,JanSkellyThapen} and then in general cases \cite{SkellyThapen,thapen,BBNotation}, using game reductions and different versions of local search problems. A similar technique is also developed for second-order bounded theories of arithmetic \cite{BBFeF,KrajR,BBV,Leszek} and even for Peano arithmetic \cite{Be}. 
In this paper, we will continue this line of research by providing a general witnessing machinery to witness the low-complexity theorems both in strong and weak theories of arithmetic using a computational entity that we call a flow. Flows are meant to formalize the idea of \emph{flowing} information and they formally are uniform suitably long sequences of $\PV$-provable implications between formulas in a suitable class, where $\PV$ is Cook's theory for polynomial time functions. We will work with two different types of flows in this paper, \emph{ordinal flows} and \emph{$k$-flows}.

\subsubsection*{Ordinal flows}
An ordinal flow is a transfinite uniform sequence of $\PV$-provable implications between universal formulas.
We use ordinal flows to witness low-complexity theorems of the theory $\PA+\bigcup_{\beta \prec \alpha} \TI(\prec_{\beta})$, where $\alpha$ is an ordinal with a certain polynomial time representation and $\TI(\prec_{\beta})$ means the transfinite induction up to the ordinal $\beta$. 
More precisely, we witness 
the provability of an implication between two universal formulas in $\PA+\bigcup_{\beta \prec \alpha}\TI(\prec_{\beta})$ by a uniform sequence of $\PV$-provable implications of length $\beta \prec \alpha$. Using Herbrand's theorem for $\PV$, we push the witnessing further to witness
the $\PA+\bigcup_{\beta \prec \alpha}\TI(\prec_{\beta})$-provable formulas in the form $A=\forall \bar{x}\exists \bar{y} B(\bar{x},\bar{y})$, where $B$ is a polynomial time computable predicate
by an algorithm to compute $\bar{y}$ by a sequence of $\PV$-provable polynomial time modifications on an initial polynomial time value, where the computational steps are indexed by the ordinals below $\alpha$, decreasing by the modifications. Our result generalizes the main theorem of \cite{Be} that developed a similar characterization for $\PA$. However, as we will explain below, even for that special case, we use a simpler and easier to generalize methodology.

To compare our result to the existing literature on ordinal analysis, it is important to focus on the role of the polynomial time computable functions and the theory $\PV$ in our contribution. First, note that changing the polynomial time functions and $\PV$ in our characterization to the elementary or primitive recursive functions and $\mathrm{ERA}$ or $\PRA$, respectively, makes the characterization an easy consequence of the known facts in the ordinal analysis literature. For instance, one can use the powerful witnessing theorems in \cite{Fri,Avigad} or the interesting algebraic presentation of the ordinals in \cite{Lev}. What is not trivial, though, is providing a low-complexity version suitable to witness the low-complexity theorems of arithmetic. To reach such a version, we have two options. The first, as followed in the above-mentioned paper \cite{Be}, rewrites the continuous cut elimination technique \cite{Buc1,Buc2}, replacing all primitive recursive functions by more careful polynomial time computable operations \cite{BBP}. The second as an indirect approach uses the 
known results in ordinal analysis as a \emph{black-box} and rewitness them in a feasible manner to circumvent redoing the tedious ordinal analysis argument. This option is what we follow in the present paper. More precisely, we first use the refined ordinal analysis in \cite{Fri} to show that a $\Pi^0_2$-formula is provable in the theory  $\PA+\bigcup_{\beta \prec \alpha}\TI(\prec_{\beta})$ iff it is provable in an extension of $\PRA$ with a weak form of transfinite induction. Then, using a suitable polynomial time representation for the ordinals below $\alpha$, we will transform a proof in the weaker theory to a sequence of $\PV$-provable polynomial time modifications described above. Our technique of using ordinally long sequence of easy modifications is similar to what used in \cite{Avigad}, although its machinery has a more model-theoretic character and also implements the ordinal analysis from the scratch. Roughly speaking, \cite{Avigad} provides a similar witnessing theorem using elementary functions rather than polynomial time functions in its ordinal flows. However, to have a verifiablity criterion, it insists 
on having the whole witnessing process provable inside the meta-theory $\PRA$. The witnessing machinery of \cite{Avigad} cannot be directly used to prove the low-complexity version we are interested in here. The reason is its use of $\PRA$-formalized Herbrand's theorem for first-order logic that uses cut elimination and it is extremely costly to be directly formalizable in $\PV$. To solve the issue, as \cite{Avigad} also suggests, one must witness the Herbrand's theorem part by a sequence of $\PV$-verifiable modifications or equivalently witness the first-order logic by such modifications, directly. This is one of the things we do in the present paper.
Therefore, although our work is inspired by  \cite{Be} and the witnessing theorems in bounded arithmetic and hence its technique was developed independent from \cite{Avigad},  one can interpret our contribution as a generalization of \cite{Avigad} making its machinery applicable even in the low-complexity settings.

\subsubsection*{$k$-flows}

A (polynomial) $k$-flow is a uniform (polynomially) exponentially long sequence of $\PV$-provable implications between $\hat{\Pi}^b_k$-formulas. Recall that $\hat{\Pi}^b_k$- ($\hat{\Sigma}^b_k$-formulas) are roughly the formulas with $k$-many bounded quantifier blocks starting with a universal (existential) block and followed by a quantifier-free formula over the language $\mathcal{L}_{\PV}$ that has a term for any polynomial time computable function. %and they represent the predicates in the $k$-th level of the polynomial hierarchy.
We will witness the provability of an implication between $\hat{\Pi}^b_k$-formulas in $T^k_2$ (resp. $S^k_2$)
by a $k$-flow (resp. polynomial $k$-flow). To push the witnessing further, we can use Herbrand's theorem again for the universal theory $\PV$. However, this time the formulas are in $\hat{\Pi}^b_k$ and hence we have $k$-many layers of quantifier to peel off.
%and the witnessing that Herbrand's theorem provides can be too complex to use. 
To control the number of layers we intend to remove, we will follow a relative approach. We fix a  number $l \leq k$ and only peel off the outmost $l$ many quantifier blocks. More precisely, we first move the $\PV$-provable implications from $\PV$ to $\PV_{k-l+1}$, a universal theory for the functions in the $(k-l+1)$-th level of the polynomial hierarchy. This way we can pretend that all the formulas in $\hat{\Sigma}^b_{k-l} \cup \hat{\Pi}^b_{k-l}$ are quantifier-free. Therefore, only $l$ many quantifier blocks are left to witness. Using Herbrand's theorem for the theory $\PV_{k-l+1}$ and reading any quantifier-free formula in the language of $\PV_{k-l+1}$ as an $l$-turn game \cite{SkellyThapen}, we can then witness any $\PV$-provable implication by an explicit $\PV_{k-l+1}$-verifiable reduction between $l$-turn games. These reductions are somewhat \emph{non-deterministic} mapping their input values to \emph{some} possible instances, where one of the options may work, (see the second part in Theorem \ref{Herbrand} to see what we mean by non-determinism in this context).
Finally, using these reductions,
we show that a formula in the form $\forall \bar{x} \exists y \leq r(\bar{x}) B(\bar{x}, y)$, where $B \in \hat{\Sigma}^b_{k-l} \cup \hat{\Pi}^b_{k-l}$ is provable in $T^k_2$ (resp. $S^k_2$) iff there is a uniform (polynomially) exponentially long sequence of $\PV_{k-l+1}$-verifiable reductions between $l$-turn games, starting from an explicit $\PV_{k-l+1}$-verifiable winning strategy for the first game. We will only spell out the details for $l=1,2$. For $l=1$, we show that our witnessing theorem reproves some of the well-known witnessing theorems for $S^k_2$ and $T^k_2$ including the usual witnessing of $\hat{\Sigma}^b_k$-definable functions of $S^k_2$ by $\Box^p_k$-functions \cite{BussThesis} and $\hat{\Sigma}^b_1$-definable multifunctions of $T^1_2$ by polynomial local search problems \cite{PLS}. For $l=2$, we provide new witnessing theorems. For $T^k_2$, there are other witnessing methods providing similar characterizations as ours based on better (i.e., deterministic) game reductions 
\cite{SkellyThapen,thapen}. The theory of flows can also prove these stronger characterizations. However, it needs to work with more involved notions of a $k$-flow than what we have here. We leave such investigations to another paper. 
For $S^k_2$, 
however, our result, to the best of our knowledge, is the only characterization in the same style of the original witnessing theorems \cite{BussThesis} that reduce the provability in $S^k_2$ to a \emph{polynomially} long sequence of feasible modifications.
Of course, one can use the
conservativity of $S^k_2$ over $T^{k-1}_2$ for $\hat{\Sigma}^b_{k}$-formulas and then using the witnessing for $T^{k-1}_2$ by the deterministic game reductions \cite{SkellyThapen,thapen} or any other characterization \cite{BBPLS,BBNotation}, find a witnessing theorem for $S^k_2$. Using this approach, 
the characterizations 
provide an exponentially long sequence of deterministic reductions while we provide a polynomially long sequence of more complex non-deterministic reductions. These two different approaches can be seen as an instance of the usual phenomenon of simulating the huge power of the deterministic exponential time with polynomial time non-determinism, where the latter, if possible, is more informative than the former. 

Finally, to compare our witnessing method to the rich literature on witnessing theorems in bounded arithmetic, let us emphasize two points that we find unique to our characterization. First, unlike the methods used in \cite{PLS,JanSkellyThapen,SkellyThapen,thapen,BBPLS,BBNotation}, our machinery is sufficiently general to \emph{directly} witness bounded theories arising from practically \emph{any} type of bounded induction  \cite{ILI}. For instance, for any $m \geq 2$, consider the language $\mathcal{L}_{\PV} \cup \{\#_m\}$, where $x\#_2 y=2^{|x||y|}$ and $x \#_{i+1} y=2^{|x| \#_i |y|}$ and define the class $\hat{\Pi}^b_k(\#_m)$ and the theory $\PV(\#_m)$ over the new language similar to $\hat{\Pi}^b_k$ and $\PV$ over $\mathcal{L}_{\PV}$. Now, for any $n \geq 0$, $m \geq n+2$, and $k \geq 1$, define the theory $R^k_{m, n}$ as the extension of a basic universal theory to handle the function symbols, by the induction axiom
\[
A(0) \wedge \forall x (A(x) \to A(x+1)) \to \forall x A(|x|_n)
\]
where $A \in \hat{\Pi}^b_k(\#_m)$, $|x|_0=x$ and $|x|_{j+1}=||x|_j|$. It is easy to imitate our technique in the present paper to witness $R^k_{m,n}$-provable implications between $\hat{\Pi}^b_k(\#_m)$-formulas by a uniform sequence of $\PV(\#_m)$-provable implications between $\hat{\Pi}^b_k(\#_m)$-formulas with the length $|t|_n$, for some term $t$. This can be even more generalized to any type of induction satisfying some basic properties \cite{ILI}.

The second point is that the length of our witnessing flows \emph{honestly} reflects the type of the induction we use. For instance, for $S^k_2$ and $T^k_2$, we use polynomially long and exponentially long $k$-flows, respectively and more generally, in $R^k_{m, n}$ where the induction is up to $|x|_n$, the length of the witnessing flow is $|t|_n$, for some term $t$, see \cite{ILI}. This honest correspondence is not typical with the above-mentioned characterizations. For instance, the polynomially long adaptation of the known characterizations for $T^k_2$ \cite{SkellyThapen,thapen}, i.e., polynomially long sequence of $\PV$-verifiable deterministic reductions between $k$-turn games, does not witness $S^k_2$-provable implications. The reason is that any polynomially long iteration of a deterministic reduction is again a deterministic reduction itself. Therefore, if such a witnessing theorem holds, one can witness the implications in $S^k_2$ between $\hat{\Pi}^b_k$-formulas by a polynomially long sequence of reductions and hence only one reduction. Thus, the $\hat{\Sigma}^b_k$-definable functions of $S^k_2$ must be all polynomial time computable and as all the functions in $\Box^p_k$ are $\hat{\Sigma}^b_k$-definable in $S^k_2$, the polynomial hierarchy must collapse. This simple observation shows that the non-determinism we use in our reductions is essential to have an honest characterization. Moreover, it shows that our characterization for $S^k_2$ is not a simple consequence of the methodologies used for $T^k_2$ in \cite{SkellyThapen,thapen} or even in \cite{BBPLS,BBNotation}. \\
 
Here is the structure of the paper. In Section \ref{Preliminaries}, we recall the basic definitions of different languages and arithmetical systems we use in this paper. In Section \ref{PTimeRepresentation}, we introduce our version of polynomial time ordinal representation and we recall the one introduced in \cite{BBP} for $\epsilon_0$. % In Section \ref{SectionOfTI}, we introduce an auxiliary system axiomatized by transfinite induction on the universally quantified polynomial time computable formulas. We also cover its relation to the theories we are interested in and its sequent calculus. Finally, i
In Section \ref{OrdinalFlowsandArithmetic}, we present ordinal flows and the witnessing technique to reduce the provability of the low complexity statements in the theory $\PA+\bigcup_{\beta \prec \alpha}\TI(\prec_{\beta})$.
%to a sequence of $\beta$ many $\PV$-implications, for some $\beta \prec \alpha$. Using the usual witnessing theorem for the universal theory $\PV$, we can prove the characterization that we have mentioned before. 
Finally, in Section \ref{kFlowsAndArithmetic}, we introduce $k$-flows to witness the provability of the low complexity statements in the theories $S^k_2$ and $T^k_2$.
%to an exponentially long sequence of $\PV$-implications. Using the usual witnessing theorem for the universal theory $\PV$, we can prove the characterization that we have mentioned before. 

\section{Preliminaries}\label{Preliminaries}

For any first-order language $\mathcal{L}$, by an $\mathcal{L}$-formula, we mean any expression constructible by the connectives $\{\wedge, \vee, \forall, \exists\}$ from the atomic formulas (including $\bot$ and $\top$) and their negations. The formula $\neg A$ is defined via de Morgan laws and $A \to B$ is an abbreviation for $\neg A \vee B$. By an $\mathcal{L}$-term, we simply mean a term in the language $\mathcal{L}$. By $\bar{t}$, we mean a sequence of terms in the language and $\bar{x}$ means a sequence of variables.\\
To introduce the system $\PV$, let us recall Cobham's machine-independent characterization of polynomial-time computable (\emph{ptime}, for short) functions \cite{Cobham}. It states that a function is ptime iff it is constructible from certain basic functions by composition and a weak sort of recursion called the \emph{bounded recursion on notation}. Any such construction provides an algorithm to compute the corresponding ptime function. Let $\mathcal{L}_{\PV}$ be a first-order language with a function symbol for any such algorithm. In \cite{Cook}, Cook introduced an equational theory over the language $\mathcal{L}_{\PV}$ to reason about ptime functions. The theory essentially consists of the defining axioms for the function symbols together with a sort of induction rule. Later, a conservative first-order extension of $\PV$, denoted by $\PV_1$, was introduced \cite{Pudlak}. The theory has the \emph{polynomial induction axiom scheme}, denoted by $\PInd$
\[
A(0) \wedge \forall x (A(\lfloor \frac{x}{2} \rfloor) \to A(x)) \to \forall x A(x),
\]
for any quantifier-free formula $A(x)$ and is universally axiomatizable \cite{Pudlak}. 
%Note that as $\PV_1$ is a universal theory, one may use Herbrand's theorem followed by a gluing argument for the witnessing terms to prove that if $\PV_1 \vdash \forall \bar{x} \exists y A(\bar{x}, y)$, for a quantifier-free formula $A(\bar{x}, y)$, then there is a term $t(\bar{x})$ in $\mathcal{L}_{\PV}$ such that $\PV_1 \vdash \forall \bar{x} A(\bar{x}, t(\bar{x}))$.
In this paper, we will only use the theory $\PV_1$ and not $\PV$. Therefore, by abuse of notation, we will use the name $\PV$ to denote its first-order extension $\PV_1$. 

In any language extending $\mathcal{L}_{\PV}$, by a \emph{bounded quantifier}, we mean a quantifier in the form $\forall x (x \leq t \to A(x))$ or $\exists x (x \leq t \wedge A(x))$, abbreviated by $\forall x \leq t \, A(x)$ and $\exists x \leq t \, A(x)$, respectively. For any sequence of variables $\bar{x}=(x_1, \ldots, x_n)$ and terms $\bar{t}=(t_1, \ldots, t_n)$, by $Q \bar{x} \leq \bar{t} \, A(\bar{x})$, we mean $Q x_1 \leq t_1 Q x_2 \leq t_2 \ldots A(x_1, \ldots, x_n)$, for any $Q \in \{\forall, \exists\}$. 

By recursion on $k$, define the classes $\hat{\Sigma}^b_k$ and $\hat{\Pi}^b_k$ of $\mathcal{L}_{\PV}$-formulas in the following way:
\begin{description}
\item[$\bullet$]
$\hat{\Pi}^b_0=\hat{\Sigma}^b_0$ is the class of all quantifier-free formulas,
\item[$\bullet$]
$\hat{\Sigma}_k^b \subseteq \hat{\Sigma}^b_{k+1}$ and $\hat{\Pi}^b_k \subseteq \hat{\Pi}^b_{k+1}$,
\item[$\bullet$]
$\hat{\Pi}^b_k$ and $\hat{\Sigma}^b_k$ are closed under conjunction and disjunction,
\item[$\bullet$]
If $B(x) \in \hat{\Sigma}^b_k$ then $\exists x \leq t \; B(x) \in \hat{\Sigma}^b_k$ and $\forall x \leq t \; B(x) \in \hat{\Pi}^b_{k+1}$ and
\item[$\bullet$]
If $B(x) \in \hat{\Pi}^b_k$ then $\forall x \leq t \; B(x) \in \hat{\Pi}^b_k$ and $\forall x \leq t \; B(x) \in \hat{\Sigma}^b_{k+1}$.
\end{description}
Define $\hat{\Sigma}^b_{\infty}=\hat{\Pi}^b_{\infty}$ as $\bigcup_{k=0}^{\infty} \hat{\Sigma}^b_{k}$ that is the same as $\bigcup_{k=0}^{\infty} \hat{\Pi}^b_{k}$. For the sake of simplicity, we suppressed the free variables in our notation. However, let us emphasize that they are also allowed to be used in the formulas. 

By the axiom scheme $\hat{\Pi}^b_k-\PInd$, we mean
\[
A(0) \wedge \forall x (A(\lfloor \frac{x}{2} \rfloor) \rightarrow A(x)) \rightarrow \forall x A(x),
\]
for any $A \in \hat{\Pi}^b_k$ and by $\hat{\Pi}^b_{k}-\Ind$, we mean
\[
A(0) \wedge \forall x (A(x) \rightarrow A(x+1)) \rightarrow \forall x A(x),
\]
for $A \in \hat{\Pi}^b_k$. The schemes $\hat{\Sigma}^b_{k}-\PInd$ and $\hat{\Sigma}^b_{k}-\Ind$ are defined similarly.
For any $k \geq 1$, define the theories $S^k_2$ and $T^k_2$ as $\PV+\hat{\Pi}^b_k-\PInd$ and  $\PV+\hat{\Pi}^b_k-\Ind$, respectively. It is known that $S^k_2$ (resp., $T^k_2$) proves $\hat{\Sigma}^b_{k}-\PInd$ (resp., $\hat{\Sigma}^b_{k}-\Ind$). It is also useful to mention that the following axiom scheme, denoted by 
$\hat{\Pi}^b_k-\LInd$
\[
A(0) \wedge \forall x (A(x) \rightarrow A(x+1)) \rightarrow \forall x A(|x|),
\]
where $A \in \hat{\Pi}^b_k$, is provable in $S^k_2$. The same also holds for $\hat{\Sigma}^b_k-\LInd$, where we replace $\hat{\Pi}^b_k$ by $\hat{\Sigma}^b_k$ \cite{BussThesis,JanBook}. The following theorem is true for theories $S^k_2$ and $T^k_2$ \cite{JanBook}.
\begin{theorem}(Parikh)
Let $T$ be either $S^k_2$ or $T^k_2$, for some $k \geq 1$ and $A(\bar{x}, y)$ be an $\mathcal{L}_{\PV}$-formula in $\hat{\Sigma}^b_{\infty}$. Then, if $T \vdash \forall \bar{x} \exists y A(\bar{x}, y)$, then there exists an $\mathcal{L}_{\PV}$-term $t(\bar{x})$ such that $T \vdash \forall \bar{x} \exists y \leq t(\bar{x}) A(\bar{x}, y)$.
\end{theorem}

It is possible to define a universal theory for any level in the polynomial hierarchy, similar to what $\PV_1$ does for the polynomial time computable functions. More precisely, for any $k \geq 2$, one can define a universal theory $\PV_{k}$ over an extended language $\mathcal{L}_{\PV_k}$ that has a term for any function in the $k$-th level of the polynomial hierarchy, denoted by $\Box^p_k$ \cite{Pudlak}. We do not spell out the details of these theories. The only thing we need to know is that $\PV_k$ has an explicit term for the characteristic functions of $\hat{\Sigma}^b_k$-formula and its term construction allows defining functions by bounded recursion on notation \cite{Pudlak,JanBook}. As $\PV_k$ is universal, it enjoys Herbrand's theorem \cite{BussProofTheory,JanBook}:
\begin{theorem}\label{Herbrand}(Herbrand) Let $A(\bar{x}, y)$ and $B(\bar{x}, y, z)$ be two quantifier-free $\mathcal{L}_{\PV_k}$-formulas. Then: 
\begin{itemize}
    \item[$\bullet$]
If $\PV_k \vdash \exists y A(\bar{x}, y)$ then there exists an $\mathcal{L}_{\PV_k}$-term $f(\bar{x})$ such that $\PV_k \vdash A(\bar{x}, f(\bar{x}))$. 
    \item[$\bullet$]
If $\PV_k \vdash \exists y \forall z B(\bar{x}, y, z)$ then there are $\mathcal{L}_{\PV_k}$-terms $f_0(\bar{x})$, $f_1(\bar{x}, z_0)$, ...,  $f_m(\bar{x}, z_0, z_1, \ldots, z_{m-1})$  such that 
$
\bigvee_{i=0}^m B(\bar{x}, f_i(\bar{x}, z_0, \ldots, z_{i-1}), z_{i})
$ is provable in $\PV_k$.
\end{itemize}
\end{theorem}
It is possible to generalize this theorem to a \emph{generalized Herbrand's theorem} to cover more alternations of quantifiers. However, in this paper, one can restrict oneself only to these two levels \cite{BussProofTheory}.

The system $\PV_k$ proves the scheme $\PInd$ for any quantifier-free $\mathcal{L}_{\PV_k}$-formula. As any $\mathcal{L}_{\PV_k}$-term can be defined by an $\mathcal{L}_{\PV}$-formula in $\hat{\Sigma}^b_k$, it is possible to represent any quantifier-free $\mathcal{L}_{\PV_k}$-formula by two $\mathcal{L}_{\PV}$-formulas, one in $\hat{\Sigma}^b_k$ and one in $\hat{\Pi}^b_k$. Using this fact, one can interpret $\PV_k$ inside the theory $S^k_2$.

Going beyond bounded theories of arithmetic, in a similar fashion to $\PV$ and using the construction of primitive recursive functions by composition and primitive recursion on certain basic functions, it is possible to extend the language $\mathcal{L}_{\PV}$ by a \emph{fresh} function symbol for any primitive recursive function. Denote this new language by $\mathcal{L}_{\PRA}$ and set the first-order theory $\PRA$ over $\mathcal{L}_{\PRA}$ as $\PV$ extended by the defining axioms for the new functional symbols and the induction axiom $A(0) \wedge \forall x (A(x) \to A(x+1)) \to \forall x A(x)$, for any quantifier-free formula in the new language. This is of course different from the usual definition of $\PRA$ as its language is extended by the ptime function symbols in $\mathcal{L}_{\PV}$ and the theory itself is extended by the theory $\PV$. Moreover, the formula in the induction axiom of $\PRA$ may contain the symbols from $\mathcal{L}_{\PV}$. However, as the functions in the Cobham calculus are constructible as primitive recursive functions, it is clear that the separation of the primitive recursive function symbols and ptime function symbols is just a technical point and is totally immaterial. In fact, our presentation of $\PRA$ is a conservative extension of the usual $\PRA$ and hence has nothing essentially different from the usual $\PRA$. \\ 
By Peano arithmetic, denoted by $\PA$, we mean the theory $\PV$ extended by full induction axiom scheme $A(0) \wedge \forall x (A(x) \to A(x+1)) \to \forall x A(x)$, for \emph{any} formula $A(x)$. This is also different from the usual definition of $\PA$. However, as all of the function symbols in $\mathcal{L}_{\PV}$ are definable in the usual language of $\PA$ and their functionality and totality are provable in the usual $\PA$, it is easy to see that our $\PA$ is a conservative extension of the usual $\PA$.\\
By $\Pi^0_2$, we mean the class of $\mathcal{L}_{\PV}$-formulas in the form $\forall \bar{x} \exists \bar{y}A(\bar{x}, \bar{y})$, where any quantifier in $A(\bar{x}, \bar{y})$ is bounded. For two theories $T$ and $S$ and a class of formulas $\Phi$, by $T \equiv_{\Phi} S$, we mean $T \vdash A$ iff $S \vdash A$, for any $A \in \Phi$.\\
Finally, let us recall some basics of the ordinal arithmetic. Apart from addition, multiplication and exponentiation of the ordinals, it is also possible to define
subtraction $\dotminus$ from \emph{left} such that $\alpha \dotminus \beta=0$, if $\alpha \prec \beta$ and $\alpha \dotminus \beta=\gamma$, if $\beta \preceq \alpha$, where $\gamma$ is the \emph{unique} ordinal with the property that $\beta + \gamma=\alpha$. Similarly, it is possible to define the division $d$ from \emph{left} such that if $\beta \neq 0$, then $d(\alpha, \beta)$ is the \emph{unique} $\gamma$ such that $\alpha=\beta \gamma + \delta$, for some $\delta \prec \beta$.
%Moreover, recall that $\epsilon_0$ is the least ordinal closed under the ordinal map $\alpha \mapsto \omega^{\alpha}$ and by Cantor normal form theorem, for any $\alpha \prec \epsilon_0$, there are unique $\alpha_n \prec \cdots \prec \alpha_2 \prec \alpha_1 \prec \epsilon_0$ and $a_1, \cdots, a_n \in \mathbb{N}-\{0\}$ such that $\alpha=\omega^{\alpha_1}a_1+ \cdots+ \omega^{\alpha_n}a_n$. Using these normal forms, it is possible to arithmetize the ordinals below $\epsilon_0$. Usually, this is done by primitive recursive functions formalized in $\PRA$. However, in this paper, we need a more careful arithmetization by polynomial time computable functions formalizable in $\PV$. 

\section{Polynomial-time Ordinal Representations}\label{PTimeRepresentation}
In this section, we will introduce polynomial time ordinal representations and recall the concrete representation for the ordinal $\epsilon_0$ provided in \cite{BBP}. Both parts will be of essential use in Section \ref{OrdinalFlowsandArithmetic}.
\begin{definition}\label{t3-4}
Let $\alpha$ be an infinite ordinal closed under addition, multiplication and the operation $\beta \mapsto \omega^{\beta}$. We call the tuple 
\[
\mathfrak{O}=(\mathcal{O}, \prec, +, \cdot, \dotminus, d(\cdot, \cdot), o, x \mapsto  \omega^{x}, \mathbf{0}, \mathbf{1}, \omega)
\] 
a \emph{polynomial time representation with a primitive recursive exponentiation} (ptime representation, for short) for the ordinal $\alpha$, if:
\begin{itemize}
\item[$\bullet$]
$\mathcal{O}$ is a unary polynomial time relation on the natural numbers represented as a quantifier-free $\mathcal{L}_{\PV}$-formula. Its intended meaning is the set of all the representations of the ordinals below $\alpha$. We use small Greek letters to denote the elements of $\mathcal{O}$. For instance, by $\forall \beta \, A(\beta)$, we actually mean $\forall x (\mathcal{O}(x) \to A(x))$.
\item[$\bullet$]
$\prec$ is a binary polynomial time relation on the natural numbers, represented as a quantifier-free $\mathcal{L}_{\PV}$-formula.
Its intended meaning is the order over the ordinals below $\alpha$. We define the relation $(\gamma \preceq \beta)$ as $(\gamma \prec \beta) \vee (\gamma=\beta)$.
\item[$\bullet$]
$+, \cdot, \dotminus$ and $ d(\cdot, \cdot)$ are binary polynomial time functions, represented as $\mathcal{L}_{\PV}$-terms. Their intended meaning is the ordinal addition, multiplication, subtraction from left and division from left, respectively.
\item[$\bullet$]
$o$ is a unary polynomial time function represented as an $\mathcal{L}_{\PV}$-term. Its intended meaning is the function that maps the natural numbers to the representation of their order-types below $\alpha$. For instance, $o(0)$ is the least element of $\mathcal{O}$ while $o(1)$ is its
second least element. 
\item[$\bullet$]
$\omega^{x}$ is a \emph{primitive recursive} unary function represented as an $\mathcal{L}_{\PRA}$-term.
Its intended meaning is the function that maps the ordinal $\beta \prec \alpha$ to the ordinal $\omega^{\beta} \prec \alpha$.
\item[$\bullet$]
$\mathbf{0}$, $\mathbf{1}$ and $\omega$ are three numbers representing the ordinals zero, one and $\omega$, respectively.
\item[$\bullet$]
The structure $(\mathfrak{O}, \prec)$ is isomorphic to $(\alpha, \prec_{\alpha})$, where $\prec_{\alpha}$ is the order on $\alpha$.
\item[$\bullet$]
$\PV$ proves that $\prec$ is a total ordering on $\mathcal{O}$ with the minimum $\mathbf{0}$.
\item[$\bullet$]
$\PV$ proves that $\prec$ is discrete over $\mathcal{O}$, i.e.,
for all $\beta, \gamma \in \mathcal{O}$, if $\gamma \prec \beta+\mathbf{1}$ then either $\gamma \prec \beta$ or $ \gamma=\beta$.
\item[$\bullet$]
$\PV$ proves the associativity of the addition and multiplication, the left distributivity of multiplication over the addition, the neutrality of $\mathbf{0}$ for the addition, the neutrality of $\mathbf{1}$ for the multiplication and the identity $\mathbf{0}\beta=\beta \mathbf{0}=\mathbf{0}$.
\item
$\PV$ proves that the addition and the non-zero multiplication from left respect the order $\prec$, i.e., if $\delta \prec \gamma$ then $\beta + \delta \prec \beta + \gamma $ and if we also have $\beta \neq \mathbf{0}$, then $\beta \delta \prec \beta \gamma$.
\item
$\PV$ proves that the addition and multiplication from right respects $\preceq$, i.e., if $\delta \preceq \gamma$ then $\delta+\beta \preceq \gamma + \beta$ and $\delta \beta \preceq \gamma \beta$.
\item
$\PV$ proves the defining axioms of $\dotminus$, i.e.,
if $\alpha \prec \beta$ then $\alpha \dotminus \beta=\mathbf{0} $ and
if $\alpha \succeq \beta$ then $\alpha=\beta+(\alpha \dotminus \beta)$.
\item
$\PV$ proves the defining axioms of $d$, i.e.,
if $\beta \neq \mathbf{0}$, then
$\beta d(\alpha, \beta) \preceq \alpha$ and
$\alpha\dotminus \beta d(\alpha, \beta) \prec \beta$.
\item[$\bullet$]
$\PV$ proves that $o$ is an order-isomorphism between the natural numbers and the ordinals below $\omega$, mapping $0$ and $1$ to $\mathbf{0}$ and $\mathbf{1}$, respectively, i.e., $\PV$ proves $o(0)=\mathbf{0}$, $o(1)=\mathbf{1}$, $\forall x [\mathcal{O}(o(x)) \wedge o(x) \prec \omega]$, $\forall \beta \prec \omega \exists ! y \, o(y)=\beta$ and $\forall xy (x < y \leftrightarrow o(x) \prec o(y))$. Where there is no risk of confusion, we will use the numbers and their ordinal reinterpretations, interchangeably. For instance, we use $1$ for $\mathbf{1}$.
\item
$\PRA$ proves that $\omega^{\mathbf{0}}=1$ and $\omega^{\mathbf{1}}=\omega$. It also proves that $\omega^{\beta}$ respects $\preceq$ and maps the addition to the multiplication.
\item
If there is no $\gamma \in \mathcal{O}$ such that $\beta=\gamma+1$, then $\omega^{\beta}$ is the supremum of the set $\{\omega^{\gamma} \mid \gamma \prec \beta \}$, i.e., for any $\delta \in \mathcal{O}$, if $\omega^{\gamma} \preceq \delta$, for any $\gamma \prec \beta$, then $\omega^{\beta} \preceq \delta$.
\item
$\PRA$ proves that for every $\beta \in \mathcal{O}$, there is a unique expansion $\beta=\omega^{\gamma_1}+ \ldots + \omega^{\gamma_n}$ such that $\gamma_n \preceq \gamma_{n-1} \preceq \ldots \preceq \gamma_1$.
\end{itemize}
\end{definition}

\begin{remark}
Here are some remarks. First, notice that the relations of being a successor and a limit ordinal are both definable by the predicates $\exists \gamma (\beta=\gamma+1)$ and $\forall \gamma \prec \beta (\gamma+1 \prec \beta)$, respectively. It is also easy to see that $\PV$ can prove the dichotomy that for any $\beta \in \mathcal{O}$, it is either a successor or a limit. Secondly, using the compatibility of the order with the addition and the multiplication, one can easily prove in $\PV$ that if $\beta=\gamma+\delta=\gamma+\eta$, then $\beta \dotminus \gamma=\delta=\eta$. This observation proves that for any $\gamma \prec \beta$, the interval $(\mathbf{0}, \beta \dotminus \gamma)$ in $\mathcal{O}$ is in one-to-one correspondence with the interval $(\gamma, \beta)$, via the map $\delta \mapsto \gamma+\delta$. Similarly, $\PV$ proves that if $\gamma \neq \mathbf{0}$, then $\beta=\gamma \delta=\gamma \eta$ implies $d(\beta, \gamma)=\delta=\eta$. Therefore, $d(\gamma \delta, \gamma)=\delta$, for $\gamma \neq \mathbf{0}$. Thirdly, let us explain the discrepancy between the polynomial time character of the order, addition, multiplication, subtraction and division and the primitive recursive character of the function $x \mapsto \omega^{x}$ in our definition. For that purpose, first, pretend that our definition uses the primitive recursive functions and predicates and $\PRA$ everywhere when it actually uses polynomial time functions and predicates and $\PV$. Then, one can easily see that this primitive recursive version of our representation is just a mild extension of the primitive recursive (even elementary) ordinal representation employed in \cite{Fri}. (Their conditions are different, but it is easy to show that our axioms imply theirs). As we use a proof-theoretic result of \cite{Fri}, using the primitive recursive version of our definition is completely justified. However, there is another role for our ordinal representation. As it is clear, in this paper, we intend to address the lower complexity formulas and for that purpose, some basic ordinal arithmetic (up to addition and multiplication and hence subtraction and division from left) is required to be implemented in polynomial time. Therefore, we are forced to lower the complexity of some parts of the representation. However, as the use of the exponentiation is only restricted to the result from \cite{Fri} that we use as a black box here, we decided to lower the complexity up to the point we need and let the exponentiation parts intact. This way we can accept more ptime representations.
\end{remark}

Let $\beta \in \mathcal{O}$. By the axiom scheme $\TI(\prec_{\beta})$, we mean the transfinite induction up to the ordinal $\beta$, i.e., 
\[
\forall \gamma \prec \beta [\forall \delta \prec \gamma A(\delta) \to A(\gamma)] \to \forall \gamma \prec \beta A(\gamma),
\]
where $A$ can be any formula in $\mathcal{L}_{\PV}$. 
In \cite{Fri}, a refined method of ordinal analysis is provided showing that the $\Pi^0_2$-consequences of the theory $\PA+\bigcup_{\beta \in \mathcal{O}} \TI(\prec_{\beta})$ are actually provable in a smaller theory extending $\PRA$ with a weak form of transfinite induction stating that for any $\beta \prec \alpha$, there is no primitive recursive decreasing sequence of ordinals below $\beta$. For more, see \cite{Fri,Realm}.

\begin{theorem}\label{ConCutElim}
Let $\alpha$ be an ordinal and $\mathfrak{O}$ be its ptime representation. Then,
$
\PA+\bigcup_{\beta \in \mathcal{O}}\TI(\prec_{\beta}) \equiv_{\Pi^0_2} \PRA+ \bigcup_{\beta \in \mathcal{O}} \PRWO(\prec_{\beta})$,
where $\PRWO(\prec_{\beta})$ is the scheme
$
\forall \bar{x} \exists y [f(\bar{x}, y+1) \nprec f(\bar{x}, y) \vee \neg \mathcal{O}(f(\bar{x}, y)) \vee f(\bar{x}, y) \nprec \beta ]$, 
for any function symbol $f$ in $\mathcal{L}_{\PRA}$.
\end{theorem}

\subsection{A Polynomial-time Representation for $\epsilon_0$}\label{Notation}

%\begin{definition}
%A primitive recursive representation of the ordinal $\alpha$ is a structure $\mathcal{A}=(A, \prec_A, +_A, \cdot_A, x \mapsto \omega^{x})$ such that:
%\begin{description}
%\item[$\bullet$] 
%$A$ is an infinite primitive recursive subset of $\mathbb{N}$. 
%\item[$\bullet$] 
%$\prec_A$ is a primitive recursive binary relation on $A$. 
%\item[$\bullet$] 
%$+_A$, $\cdot_A$ are binary, and $x \mapsto \omega^x$ is unary, primitive recursive functions on $A$.
%\item[$\bullet$]
%$\PRA$ proves that $\mathcal{A}$ satisfies \emph{all the usual order and algebraic properties} of an initial segment of ordinals that are defined in detail in \cite{Fri}.
%\item[$\bullet$]
%The structure $\mathcal{A}$ is isomorphic to the structure $(\alpha, \prec_{\alpha}, +_{\alpha}, \cdot_{\alpha}, \beta \mapsto \omega^{\beta})$, where the order and the functions in the latter structure are the usual ordinal theoretic order and ordinal operations. Note that this condition implies that the ordinal $\alpha$ is closed under the operation $\beta \mapsto \omega^{\beta}$.
%\end{description}
%\end{definition}
In this subsection, we will recall the basics of the ptime notation system for the ordinal $\epsilon_0$, introduced in \cite{BBP}. Define $\mathcal{O}_{0}$ and $\prec_{0}$ inductively and simultaneously in the following way: $\mathcal{O}_{0}$ is the least set of expressions containing the empty string $\mathbf{0}$ and is closed under the operation $(\alpha_1, \ldots, \alpha_n) \mapsto \omega^{\alpha_1}a_1+ \ldots+\omega^{\alpha_n}a_n$, where $a_i \neq 0$ are natural numbers and $\alpha_n \prec_{0} \ldots \prec_{0} \alpha_2 \prec_{0} \alpha_1$ and set $\omega^{\alpha_1}a_1+ \ldots+\omega^{\alpha_n}a_n \prec_{0} \omega^{\beta_1}b_1+ \ldots+\omega^{\beta_m}b_m$, if there exists $i \leq min\{m,n\}$ such that $\alpha_j=\beta_j$ and $a_j=b_j$, for any $j \leq i$ and one of the following takes place:
\begin{itemize}
    \item[$\bullet$]
    $i=n<m$,
    \item[$\bullet$]
    $i < min\{m,n\}$ and $\alpha_{i+1} \prec_{0} \beta_{i+1}$
    \item[$\bullet$]
    $i < min\{m,n\}$ and $\alpha_{i+1} = \beta_{i+1}$ and $a_i < b_i$.
\end{itemize}
Using some efficient method of sequence
encoding, it is possible to arithmetize the set $\mathcal{O}_{0}$ and the predicate $\prec_{0}$. It is also possible to implement the arithmetization in a way that the length of the G\"{o}del number of $\alpha \in \mathcal{O}_{0}$ is proportional to the number of symbols in the expression $\alpha$. By this fact, \cite{BBP} shows that both $\mathcal{O}_{0}$ and $\prec_{0}$ are polynomial time computable and hence formalizable in $\PV$. (Technically, it uses a conservative extension of $\PV$, but the difference does not affect us here). We fix \emph{quantifier-free} predicates $\mathcal{O}_{0}(x)$ and $x \prec_{0} y$ to denote the formalized versions in the language $\mathcal{L}_{\PV}$.  In \cite{BBP}, it is shown that $\PV$ proves that $\prec$ is a total ordering on $\mathcal{O}_{0}$. It is clear that $\PV$ also proves that $\mathbf{0}$ is the minimum element of $\mathcal{O}_{0}$. %It is also easy to see that the successor function over $\mathcal{O}$ mapping $\alpha$ to its successor, denoted by $\alpha+1$, is a ptime function and $\PV \vdash \forall \beta [\beta \prec \alpha+1 \leftrightarrow \beta \prec \alpha \vee \beta=\alpha]$. We can also easily observe that there are ptime predicates defining whether an ordinal is $0$, or a successor or a limit. This trichotomy is also provable in $\PV$. 
Define $\mathbf{1}$ as $\omega^{\mathbf{0}}1$ and for $o$, consider the function that maps the number $n$ to $\omega^{\mathbf{0}}n$. Denote $\omega^{o(1)}$ by $\omega$. Then, we have $\omega^{\mathbf{0}}=\mathbf{1}$ and $\omega^{\mathbf{1}}=\omega$. The map $o$ is ptime and it is easy to prove in $\PV$ that $o$ is an order-isomorphism, i.e., $\PV \vdash \forall x [\mathcal{O}_{0}(o(x)) \wedge o(x) \prec \omega]$, $\PV \vdash \forall \alpha \prec \omega \exists ! y \, o(y)=\alpha$ and $\PV \vdash x < y \leftrightarrow o(x) \prec_{0} o(y)$. For $x \mapsto \omega^x$, use the evident function mapping the expression $\beta$ to the expression $\omega^{\beta}$ and note that it is clearly primitive recursive.\\
In the rest of this subsection, we will explain how to formalize the basic ordinal arithmetic in $\PV$, using the aforementioned representation. For that purpose, first consider the following equalities over the real ordinals below $\epsilon_0$. We assumed that the inputs are \emph{non-zero} as the operations with one zero input are trivial. These equalities make the computation of the addition, multiplication, subtraction from left and division from left possible, using the Cantor normal form of the ordinals. We will not provide a proof for these equalities as they are just simple computations, see \cite{takeuti2012introduction}.
\[
{\footnotesize
(\sum_{i=1}^n \omega^{\alpha_i}a_i)+(\sum_{j=1}^m \omega^{\beta_j}b_j)=
\begin{cases}
\sum_{i=1}^n \omega^{\alpha_i}a_i+\sum_{j=1}^m \omega^{\beta_j}b_j & \alpha_n \succ \beta_1 \\
\sum_{j=1}^m \omega^{\beta_j}b_j & \alpha_1 \prec \beta_1\\
\sum_{i=1}^{k} \omega^{\alpha_i}a_i+ \sum_{j=1}^m \omega^{\beta_j}b_j & \alpha_{k+1} \prec \beta_1 \prec \alpha_{k}\\
\sum_{i=1}^{k-1} \omega^{\alpha_i}a_i+\omega^{\alpha_k}(a_k+b_1) +\sum_{j=2}^m \omega^{\beta_j}b_j & \alpha_{k}=\beta_1
\end{cases} 
}
\]
\[
{\footnotesize
(\sum_{i=1}^n \omega^{\alpha_i}a_i)\dotminus(\sum_{j=1}^m \omega^{\beta_j}b_j)=
\begin{cases}
0 & \alpha_{k+1} \prec \beta_{k+1} \\
\sum_{i=k+1}^n \omega^{\alpha_i}a_i & \alpha_{k+1} \succ \beta_{k+1}\\
\omega^{\alpha_{k+1}}(a_{k+1}-b_{k+1})+\sum_{i=k+2}^n \omega^{\alpha_i}a_i & \alpha_{k+1} = \beta_{k+1}, a_{k+1}>b_{k+1}\\
0 & \alpha_{k+1} = \beta_{k+1}, a_{k+1}<b_{k+1}
\end{cases} 
}
\]
where $k$ is the maximum $i$ such that $\alpha_i=\beta_i$ and $a_i=b_i$, if there is any and otherwise $k=0$,
\[
{\small
(\sum_{i=1}^n \omega^{\alpha_i}a_i)(\sum_{j=1}^m \omega^{\beta_j}b_j)=
\begin{cases}
\sum_{j=1}^m \omega^{\alpha_1+\beta_j}b_j & \beta_m \succ 0 \\
\sum_{j=1}^{m-1} \omega^{\alpha_1+\beta_j}b_j+\omega^{\alpha_1}a_1b_m+\sum_{i=2}^n \omega^{\alpha_i}a_i & \beta_m = 0, m >1\\
\omega^{\alpha_1}a_1b_1+\sum_{i=2}^n \omega^{\alpha_i}a_i & \beta_m = 0, m=1
\end{cases} 
}
\]
\[
{\small
d(\sum_{i=1}^n \omega^{\alpha_i}a_i,\sum_{j=1}^m \omega^{\beta_j}b_j)=
\begin{cases}
0 & \alpha_1 \prec \beta_1  \\
\sum_{i=1}^k \omega^{\alpha_i \dotminus \beta_1}a_i  & \alpha_1 \succeq \beta_1, \alpha_{k} \neq \beta_1\\
\sum_{i=1}^{k-1} \omega^{\alpha_i \dotminus \beta_1}a_i+d(a_k,b_1)  & \alpha_1 \succeq \beta_1, \alpha_{k} = \beta_1, (*)\\
\sum_{i=1}^{k-1} \omega^{\alpha_i \dotminus \beta_1}a_i+(d(a_k,b_1)-1)  & \text{otherwise}
\end{cases} 
}
\]
where $k$ is the greatest $i$ such that $\alpha_i \succeq \beta_1$, $d(a_k,b_1)$ is the quotient of $a_k$ divided by $b_1$ and $(*)$ is the condition that $\sum_{i=k}^n \omega^{\alpha_i}a_i \succeq \omega^{\alpha_k}b_1d(a_k,b_1)+\sum_{j=2}^m \omega^{\beta_j}b_j$.
%\[
%d(\sum_{i=1}^n \omega^{\alpha_i}a_i,\sum_{j=1}^m \omega^{\beta_j}b_j)=
%\begin{cases}
%\sum_{\alpha_i \succeq \beta_1} \omega^{\alpha_i \dotminus \beta_1}a_i & \beta_1 \succ 0 \\
%\sum_{i=1}^n \omega^{\alpha_i}a_i  & \beta_1 = 0, \alpha_n \neq 0 \\
%d(a_n, b_1)    & \beta_1 = 0, \alpha_n = 0
%\end{cases} 
%\]
Note that to compute any of the operations, it is enough to do constant many comparisons and basic numerical computations, a search to find the maximum index that takes at most as long as the length of the inputs and at most $m$ or $n$ many applications of a ptime function. Hence, all the operations are ptime and hence representable in $\PV$. It is easy to see but tedious to show that all the claimed properties in Definition \ref{t3-4} hold. Therefore, the described data in this subsection defines a ptime representation for $\epsilon_0$ that we denote by $\mathfrak{O}_{0}$.
%It is also easy to see that $\PV$ proves the associativity of the addition and multiplication, left distributivity of multiplication over addition, the neutrality of $0$ for the addition, the neutrality of $1$ for multiplication, $0\alpha=\alpha 0=0$ and the defining axioms of $\dotminus$ and $d$, i.e., $\alpha \dotminus \beta=0 $, if $\alpha \prec \beta$, $\alpha=\beta+(\alpha \dotminus \beta)$, if $\alpha \succeq \beta$,  $\beta d(\alpha, \beta) \preceq \alpha$ and $\alpha\dotminus \beta d(\alpha, \beta) \prec \beta$. The theory $\PV$ also proves that $\prec$ is preserved under the left addition and the left multiplication by a non-zero ordinal. Using this property, one can show that if $\alpha=\beta+\gamma=\beta+\delta$, then $\alpha \dotminus \beta=\gamma=\delta$. This proves that the interval $[0, \alpha \dotminus \beta]$ in $\mathcal{O}$ is in one-to-one correspondence with $[\beta, \alpha]$ via the map $\gamma \mapsto \beta+\gamma$. Similarly, if $\beta \neq 0$, $\alpha=\beta\gamma=\beta\delta$ implies $d(\alpha, \beta)=\gamma=\delta$.

\section{Ordinal Flows and Arithmetic} \label{OrdinalFlowsandArithmetic}
Let $\alpha$ be an ordinal and  $\mathfrak{O}$ be its ptime representation. In this section, we develop a witnessing method for the theory $\PA+\bigcup_{\beta \in \mathcal{{O}}}\TI(\prec_{\beta})$. The section consists of three parts. First, in Subsection \ref{SectionOfTI}, we will introduce an auxiliary theory $\TI(\forall_1, \prec)$ with a transfinite induction on the universal formulas in the language of $\PV$. The system is powerful enough to interpret $\PRA+\bigcup_{\beta \in \mathcal{O}} \PRWO(\prec_{\beta})$ and hence proves all $\Pi^0_2$-theorems of $\PA+\bigcup_{\beta \in \mathcal{{O}}}\TI(\prec_{\beta})$. Then in Subsection \ref{SectionOfFlows}, we will provide a witnessing method for $\TI(\forall_1, \prec)$ that transforms the provability between two universal formulas in $\TI(\forall_1, \prec)$ to an ordinal-length sequence of $\PV$-provable implications. Finally, in Subsection \ref{OrdinalLocalSearch}, we use Herbrand's theorem, Theorem \ref{Herbrand}, to witness the implications in $\PV$ to provide a characterization for the low complexity theorems of $\PA+\bigcup_{\beta \in \mathcal{{O}}}\TI(\prec_{\beta})$.
\subsection{The System  $\TI(\forall_1, \prec)$} \label{SectionOfTI}

This subsection is devoted to the introduction and investigation of the auxiliary theory $\TI(\forall_1, \prec)$.

\begin{definition}
Define $\forall_1$ (resp., $\exists_1$) as the least set of $\mathcal{L}_{\PV}$-formulas containing all atomic formulas and their negations and closed under conjunction, disjunction and universal (resp. existential) quantifiers.
\end{definition}

Let $I\forall_1$ (resp. $I\exists_1$) be the theory extending $\PV$ by the $\forall_1$-induction (resp. $\exists_1$-induction) scheme, i.e.,
$A(0) \wedge \forall x (A(x) \to A(x+1)) \to \forall x A(x)$, for any $A(x) \in \forall_1$ (resp. $A(x) \in \exists_1$). Note that $I\exists_1=I\forall_1$. The proof uses the usual technique of using $\forall_1$-induction on $B(x)=\neg A(y \dotminus x)$ to prove $\exists_1$-induction on $A(y)$ and similarly for the other direction, see \cite{BussProofTheoryArithmetic}.

\begin{lemma}\label{PRAEmbedding}
For any primitive recursive function $f: \mathbb{N}^k \to \mathbb{N}$, there is a $\exists_1$-formula $D_f(\bar{x}, y)$ such that $I\exists_1 \vdash \forall \bar{x} \exists ! y D_f(\bar{x}, y)$ and $\mathbb{N} \vDash D_f(\bar{n},m)$ iff $f(\bar{n})=m$, for any $\bar{n}, m \in \mathbb{N}$.
\end{lemma}
\begin{proof}
For any primitive recursive function $f$, we provide a quantifier-free formula $C_f(\bar{x}, w, y) \in \mathcal{L}_{\PV}$ encoding that $w$ is a computation of $f$ with the input $\bar{x}$ and the output $y$. To that aim, we use recursion on the construction of $f$. The cases for the basic functions and composition are easy. For the recursion case, if $f(\bar{x}, y)$ is defined via recursive equations $f(\bar{x}, 0)=g(\bar{x})$ and $f(\bar{x}, y+1)=h(\bar{x}, y, f(\bar{x}, y))$, define $C_f(\bar{x}, y, \langle u, v \rangle, z)$ as $C_g(\bar{x}, u_0, v_0) \wedge \forall i \leq l(v) C_h(\bar{x}, i, v_i, u_{i+1}, v_{i+1}) \wedge v_{l(v)}=z$, where $v$ encodes the sequence $\{f(\bar{x}, i)\}_{i=0}^{l(v)}$, the number $l(v)$ is the length of this sequence and $u$ encodes the sequence of computations $\{u_{i}\}_{i=0}^{l(v)}$, where $u_0$ reads $\bar{x}$ and computes $v_0=f(\bar{x}, 0)$ and $u_{i+1}$ reads $\bar{x}$, $i$ and $f(\bar{x}, i)$ and computes $f(\bar{x}, i+1)$ via the function $h$. Note that the predicate $\forall i \leq l(v) C_h(\bar{x}, i, v_i, u_i, v_{i+1})$ is polynomial computable, as $l(v) \leq |v|$, where $|v|$ is the binary length of $v$. Hence there exists a polynomial time function symbol in $\PV$ like $F$ such that $\PV$ proves that $F(\bar{x}, u, v)=1$ iff $\forall i \leq l(v) C_h(\bar{x}, i, v_i, u_i, v_{i+1})$. Therefore, $C_f$ can be written in a quantifier-free form. Now, set $D_f(\bar{x}, y)=\exists w C_f(\bar{x}, w, y)$. It is clear that $D_f \in \exists_1$ and $\mathbb{N} \vDash D_f(\bar{n},m)$ iff $f(\bar{n})=m$, for any $\bar{n}, m \in \mathbb{N}$. Finally, the proof of the claim that $I\exists_1 \vdash \forall \bar{x} \exists ! y D_f(\bar{x},y)$ is similar to the similar claim in the representation of primitive recursive functions in $I\Sigma_1$.
\end{proof}

\begin{definition}\label{t4-1}
Define the theory $\TI(\forall_1, \prec)$ over $\mathcal{L}_{\PV}$ as the theory $\PV$ extended by the transfinite induction scheme $\forall \delta (\forall \gamma \prec \delta \; A(\gamma) \to A(\delta)) \to A(\theta)$,
for any $A(\gamma) \in \forall_1$ and any constant $\theta \in \mathcal{O}$.
\end{definition}

Note that $\TI(\forall_1, \prec)$ extends the theory $I\forall_1$ as $\TI(\forall_1, \prec)$ proves $\forall \delta \prec \omega (\forall \gamma \prec \delta \; A(\gamma) \to A(\delta)) \to \forall \delta \prec \omega A(\delta)$, for any $A \in \forall_1$. Using the function $o$ and the fact that it is an order-isomorphism between the numbers and the ordinals below $\omega$, we will have $\forall x (\forall y < x \; A(y) \to A(x)) \to \forall x A(x)$ which implies $A(0) \wedge \forall x (A(x) \to A(x+1)) \to \forall x A(x)$. Therefore, by Lemma \ref{PRAEmbedding}, $\TI(\forall_1, \prec)$ represents any primitive recursive function with an $\exists_1$-definition. As it is routine in arithmetic \cite{BussProofTheoryArithmetic}, this provides both $\forall_1$ and $\exists_1$ definitions for any atomic formula in $\mathcal{L}_{\PRA}$. Hence, it is possible to interpret any $\forall_1$-formula in $\mathcal{L}_{\PRA}$ as an $\forall_1$-formula in $\mathcal{L}_{\PV}$. Using that interpretation, we can pretend that $\TI(\forall_1, \prec)$ has a fresh function symbol for any primitive recursive function and the $\forall_1$-formulas in the new language are allowed in the transfinite induction. Moreover, we can also pretend that $\TI(\forall_1, \prec)$ extends the theory $\PRA$. The reason simply is that the equational defining axioms in $\PRA$ are all provable in $I\forall_1=I\exists_1$ and hence in $\TI(\forall_1, \prec)$, as they are actually encoded in the definition $D_f$ of $f$. For the quantifier-free induction of $\PRA$, as we have seen before, it is possible to use the isomorphism $o$ to prove the induction in $\TI(\forall_1, \prec)$.

\begin{lemma}\label{PRWOtoTI}
If $\PRA+\bigcup_{\beta \in \mathcal{O}} \PRWO(\prec_{\beta}) \vdash A$ then $\TI(\forall_1, \prec) \vdash A$, for any $A \in \mathcal{L}_{\PV}$.
\end{lemma}
\begin{proof}
Pretend $\TI(\forall_1, \prec)$ has a function symbol for any primitive recursive function, allowed in the $\forall_1$-formulas. As $\TI(\forall_1, \prec)$ extends $\PRA$, it is enough to prove $\TI(\forall_1, \prec) \vdash \PRWO(\prec_{\beta})$, for any $\beta \in \mathcal{O}$.
For the sake of contradiction, assume 
$\forall y [f(\bar{x}, y+1) \prec f(\bar{x}, y) \wedge \mathcal{O}(f(\bar{x}, y)) \wedge f(\bar{x}, y) \prec \beta ]$.
Set $B(\gamma, \bar{x})= \forall y (f(\bar{x}, y) \neq \gamma)$ and note that $B(\gamma, \bar{x}) \in \forall_1$. By transfinite induction, we prove $\forall \gamma \prec \beta \,B(\gamma, \bar{x})$. 
%For $\gamma=0$, the claim is clear, because if $f(\bar{x}, y)=0$ then $f(\bar{x}, y+1) \prec 0$ which is impossible. Now if 
For that purpose, assume $\forall \delta \prec \gamma [\delta \prec \beta \to B(\delta, \bar{x})]$. Then, to prove $[\gamma \prec \beta \to B(\gamma, \bar{x})]$, if $f(\bar{x}, y) = \gamma$, for some $\gamma \prec \beta$, as $f(\bar{x}, y+1) \prec f(\bar{x}, y)$, we have $f(\bar{x}, y+1) \prec \gamma \prec \beta$. On the other hand, by $\forall \delta \prec \gamma [\delta \prec \beta \to B(\delta, \bar{x})]$, we know that none of the ordinals $\delta$ below $\gamma$ is in the form of $f(\bar{x}, z)$, which contradicts with $f(\bar{x}, y+1) \prec \gamma$. Hence, $[\gamma \prec \beta \to B(\gamma, \bar{x})]$. Therefore, $\forall \delta \prec \gamma [\delta \prec \beta \to B(\delta, \bar{x})]$ implies $[\gamma \prec \beta \to B(\gamma, \bar{x})]$. Hence, by transfinite induction, we have $\forall \gamma \prec \beta \, B(\gamma, \bar{x})$ which for $\gamma=f(\bar{x}, 0) \prec \beta$ implies $\forall y (f(\bar{x}, y) \neq f(\bar{x}, 0))$ which is a contradiction.
\end{proof}

\begin{corollary}\label{PAToTI}
$\PA+\bigcup_{\beta \in \mathcal{O}}\TI(\prec_{\beta}) \equiv_{\Pi^0_2} \TI(\forall_1, \prec)$.
\end{corollary}
\begin{proof}
One direction is a consequence of the fact that $\PA+\bigcup_{\beta \in \mathcal{O}}\TI(\prec_{\beta})$ proves the transfinite induction for any formulas and hence extends the theory $\TI(\forall_1, \prec)$. The other direction is a consequence of Theorem \ref{ConCutElim} and Lemma \ref{PRWOtoTI}.
\end{proof}

%Here are some remarks. First, note that the theory $\TI(\forall_1, \prec)$ proves $\forall \delta \prec \theta (\forall \gamma \prec \delta \; A(\gamma) \rightarrow A(\delta)) \to \forall \delta \prec \theta A(\delta)$, where $\theta \in O$ is a constant. It is enough to use the induction scheme on the formula $B(\delta)=\delta \preceq \theta \to \forall \eta \prec \delta  A(\eta)$. \\
\subsubsection{A proof system for $\TI(\forall_1, \prec)$} \label{ProofSystemForTI}
We now present a sequent calculus for the theory $\TI(\forall_1, \prec)$. By a sequent over $\mathcal{L}_{\PV}$, we mean an expression in the form $S=\Gamma \Rightarrow \Delta$, where $\Gamma$ and $\Delta$ are multisets of formulas in $\mathcal{L}_{\PV}$.
Define $\mathbf{LPV}$ as the usual system $\mathbf{LK}$ augmented with the equality axioms for atomic formulas and their negations and all quantifier-free \emph{theorems} of $\PV$ as the initial sequents:\\

\noindent \textbf{Axioms:}

\begin{center}
 \begin{tabular}{c c}
 \AxiomC{}
 \UnaryInfC{$\bot \Rightarrow $}
 \DisplayProof \hspace{15pt}
 &
 \AxiomC{}
 \UnaryInfC{$ \Rightarrow \top$}
 \DisplayProof 
 \\[3ex]
 \end{tabular}

 \begin{tabular}{c c c c c}
 \AxiomC{}
 \UnaryInfC{$P \Rightarrow P$}
 \DisplayProof \hspace{5pt}
 &
 \AxiomC{}
 \UnaryInfC{$\neg P \Rightarrow \neg P$}
 \DisplayProof \hspace{5pt}
 &
 \AxiomC{}
 \UnaryInfC{$P, \neg P \Rightarrow $}
 \DisplayProof \hspace{5pt}
 &
 \AxiomC{}
 \UnaryInfC{$ \Rightarrow P, \neg P$}
 \DisplayProof \hspace{5pt}
 &
 \AxiomC{}
 \UnaryInfC{$ \Rightarrow A$}
 \DisplayProof
 \\[3ex]
 \end{tabular}
 
 \begin{tabular}{c}
 \AxiomC{}
 \UnaryInfC{$ \Rightarrow t=t$}
 \DisplayProof
 \\[3ex]
 \AxiomC{}
 \UnaryInfC{$s_1=t_1, \ldots, s_n=t_n \Rightarrow f(\bar{s})=f(\bar{t})$}
 \DisplayProof
 \\[3ex]
 \AxiomC{}
 \UnaryInfC{$s_1=t_1, \ldots, s_n=t_n, Q(\bar{s}) \Rightarrow Q(\bar{t})$}
 \DisplayProof
 \\[3ex]
 \end{tabular}
 \end{center}
where $P$ ranges over all atomic formulas, $f$ ranges over all function symbols in the language, $Q$ ranges over all atmoic formulas or their negations, and $A$ ranges over all quantifier-free theorems of $\PV$.\\
 
\noindent \textbf{Structural Rules:}\\

 \begin{center}
 \begin{tabular}{c c}
 \AxiomC{$\Gamma, A, A \Rightarrow \Delta$}
  \RightLabel{\footnotesize$ Lc$}
 \UnaryInfC{$\Gamma, A \Rightarrow \Delta$}
 \DisplayProof \hspace{7pt}
 &
 \AxiomC{$\Gamma \Rightarrow A, A, \Delta$}
 \RightLabel{\footnotesize$Rc$}
 \UnaryInfC{$\Gamma \Rightarrow A, \Delta$}
 \DisplayProof \hspace{7pt}
 \\[4ex]
 \AxiomC{$\Gamma \Rightarrow \Delta$}
  \RightLabel{\footnotesize$ Lw$}
 \UnaryInfC{$\Gamma, A \Rightarrow \Delta$}
 \DisplayProof \hspace{5pt}
 &
 \AxiomC{$\Gamma \Rightarrow \Delta$}
 \RightLabel{\footnotesize$Rw$}
 \UnaryInfC{$\Gamma \Rightarrow A, \Delta$}
 \DisplayProof
 \\[4ex]
 \end{tabular}
 
\begin{tabular}{c c}
\AxiomC{$\Gamma \Rightarrow A, \Delta$}
\AxiomC{$\Pi, A \Rightarrow \Lambda$}
 \RightLabel{\footnotesize$cut$}
 \BinaryInfC{$\Gamma, \Pi \Rightarrow \Delta, \Lambda$}
 \DisplayProof
 \\[3ex]
 \end{tabular}
 \end{center}
 
\noindent \textbf{Logical Rules:}
\begin{center}
\begin{tabular}{c c}
 \AxiomC{$\Gamma, A_i \Rightarrow \Delta$}
 \RightLabel{\footnotesize$L \wedge$} 
 \LeftLabel{\footnotesize$i \in \{0, 1\}$} 
 \UnaryInfC{$\Gamma, A_0 \wedge A_1 \Rightarrow \Delta$}
 \DisplayProof
 &
 \AxiomC{$\Gamma \Rightarrow A, \Delta$}
 \AxiomC{$\Gamma \Rightarrow B, \Delta$}
 \RightLabel{\footnotesize$R \wedge$} 
 \BinaryInfC{$\Gamma \Rightarrow A \wedge B, \Delta$}
 \DisplayProof
  \\[4ex]
 \AxiomC{$\Gamma, A \Rightarrow \Delta$}
 \AxiomC{$\Gamma, B \Rightarrow \Delta$}
 \RightLabel{\footnotesize$L \vee$} 
 \BinaryInfC{$\Gamma, A \vee B \Rightarrow \Delta$}
 \DisplayProof
 &
 \AxiomC{$\Gamma \Rightarrow A_i, \Delta$}
 \RightLabel{\footnotesize$R \vee$} 
 \LeftLabel{\footnotesize$i \in \{0, 1\}$} 
 \UnaryInfC{$\Gamma \Rightarrow A_0 \vee A_1, \Delta$}
 \DisplayProof
 \\[4ex]
 \AxiomC{$\Gamma, A(t) \Rightarrow \Delta$}
 \RightLabel{\footnotesize$L \forall$}
 \UnaryInfC{$\Gamma, \forall x A(x) \Rightarrow \Delta$}
 \DisplayProof
 &
 \AxiomC{$\Gamma \Rightarrow A(y), \Delta$}
 \RightLabel{\footnotesize$R \forall$}
 \UnaryInfC{$\Gamma \Rightarrow \forall x A(x), \Delta$}
 \DisplayProof
  \\[4ex]
 \AxiomC{$\Gamma, A(y) \Rightarrow \Delta$}
 \RightLabel{\footnotesize$L \exists$}
 \UnaryInfC{$\Gamma, \exists x A(x) \Rightarrow \Delta$}
 \DisplayProof
 &
 \AxiomC{$\Gamma \Rightarrow A(t), \Delta$}
 \RightLabel{\footnotesize$R \exists$}
 \UnaryInfC{$\Gamma \Rightarrow \exists x A(x), \Delta$}
 \DisplayProof
 \\[4ex]
	\end{tabular}
\end{center}
In the rules $(R\forall)$ and $(L\exists)$, the variable $y$ should not appear in the consequence. 
Adding the rule
\begin{center}
    	\begin{tabular}{c}
\AxiomC{$\Gamma, \forall \gamma \prec \delta \; A(\gamma) \Rightarrow \Delta, A(\delta)$}
\RightLabel{\footnotesize$Ind_{\alpha}$} 
		\UnaryInfC{$\Gamma \Rightarrow \Delta, A(\theta)$}
		\DisplayProof
	\end{tabular}
\end{center}
to $\mathbf{LPV}$, we get $\mathbf{G}_0$. Note that in $(Ind_{\alpha})$, the variable $\delta$ should not appear in the consequence. Moreover, the constant $\theta \in \mathcal{O}$ is arbitrary and can take any value. For more on the proof theory of first-order theories and specially arithmetic, see \cite{BussProofTheory,BussProofTheoryArithmetic}.\\
By the usual cut reduction method \cite{BussProofTheory,BussProofTheoryArithmetic}, it is easy to prove that for any $\Gamma \cup \Delta \subseteq \forall_1$, if $\Gamma \Rightarrow \Delta$ is provable in $\TI(\forall_1, \prec)$ (resp., $\PV$), then it has a $\mathbf{G}_0$-proof (resp. $\mathbf{LPV}$-proof) consisting only of $\forall_1$-formulas.
For some practical reasons, we simplify the system $\mathbf{G}_0$ by changing the cut and the induction rules to the weak cut and weak induction rules, respectively:
\begin{center}
	\begin{tabular}{c c}
	    \AxiomC{$\Gamma \Rightarrow A$}
	    \AxiomC{$A \Rightarrow \Delta$}
	    \RightLabel{\footnotesize$wCut$} 
		\BinaryInfC{$\Gamma \Rightarrow \Delta$}
		\DisplayProof \hspace{7pt}
		&
		\AxiomC{$\Gamma, \forall \gamma \prec \delta \; A(\gamma) \Rightarrow \forall \gamma \prec \delta+1 \; A(\gamma)$}
	    \RightLabel{\footnotesize$wInd_{\alpha}$} 
		\UnaryInfC{$\Gamma \Rightarrow A(\theta)$}
		\DisplayProof
	\end{tabular}
\end{center}	
Denote this system by $\mathbf{G}_1$. Note that the difference between $(Ind_{\alpha})$ and $(wInd_{\alpha})$ is that in the latter $\Delta$ is omitted and $A(\delta)$ is replaced by $\forall \gamma \prec \delta+1 \; A(\gamma)$.  
\begin{lemma}\label{CutConsequence}
For any $\Gamma \cup \Delta \subseteq \forall_1$, if $\TI(\forall_1, \prec) \vdash \bigwedge \Gamma \rightarrow \bigvee \Delta$, then $\Gamma \Rightarrow \Delta$ has a $\mathbf{G}_1$-proof only consisting of $\forall_1$-formulas. 
\end{lemma}
\begin{proof}
By a $\forall_1$-proof in $\mathbf{G}_1$ (resp. $\mathbf{LPV}$), we mean a proof in $\mathbf{G}_1$ (resp. $\mathbf{LPV}$) consisting only of $\forall_1$-formulas. We show that the cut rule and the induction rule (over $\forall_1$-formulas) are derivable in $\mathbf{G}_1$ (by a $\forall_1$-proof). We only investigate the harder case of $\forall_1$-proofs. The other is the same omitting the restrictions everywhere.\\
For cut, consider the following proof-tree in $\mathbf{G}_1$, where the double lines mean simple omitted proofs in $\mathbf{G}_1$. The tree proves $\Gamma, \Sigma \Rightarrow \Lambda, \Delta$ from $\Gamma \Rightarrow A, \Delta$ and $\Sigma, A \Rightarrow \Lambda$.
\begin{center}
	\begin{tabular}{c}
	\AxiomC{ }
	\doubleLine
	\UnaryInfC{$\Sigma \Rightarrow \bigwedge \Sigma$}
	\doubleLine
	\UnaryInfC{$\Gamma, \Sigma \Rightarrow \bigwedge \Sigma, \Delta$}
	    \AxiomC{$\Gamma \Rightarrow A, \Delta$}
	    \doubleLine
		\UnaryInfC{$\Gamma, \Sigma \Rightarrow A, \Delta$}
		\BinaryInfC{$\Gamma, \Sigma \Rightarrow A \wedge \bigwedge \Sigma, \Delta$}
		\doubleLine
		\UnaryInfC{$\Gamma, \Sigma \Rightarrow (A \wedge \bigwedge \Sigma) \vee \bigvee \Delta$}

     \AxiomC{ }
    \doubleLine
	\UnaryInfC{$\bigvee \Delta \Rightarrow \Delta$}
	\UnaryInfC{$\bigvee \Delta \Rightarrow \Lambda, \Delta$}
	    \AxiomC{$\Sigma, A \Rightarrow \Lambda$}
	    \doubleLine
		\UnaryInfC{$A \wedge \bigwedge \Sigma \Rightarrow \Lambda$}
		\doubleLine
		\UnaryInfC{$A \wedge \bigwedge \Sigma \Rightarrow \Lambda, \Delta$}
		\BinaryInfC{$(A \wedge \bigwedge \Sigma) \vee \bigvee \Delta \Rightarrow \Lambda, \Delta$}
		\RightLabel{\footnotesize$wCut$} 
		\BinaryInfC{$\Gamma, \Sigma \Rightarrow \Lambda, \Delta$}
		\DisplayProof
	\end{tabular}
\end{center}
Note that the simulation of the cut rule in $\mathbf{G}_1$ implies that $\mathbf{G}_1$ is as powerful as $\mathbf{LPV}$. It also transforms a $\forall_1$-proof in $\mathbf{LPV}$ to a $\forall_1$-proof in $\mathbf{G}_1$. For the induction rule, consider the following proof-tree proving $\Gamma \Rightarrow A(\theta), \Delta$ from $\Gamma, \forall \gamma \prec \delta \; A(\gamma) \Rightarrow A(\delta), \Delta$:
\begin{center}
	\begin{tabular}{c}
	\AxiomC{ }
	    \doubleLine
	    \UnaryInfC{$\Gamma, \bigvee \Delta \Rightarrow [A(\delta) \vee \bigvee \Delta]$}
	
	    \AxiomC{$\Gamma, \forall \gamma \prec \delta \; A(\gamma) \Rightarrow A(\delta), \Delta$}
	    \doubleLine
	    \UnaryInfC{$\Gamma, \forall \gamma \prec \delta \; A(\gamma) \Rightarrow [A(\delta) \vee \bigvee \Delta]$}
	    \BinaryInfC{$\Gamma, [\forall \gamma \prec \delta \; A(\gamma)] \vee \bigvee \Delta \Rightarrow [A(\delta) \vee \bigvee \Delta]$}
	    \doubleLine
	    \RightLabel{$*$}
	    \UnaryInfC{$\Gamma, \forall \gamma \prec \delta \, [A(\gamma) \vee \bigvee \Delta] \Rightarrow [A(\delta) \vee \bigvee \Delta]$}
	    \doubleLine
	    \RightLabel{$**$}
	   	\UnaryInfC{$\Gamma, \forall \gamma \prec \delta \, [A(\gamma) \vee \bigvee \Delta] \Rightarrow \forall \gamma \prec \delta+1 \, [A(\gamma) \vee \bigvee \Delta]$}
	   	 \RightLabel{\footnotesize$wInd$} 
	   	\UnaryInfC{$\Gamma \Rightarrow A(\theta) \vee \bigvee \Delta$}
	   	 \RightLabel{$\dagger$}
		\UnaryInfC{$\Gamma \Rightarrow A(\theta), \Delta$}
		\DisplayProof
	\end{tabular}
\end{center}
where $(*)$ is the result of a cut with the sequent $\forall \gamma \prec \delta \, [A(\gamma) \vee \bigvee \Delta]  \Rightarrow [\forall \gamma \prec \delta \; A(\gamma)] \vee \bigvee \Delta$ which has a proof in $\mathbf{LPV}$ and hence a $\forall_1$-proof in $\mathbf{LPV}$ and by the observation we have just made, a $\forall_1$-proof in $\mathbf{G}_1$. Note that the use of cut is allowed as we showed its derivability in $\mathbf{G}_1$. Moreover, $(**)$ is the result of a cut with the $\PV$-provable sequent $[A(\delta) \vee \bigvee \Delta], \forall \gamma \prec \delta \, [A(\gamma) \vee \bigvee \Delta] \Rightarrow \forall \gamma \prec \delta+1 \, [A(\gamma) \vee \bigvee \Delta]$. The latter is provable in $\mathbf{LPV}$. Therefore, it has a $\forall_1$-proof in $\mathbf{LPV}$ and hence in $\mathbf{G}_1$. Finally, $\dagger$ is the result of a cut with 
$ A(\theta) \vee \bigvee \Delta \Rightarrow A(\theta), \Delta$ that has a trivial $\forall_1$-proof.
\end{proof}

\subsection{Ordinal Flows} \label{SectionOfFlows}

In this subsection, we will witness $\TI(\forall_1, \prec)$-provable implications between $\forall_1$-formulas by a sequence of $\beta $ many $\PV$-provable implications, for some $\beta \in \mathcal{O}$. 
\begin{definition}\label{OrdinalFlow}
Let $A(\bar{x}), B(\bar{x}) \in \forall_1$. A pair $(H(\gamma, \bar{x}), \beta)$ of a $\forall_1$-formula and $\beta \in \mathcal{O}$ such that $\beta \succeq 1$ is called an \emph{$\alpha$-flow} from $A(\bar{x})$ to $B(\bar{x})$, if:
\begin{description}
\item[$\bullet$]
$\PV \vdash A(\bar{x}) \leftrightarrow H(0, \bar{x})$.
\item[$\bullet$]
$\PV \vdash \forall \; 1 \preceq \delta \preceq \beta \; [\forall \gamma \prec \delta \; H(\gamma, \bar{x}) \rightarrow H(\delta, \bar{x})]$.
\item[$\bullet$]
$\PV \vdash H(\beta, \bar{x}) \leftrightarrow B(\bar{x})$.
\end{description}
We denote the existence of an $\alpha$-flow from $A(\bar{x})$ to $B(\bar{x})$ by $A(\bar{x}) \rhd_{\alpha} B(\bar{x})$. For any multisets $\Gamma$ and $\Delta$ of $\forall_1$-formulas, by $\Gamma \rhd_{\alpha} \Delta$, we mean $\bigwedge \Gamma \rhd_{\alpha} \bigvee \Delta$.
\end{definition}
In order to use $\alpha$-flows to witness the proofs in $\TI(\forall_1, \prec)$, we will develop a high level calculus for this new notion, implemented in the following series of lemmas.

\begin{lemma}\label{ImplicationToFlow} Let $A(\bar{x}), B(\bar{x}), C(\bar{x}) \in \forall_1$. Then:
\begin{description}
 \item[$(i)$]
If $\PV \vdash A(\bar{x}) \to B(\bar{x})$, then $A(\bar{x}) \rhd_{\alpha} B(\bar{x})$.
 \item[$(ii)$]
If $A(\bar{x}) \rhd_{\alpha} B(\bar{x}) $, then $A(\bar{x}) \circ C(\bar{x}) \rhd_{\alpha} B(\bar{x}) \circ C(\bar{x}) $, for any $\circ \in \{\wedge, \vee\}$.
\end{description}
\end{lemma}
\begin{proof}
For $(i)$, set $\beta=1$ and $H(\gamma, \bar{x})=(\gamma=0 \to A(\bar{x})) \wedge (\gamma=1 \to B(\bar{x}))$. It is clear that $\PV \vdash H(0, \bar{x}) \leftrightarrow A(\bar{x})$ and $\PV \vdash H(1, \bar{x}) \leftrightarrow B(\bar{x})$. As $\PV \vdash A(\bar{x}) \to B(\bar{x})$, we can see that $(H(\gamma, \bar{x}), \beta)$ is an $\alpha$-flow from $A(\bar{x})$ to $B(\bar{x})$.\\
For $(ii)$, we only prove the conjunction case. The disjunction case is similar. Since $A(\bar{x}) \rhd_{\alpha} B(\bar{x})$, by Definition \ref{OrdinalFlow}, there exist an ordinal $\beta \succeq 1$ and a formula $H(\gamma, \bar{x}) \in \forall_1$ satisfying the conditions in Definition \ref{OrdinalFlow}. Set $I(\gamma, \bar{x})=H(\gamma, \bar{x}) \wedge C(\bar{x})$ and note that $I(\gamma, \bar{x}) \in \forall_1$. It is easy to see that the pair $(I(\gamma, \bar{x}), \beta)$ is an $\alpha$-flow from $A(\bar{x}) \wedge C(\bar{x})$ to $B(\bar{x}) \wedge C(\bar{x})$, as the $\PV$-provability of $\forall \; 1 \preceq \delta \preceq \beta \; [\forall \gamma \prec \delta \; H(\gamma, \bar{x}) \rightarrow H(\delta, \bar{x})]$ implies the $\PV$-provability of $\forall \; 1 \preceq \delta \preceq \beta \; [\forall \gamma \prec \delta \; (H(\gamma, \bar{x}) \wedge C(\bar{x})) \rightarrow (H(\delta, \bar{x}) \wedge C(\bar{x}))]$.
\end{proof}

In the next lemma, we glue $\alpha$-flows together to construct longer $\alpha$-flows. Notice that the proof heavily uses the fact that the operations $\{+, \dotminus, \cdot, d\}$ and their basic properties are representable in $\PV$.

\begin{lemma}\label{Gluing}
\begin{description}
\item[$(i)$]
If $A(\bar{x}) \rhd_{\alpha} B(\bar{x}) $ and $ B(\bar{x}) \rhd_{\alpha} C(\bar{x})$, then $A(\bar{x}) \rhd_{\alpha} C(\bar{x})$.
\item[$(ii)$]
If $\Gamma, \forall \gamma \prec \delta \; A(\gamma, \bar{x}) \rhd_{\alpha} \forall \gamma \prec \delta+1 \; A(\gamma, \bar{x})$, then $\Gamma \rhd_{\alpha} A(\theta, \bar{x})$, for any $\theta \in \mathcal{O}$.
\end{description}
\end{lemma}
\begin{proof}
For $(i)$, as $A(\bar{x}) \rhd_{\alpha} B(\bar{x})$, there exists an $\alpha$-flow $(H(\gamma, \bar{x}), \beta)$ from $A(\bar{x})$ to $B(\bar{x})$. Similarly, as $B(\bar{x}) \rhd_{\alpha} C(\bar{x})$, there is an $\alpha$-flow $(H'(\gamma, \bar{x}), \beta')$ from $B(\bar{x})$ to $C(\bar{x})$. Set $\beta''=\beta+\beta'$ and
$H''(\gamma, \bar{x})=[\gamma \preceq \beta \to
H(\gamma, \bar{x})] \wedge
[\beta \prec \gamma \preceq \beta+\beta' \to H'(\gamma \dotminus \beta, \bar{x})]$.
We claim that the pair $(H''(\gamma, \bar{x}), \beta'')$ is an $\alpha$-flow from $A(\bar{x})$ to $C(\bar{x})$. First, note that $H''(0, \bar{x})$ is $\PV$-equivalent to $H(0, \bar{x})$ which is $\PV$-equivalent to $A(\bar{x})$. Similarly, as $(\beta+\beta')\dotminus \beta=\beta'$ is provable in $\PV$, we know that $H''(\beta+\beta', \bar{x})$ is $\PV$-equivalent to $H'(\beta', \bar{x})$ which is $\PV$-equivalent to $C(\bar{x})$. To prove
$
\PV \vdash \forall \; 1 \preceq \delta \preceq \beta'' \; [\forall \gamma \prec \delta \; H''(\gamma, \bar{x}) \rightarrow  H''(\delta, \bar{x})],
$
note that if $\delta \preceq \beta$, then the claim reduces to the same claim for $H(\gamma, \bar{x})$ which is provable. If $\beta \prec \delta \preceq \beta+\beta'$, assume $\forall \gamma \prec \delta \; H''(\gamma, \bar{x})$ to prove $H''(\delta, \bar{x})$ or equivalently $H'(\delta \dotminus \beta, \bar{x})$. Note that $\forall \gamma \prec \delta \; H''(\gamma, \bar{x})$ implies $\forall \beta \preceq \gamma \prec \delta \; H''(\gamma, \bar{x})$. As the interval $(0, \delta \dotminus \beta)$ is isomorphic to $(\beta, \delta)$, by the map $\gamma \mapsto \beta+\gamma$, then $\forall \beta \preceq \gamma \prec \delta \; H''(\gamma, \bar{x})$ implies $\forall 0 \prec \gamma \prec \delta \dotminus \beta \; H''(\beta+\gamma, \bar{x})$ which implies $\forall 0 \prec \gamma \prec \delta \dotminus \beta \; H'(\gamma, \bar{x})$. On the other hand, $\forall \beta \preceq \gamma \prec \delta \; H''(\gamma, \bar{x})$ implies $H''(\beta, \bar{x})$ which is $\PV$-equivalent to $H(\beta, \bar{x})$, by definition.
As $H(\beta, \bar{x})$ is $\PV$-equivalent to $B(\bar{x})$ which is also $\PV$-equivalent to $H'(0, \bar{x})$, we can claim that $H(\beta, \bar{x})$ and $H'(0, \bar{x})$ are $\PV$-equivalent. Hence, $\forall \beta \preceq \gamma \prec \delta \; H''(\gamma, \bar{x})$ implies $\forall \gamma \prec \delta \dotminus \beta \; H'(\gamma, \bar{x})$ which also implies $H'(\delta \dotminus \beta, \bar{x})$, as $(H'(\gamma, \bar{x}), \beta')$ is an $\alpha$-flow.\\
For $(ii)$, as $\bigwedge \Gamma \wedge \forall \gamma \prec \delta \, A(\gamma, \bar{x}) \rhd_{\alpha} \forall \gamma \prec \delta+1 \; A(\gamma, \bar{x})$, by Lemma \ref{ImplicationToFlow}, we have $\bigwedge \Gamma \wedge \forall \gamma \prec \delta \, A(\gamma, \bar{x}) \rhd_{\alpha} \bigwedge \Gamma \wedge \forall \gamma \prec \delta+1 \; A(\gamma, \bar{x})$. Set $B(\delta, \bar{x})=\bigwedge \Gamma \wedge \forall \gamma \prec \delta \, A(\gamma, \bar{x})$. Therefore, $B(\delta, \bar{x}) \rhd_{\alpha} B(\delta+1, \bar{x})$. Let $(H(\eta, \delta, \bar{x}), \beta)$ be the $\alpha$-flow from $B(\delta, \bar{x})$ to $B(\delta+1, \bar{x})$. Note that $H(0, \delta, \bar{x})$ is $\PV$-equivalent to $B(\delta,\bar{x})$ and $H(\beta, \delta, \bar{x})$ is $\PV$-equivalent to $H(0, \delta+1, \bar{x})$, as both are $\PV$-equivalent to $B(\delta+1)$. Define $\beta'=\beta (\theta +1)$ and $I(\tau, \bar{x})=H(\tau \dotminus \beta d(\tau, \beta), d(\tau, \beta), \bar{x})$ and note that $I(\tau, \bar{x}) \in \forall_1$.
We show that $(I(\tau, \bar{x}), \beta')$ is an $\alpha$-flow from $B(0, \bar{x})$ to $B(\theta+1, \bar{x})$. Note that $(I(\tau, \bar{x}), \beta')$ is nothing but the result of gluing the $\alpha$-flows $(H(\eta, \delta, \bar{x}), \beta)$, for all $\delta \prec \theta+1$, one after another as depicted in the following figure, (for simplicity, in the figures, we drop the free variables $\bar{x}$).
% https://q.uiver.app/?q=WzAsMTYsWzAsMCwiQigwKSJdLFs1LDAsIkIoXFx0aGV0YSsxKSJdLFswLDEsIkgoMCwwKSJdLFs1LDEsIkgoMCxcXHRoZXRhKzEpIl0sWzEsMSwiSCgxLDApIl0sWzQsMSwiXFxjZG90cyJdLFszLDAsIkIoMSkiXSxbMCwyLCJJKDApIl0sWzEsMiwiSSgxKSJdLFs0LDIsIlxcY2RvdHMiXSxbNSwyLCJJKFxcYmV0YShcXHRoZXRhKzEpKSJdLFs0LDAsIlxcY2RvdHMiXSxbMiwxLCJcXGNkb3RzIl0sWzIsMiwiXFxjZG90cyJdLFszLDEsIkgoXFxiZXRhLDApIFxcZXF1aXYgSCgwLDEpIl0sWzMsMiwiSShcXGJldGEpIl0sWzIsNF0sWzAsMiwiXFxlcXVpdiIsMyx7InN0eWxlIjp7ImJvZHkiOnsibmFtZSI6Im5vbmUifSwiaGVhZCI6eyJuYW1lIjoibm9uZSJ9fX1dLFsxLDMsIlxcZXF1aXYiLDMseyJzdHlsZSI6eyJib2R5Ijp7Im5hbWUiOiJub25lIn0sImhlYWQiOnsibmFtZSI6Im5vbmUifX19XSxbNSwzXSxbMiw3LCJcXGVxdWl2IiwzLHsic3R5bGUiOnsiYm9keSI6eyJuYW1lIjoibm9uZSJ9LCJoZWFkIjp7Im5hbWUiOiJub25lIn19fV0sWzQsOCwiXFxlcXVpdiIsMyx7InN0eWxlIjp7ImJvZHkiOnsibmFtZSI6Im5vbmUifSwiaGVhZCI6eyJuYW1lIjoibm9uZSJ9fX1dLFszLDEwLCJcXGVxdWl2IiwzLHsic3R5bGUiOnsiYm9keSI6eyJuYW1lIjoibm9uZSJ9LCJoZWFkIjp7Im5hbWUiOiJub25lIn19fV0sWzcsOF0sWzksMTBdLFsxMSwxXSxbNCwxMl0sWzEyLDE0XSxbMTQsNV0sWzgsMTNdLFsxMywxNV0sWzE1LDldLFsxNCwxNSwiXFxlcXVpdiIsMyx7InN0eWxlIjp7ImJvZHkiOnsibmFtZSI6Im5vbmUifSwiaGVhZCI6eyJuYW1lIjoibm9uZSJ9fX1dLFswLDZdLFs2LDE0LCJcXGVxdWl2IiwzLHsic3R5bGUiOnsiYm9keSI6eyJuYW1lIjoibm9uZSJ9LCJoZWFkIjp7Im5hbWUiOiJub25lIn19fV0sWzYsMTFdXQ==
\[
{\footnotesize
\begin{tikzcd}
	{B(0)} &&& {B(1)} & \cdots & {B(\theta+1)} \\
	{H(0,0)} & {H(1,0)} & \cdots & {H(\beta,0) \equiv H(0,1)} & \cdots & {H(0,\theta+1)} \\
	{I(0)} & {I(1)} & \cdots & {I(\beta)} & \cdots & {I(\beta(\theta+1))}
	\arrow[from=2-1, to=2-2]
	\arrow["\equiv"{marking}, draw=none, from=1-1, to=2-1]
	\arrow["\equiv"{marking}, draw=none, from=1-6, to=2-6]
	\arrow[from=2-5, to=2-6]
	\arrow["\equiv"{marking}, draw=none, from=2-1, to=3-1]
	\arrow["\equiv"{marking}, draw=none, from=2-2, to=3-2]
	\arrow["\equiv"{marking}, draw=none, from=2-6, to=3-6]
	\arrow[from=3-1, to=3-2]
	\arrow[from=3-5, to=3-6]
	\arrow[from=1-5, to=1-6]
	\arrow[from=2-2, to=2-3]
	\arrow[from=2-3, to=2-4]
	\arrow[from=2-4, to=2-5]
	\arrow[from=3-2, to=3-3]
	\arrow[from=3-3, to=3-4]
	\arrow[from=3-4, to=3-5]
	\arrow["\equiv"{marking}, draw=none, from=2-4, to=3-4]
	\arrow[from=1-1, to=1-4]
	\arrow["\equiv"{marking}, draw=none, from=1-4, to=2-4]
	\arrow[from=1-4, to=1-5]
\end{tikzcd}
}
\]
First, as $d(0,\beta)=0$ and $0 \dotminus \beta d(0,\beta)=0$, provably in $\PV$, we know that $I(0, \bar{x})$ is $\PV$-equivalent to $H(0, 0, \bar{x})$ which is itself $\PV$-equivalent to $B(0, \bar{x})$. Secondly, as $d(\beta (\theta+1), \beta)=\theta+1$ and $\beta (\theta+1) \dotminus \beta d(\beta (\theta+1), \beta)=0$, provably in $\PV$, we know that $I(\beta (\theta+1), \bar{x})$ is $\PV$-equivalent to $H(0, \theta+1, \bar{x})$ which is $\PV$-equivalent to $B(\theta+1, \bar{x})$. For the middle condition, we must prove 
$
\PV \vdash \forall \; 1 \preceq \tau \preceq \beta (\theta+1) \; [\forall \zeta \prec \tau \; I(\zeta, \bar{x}) \rightarrow I(\tau, \bar{x})]$.
There are two cases to consider, either $\beta d(\tau, \beta) \prec \tau$ or $\beta d(\tau, \beta)= \tau$. If $\beta d(\tau, \beta) \prec \tau$ then $\beta d(\tau, \beta)+1 \preceq \tau$ which implies $\tau=\beta d(\tau, \beta)+\mu$ for $\mu=\tau \dotminus \beta d(\tau, \beta) \succeq 1$. As for any $\eta \prec \mu$, we have $\beta d(\tau, \beta) + \eta \prec \tau $, we know that $\forall \zeta \prec \tau \; I(\zeta, \bar{x})$ implies $\forall \eta \prec \mu \; H(\eta, d(\tau, \beta), \bar{x})$. As we have $\mu \succeq 1$, the latter proves $H(\mu, d(\tau, \beta), \bar{x})$ which is $\PV$-equivalent to $I(\tau, \bar{x})$. 
% https://q.uiver.app/?q=WzAsMTIsWzEsMCwiSCgwLGQoXFx0YXUsIFxcYmV0YSkpIl0sWzQsMCwiSChcXG11LGQoXFx0YXUsIFxcYmV0YSkpIl0sWzUsMCwiXFxjZG90cyJdLFsxLDEsIkkoXFxiZXRhIGQoXFx0YXUsIFxcYmV0YSkpIl0sWzQsMSwiSShcXHRhdSkiXSxbNSwxLCJcXGNkb3RzIl0sWzAsMCwiXFxjZG90cyJdLFswLDEsIlxcY2RvdHMiXSxbMywxLCJcXGNkb3RzIl0sWzMsMCwiXFxjZG90cyJdLFsyLDAsIkgoMSxkKFxcdGF1LCBcXGJldGEpKSJdLFsyLDEsIkkoXFxiZXRhIGQoXFx0YXUsIFxcYmV0YSkrMSkiXSxbMSwyXSxbMCwzLCJcXGVxdWl2IiwzLHsic3R5bGUiOnsiYm9keSI6eyJuYW1lIjoibm9uZSJ9LCJoZWFkIjp7Im5hbWUiOiJub25lIn19fV0sWzEsNCwiXFxlcXVpdiIsMyx7InN0eWxlIjp7ImJvZHkiOnsibmFtZSI6Im5vbmUifSwiaGVhZCI6eyJuYW1lIjoibm9uZSJ9fX1dLFs0LDVdLFs2LDBdLFs3LDNdLFs5LDFdLFs4LDRdLFswLDEwXSxbMTAsOV0sWzMsMTFdLFsxMSw4XSxbMTAsMTEsIlxcZXF1aXYiLDMseyJzdHlsZSI6eyJib2R5Ijp7Im5hbWUiOiJub25lIn0sImhlYWQiOnsibmFtZSI6Im5vbmUifX19XV0=
\[
{\footnotesize
\begin{tikzcd}
	\cdots & {H(0,d(\tau, \beta))} & {H(1,d(\tau, \beta))} & \cdots & {H(\mu,d(\tau, \beta))} & \cdots \\
	\cdots & {I(\beta d(\tau, \beta))} & {I(\beta d(\tau, \beta)+1)} & \cdots & {I(\tau)} & \cdots
	\arrow[from=1-5, to=1-6]
	\arrow["\equiv"{marking}, draw=none, from=1-2, to=2-2]
	\arrow["\equiv"{marking}, draw=none, from=1-5, to=2-5]
	\arrow[from=2-5, to=2-6]
	\arrow[from=1-1, to=1-2]
	\arrow[from=2-1, to=2-2]
	\arrow[from=1-4, to=1-5]
	\arrow[from=2-4, to=2-5]
	\arrow[from=1-2, to=1-3]
	\arrow[from=1-3, to=1-4]
	\arrow[from=2-2, to=2-3]
	\arrow[from=2-3, to=2-4]
	\arrow["\equiv"{marking}, draw=none, from=1-3, to=2-3]
\end{tikzcd}
}
\]
For the other case, if $\beta d(\tau, \beta)=\tau$, we should use $\forall \zeta \prec \tau \; I(\zeta, \bar{x})$ to prove the formula $I(\tau, \bar{x})=H(0, d(\tau, \beta), \bar{x})$. Again, there are two cases to consider: either $d(\tau, \beta)$ is a successor or a limit ordinal. If $d(\tau, \beta)=\rho+1$, for some $\rho$, as $H(0, \rho+1, \bar{x})$ is $\PV$-equivalent to $H(\beta, \rho, \bar{x})$, it is enough to prove $H(\beta, \rho, \bar{x})$. As $\beta \rho+\eta \prec \beta \rho+\beta=\beta(\rho+1)=\tau$, for any $\eta \prec \beta$, we know that $\forall \zeta \prec \tau \, I(\zeta, \bar{x})$ implies $\forall \eta \prec \beta \, H(\eta, \rho, \bar{x})$ which implies $H(\beta, \rho, \bar{x})$.
% https://q.uiver.app/?q=WzAsMTIsWzEsMCwiSCgwLFxccmhvKSJdLFs0LDAsIkgoXFxiZXRhLCBcXHJobykgXFxlcXVpdiBIKDAsXFxyaG8rMSkiXSxbMiwwLCJIKDEsXFxyaG8pIl0sWzMsMCwiXFxjZG90cyJdLFsxLDEsIkkoXFxiZXRhXFxyaG8pIl0sWzAsMCwiXFxjZG90cyJdLFswLDEsIlxcY2RvdHMiXSxbNSwwLCJcXGNkb3RzIl0sWzUsMSwiXFxjZG90cyJdLFsyLDEsIkkoXFxiZXRhXFxyaG8rMSkiXSxbNCwxLCJJKFxcYmV0YShcXHJobysxKSkiXSxbMywxLCJcXGNkb3RzIl0sWzAsMl0sWzIsM10sWzMsMV0sWzUsMF0sWzYsNF0sWzEsN10sWzEwLDhdLFs0LDldLFs5LDExXSxbMTEsMTBdLFswLDQsIlxcZXF1aXYiLDMseyJzdHlsZSI6eyJib2R5Ijp7Im5hbWUiOiJub25lIn0sImhlYWQiOnsibmFtZSI6Im5vbmUifX19XSxbMiw5LCJcXGVxdWl2IiwzLHsic3R5bGUiOnsiYm9keSI6eyJuYW1lIjoibm9uZSJ9LCJoZWFkIjp7Im5hbWUiOiJub25lIn19fV0sWzEsMTAsIlxcZXF1aXYiLDMseyJzdHlsZSI6eyJib2R5Ijp7Im5hbWUiOiJub25lIn0sImhlYWQiOnsibmFtZSI6Im5vbmUifX19XV0=
\[
{\footnotesize
\begin{tikzcd}
	\cdots & {H(0,\rho)} & {H(1,\rho)} & \cdots & {H(\beta, \rho) \equiv H(0,\rho+1)} & \cdots \\
	\cdots & {I(\beta\rho)} & {I(\beta\rho+1)} & \cdots & {I(\beta(\rho+1))} & \cdots
	\arrow[from=1-2, to=1-3]
	\arrow[from=1-3, to=1-4]
	\arrow[from=1-4, to=1-5]
	\arrow[from=1-1, to=1-2]
	\arrow[from=2-1, to=2-2]
	\arrow[from=1-5, to=1-6]
	\arrow[from=2-5, to=2-6]
	\arrow[from=2-2, to=2-3]
	\arrow[from=2-3, to=2-4]
	\arrow[from=2-4, to=2-5]
	\arrow["\equiv"{marking}, draw=none, from=1-2, to=2-2]
	\arrow["\equiv"{marking}, draw=none, from=1-3, to=2-3]
	\arrow["\equiv"{marking}, draw=none, from=1-5, to=2-5]
\end{tikzcd}
}
\]
If $d(\tau, \beta)$ is a limit ordinal, then $\forall \zeta \prec \beta d(\tau, \beta) \, I(\zeta, \bar{x})$ implies the formula $\forall \delta \prec d(\tau, \beta) \, H(0, \delta, \bar{x})$ which implies $\forall \delta \prec d(\tau, \beta) B(\delta, \bar{x})$. The latter is $\forall \delta \prec d(\tau, \beta) [\bigwedge \Gamma \wedge \forall \gamma \prec \delta \, A(\gamma, \bar{x})]$ that implies $\bigwedge \Gamma \wedge \forall \gamma \prec d(\tau, \beta) \; A(\gamma, \bar{x})$, as $d(\tau, \beta)$ is a limit ordinal. The latter is $\PV$-equivalent to $H(0, d(\tau, \beta), \bar{x})=I(\tau, \bar{x})$. 
This completes the proof of the claim and shows that $B(0, \bar{x}) \rhd_{\alpha} B(\theta+1, \bar{x})$.
Now, as $\PV \vdash \bigwedge \Gamma \to (\bigwedge \Gamma \wedge \forall \gamma \prec 0 \; A(\gamma, \bar{x}))$ and $\PV \vdash (\bigwedge \Gamma \wedge \forall \gamma \prec \theta+1 \; A(\gamma, \bar{x})) \to A(\theta, \bar{x})$, by Lemma \ref{ImplicationToFlow}, we have $\bigwedge \Gamma \rhd_{\alpha} \bigwedge \Gamma \wedge \forall \gamma \prec 0 \; A(\gamma, \bar{x})$ and $\bigwedge \Gamma \wedge \forall \gamma \prec \theta+1 \; A(\gamma, \bar{x}) \rhd_{\alpha} A(\theta, \bar{x})$. Hence, by part $(i)$, we have $\bigwedge \Gamma \rhd_{\alpha} A(\theta, \bar{x})$ which completes the proof.
\end{proof}

\begin{lemma}\label{ConjAndDisj}(Conjunction and Disjunction Rules)
\begin{description}
\item[$(i)$]
If $\Gamma, A \rhd_{\alpha} \Delta$ or $\Gamma, B \rhd_{\alpha} \Delta$, then $\Gamma, A \wedge B \rhd_{\alpha} \Delta$.
\item[$(ii)$]
If $\Gamma \rhd_{\alpha} A, \Delta$ and $\Gamma \rhd_{\alpha} B, \Delta$, then $\Gamma \rhd_{\alpha} A \wedge B, \Delta$.
\item[$(iii)$]
If $\Gamma \rhd_{\alpha} A, \Delta$ or $\Gamma \rhd_{\alpha} B, \Delta$, then $\Gamma \rhd_{\alpha} A \vee B, \Delta$.
\item[$(iv)$]
If $\Gamma, A \rhd_{\alpha} \Delta$ and $\Gamma, B \rhd_{\alpha} \Delta$, then $\Gamma, A \vee B \rhd_{\alpha} \Delta$.
\end{description}
\end{lemma}
\begin{proof}
For $(i)$ and $(iii)$, as the implications $[(\bigwedge \Gamma \wedge (A \wedge B)) \to (\bigwedge \Gamma \wedge A)]$, $[(\bigwedge \Gamma \wedge (A \wedge B)) \to (\bigwedge \Gamma \wedge B)]$, $[(\bigvee \Delta \vee A) \to (\bigvee \Delta \vee (A \vee B))]$ and $[(\bigvee \Delta \vee B) \to (\bigvee \Delta \vee (A \vee B))]$ are all provable in $\PV$, using Lemma \ref{ImplicationToFlow} and Lemma \ref{Gluing}, we reach what we wanted.
For $(ii)$, if $\Gamma \rhd_{\alpha} \Delta, A $ then $\bigwedge \Gamma \rhd_{\alpha} \bigvee \Delta \vee A $, by definition. By Lemma \ref{ImplicationToFlow}, we reach $\bigwedge \Gamma \rhd_{\alpha} (\bigvee \Delta \vee A) \wedge \bigwedge \Gamma$. Similarly, we have $\bigwedge \Gamma \rhd_{\alpha} \bigvee \Delta \vee B$ and by Lemma \ref{ImplicationToFlow}, we reach
$\bigwedge \Gamma \wedge (\bigvee \Delta \vee A) \rhd_{\alpha} (\bigvee \Delta \vee B) \wedge (\bigvee \Delta \vee A)$.
Therefore,
$\bigwedge \Gamma \rhd_{\alpha} (\bigvee \Delta \vee B) \wedge (\bigvee \Delta \vee A)$, by part $(i)$ in Lemma \ref{Gluing}.
Finally, as
$(\bigvee \Delta \vee B) \wedge (\bigvee \Delta \vee A)  \to \bigvee \Delta \vee (A \wedge B)$ is provable in $\PV$, by Lemma \ref{ImplicationToFlow} and Lemma \ref{Gluing}, we reach
$\bigwedge \Gamma \rhd_{\alpha} \bigvee \Delta \vee (A \wedge B)$.
The proof for $(iv)$ is similar.
\end{proof}
Having the required lemmas, we are now ready to prove the following theorem as the main extraction technique that witnesses the proofs in $\TI(\forall_1, \prec)$ by $\alpha$-flows.
\begin{theorem}\label{SoundnessTheorem}
Let $\Gamma(\bar{x}) \cup \Delta(\bar{x}) \subseteq \forall_1$. Then, $\TI(\forall_1, \prec) \vdash \bigwedge \Gamma(\bar{x}) \rightarrow \bigvee \Delta(\bar{x})$ iff $\Gamma(\bar{x}) \rhd_{\alpha} \Delta(\bar{x})$.
\end{theorem}
\begin{proof}
We first prove the easier direction. Assume $\Gamma(\bar{x}) \rhd_{\alpha} \Delta(\bar{x})$ and the pair $(H(\gamma, \bar{x}), \beta)$ is an $\alpha$-flow from $\bigwedge \Gamma(\bar{x})$ to $\bigvee \Delta(\bar{x})$. As $\PV \vdash \forall \; 1 \preceq \delta \preceq \beta \; [\forall \gamma \prec \delta \; H(\gamma, \bar{x}) \rightarrow  H(\delta, \bar{x})]$ and $\TI(\forall_1, \prec)$ extends $\PV$, we have
\[
\TI(\forall_1, \prec) \vdash \forall \; 1 \preceq \delta \preceq \beta \; [\forall \gamma \prec \delta \; H(\gamma, \bar{x}) \rightarrow  H(\delta, \bar{x})].
\]
Then, as $H(\gamma, \bar{x}) \in \forall_1$, by the transfinite induction in $\TI(\forall_1, \prec)$, we reach $\TI(\forall_1, \prec) \vdash H(0, \bar{x}) \rightarrow H(\beta, \bar{x})$.
Finally, using the $\PV$-provable equivalences $\bigwedge \Gamma(\bar{x}) \leftrightarrow H(0, \bar{x})$ and
$H(\beta, \bar{x}) \leftrightarrow \bigvee \Delta(\bar{x})$, we reach $\TI(\forall_1, \prec) \vdash \bigwedge \Gamma(\bar{x}) \rightarrow \bigvee \Delta(\bar{x})$.\\
For the other direction, assume $\TI(\forall_1, \prec) \vdash \bigwedge \Gamma(\bar{x}) \rightarrow \bigvee \Delta(\bar{x})$. By Lemma \ref{CutConsequence}, $\Gamma(\bar{x}) \Rightarrow \Delta(\bar{x})$ has a $\mathbf{G}_1$-proof only consisting of $\forall_1$-formulas. By induction on this proof, we show that for any sequent $\Sigma \Rightarrow \Lambda$ in the proof, we have $\Sigma \rhd_{\alpha} \Lambda$.\\
For the axioms, as they are provable in $\PV$, using Lemma \ref{ImplicationToFlow}, there is nothing to prove. The case of structural rules (except for the weak cut) is easy. Weak cut and weak induction are addressed in Lemma \ref{Gluing}. The conjunction and disjunction rules are proved in Lemma \ref{ConjAndDisj}.
For the right universal quantifier rule, if $\Sigma(\bar{x}) \Rightarrow \Lambda(\bar{x}), \forall z  B(\bar{x}, z)$ is proved from $\Sigma(\bar{x}) \Rightarrow \Lambda(\bar{x}), B(\bar{x}, z)$, then by induction hypothesis, $\Sigma(\bar{x}) \rhd_{\alpha} \Lambda(\bar{x}), B(\bar{x}, z)$. Therefore, there exists an $\alpha$-flow $(H(\gamma, \bar{x}, z), \beta)$ from $\bigwedge \Sigma(\bar{x})$ to $B(\bar{x}, z) \vee \bigvee \Lambda(\bar{x})$. Define $I(\gamma, \bar{x})=\forall z H(\gamma, \bar{x}, z)$ and note that $I(\bar{x}, z) \in \forall_1$. It is easy to see that $(I(\gamma, \bar{x}), \beta)$ is an $\alpha$-flow from $\forall z [\bigwedge \Sigma(\bar{x})]$ to $\forall z [B(\bar{x}, z) \vee \bigvee \Lambda(\bar{x})]$, as $\PV$-provability of $\forall \gamma \prec \delta H(\gamma, z, \bar{x}) \rightarrow H(\delta, z, \bar{x})$
implies the $\PV$-provability of
$\forall \gamma \prec \delta \forall z H(\gamma, z, \bar{x}) \rightarrow \forall z H(\delta, z, \bar{x})$.
Finally, as $z$ does not occur as a free variable in $\Sigma(\bar{x}) \cup \Lambda(\bar{x})$, we have the $\PV$-equivalence between $\forall z [\bigwedge \Sigma(\bar{x})]$ and $\bigwedge \Sigma(\bar{x})$ and similarly between $\forall z [B(\bar{x}, z) \vee \bigvee \Lambda(\bar{x})]$ and $\bigvee \Lambda(\bar{x}) \vee \forall z  B(\bar{x}, z)$. Using Lemma \ref{ImplicationToFlow} and Lemma \ref{Gluing}, we can prove $\bigwedge \Sigma(\bar{x}) \rhd_{\alpha} \bigvee \Lambda(\bar{x}) \vee \forall z  B(\bar{x}, z)$.
For the left universal quantifier rule, if $\Sigma(\bar{x}), \forall z B(\bar{x}, z) \Rightarrow \Lambda(\bar{x})$ is proved from $\Sigma(\bar{x}), B(\bar{x}, s(\bar{x})) \Rightarrow \Lambda(\bar{x})$, then by induction hypothesis
$
\Sigma(\bar{x}), B(\bar{x}, s(\bar{x})) \rhd_{\alpha} \Lambda(\bar{x})$. Since $\PV \vdash \bigwedge \Sigma(\bar{x}) \wedge \forall z B(\bar{x}, z) \rightarrow \bigwedge \Sigma(\bar{x}) \wedge B(\bar{x}, s(\bar{x}))$, by Lemma \ref{ImplicationToFlow}
and Lemma \ref{Gluing}, we reach
$\Sigma(\bar{x}), \forall z B(\bar{x}, z) \rhd_{\alpha} \Lambda(\bar{x})$.
\end{proof}

\begin{corollary}\label{MainCor} Let $\alpha$ be an ordinal with the ptime representation $\mathfrak{O}$. Then,
$\PA+\bigcup _{\beta \in \mathcal{O}}\TI(\prec_\beta) \vdash \bigwedge \Gamma(\bar{x}) \rightarrow \bigvee \Delta(\bar{x})$ iff $\Gamma(\bar{x}) \rhd_{\alpha} \Delta(\bar{x})$, for $\Gamma(\bar{x}) \cup \Delta(\bar{x}) \subseteq \forall_1$.
\end{corollary}
\begin{proof}
As any implication in the form $\bigwedge \Gamma(\bar{x}) \rightarrow \bigvee \Delta(\bar{x})$ is logically equivalent to a $\Pi^0_2$-formula, the claim is a consequence of Theorem \ref{SoundnessTheorem} and Corollary \ref{PAToTI}.
\end{proof}

\begin{corollary}\label{MainCorConcrete} 
Let $\mathfrak{O}_{0}$ be the ptime representation for $\epsilon_0$ introduced in Subsection \ref{Notation}. Then,
$\PA \vdash \bigwedge \Gamma(\bar{x}) \rightarrow \bigvee \Delta(\bar{x})$ iff $\Gamma(\bar{x}) \rhd_{\epsilon_0} \Delta(\bar{x})$, for $\Gamma(\bar{x}) \cup \Delta(\bar{x}) \subseteq \forall_1$.
\end{corollary}

\subsection{Ordinal Local Search Programs} \label{OrdinalLocalSearch}

In this subsection, we will first introduce the notion of an ordinal local search program as a formalized version of the transfinite ptime modifications over an initial ptime value that we explained before. We will then use these programs to witness some provable statements in the theory $\PA+\bigcup_{\beta \in \mathcal{O}}\TI(\prec_{\beta})$. 

\begin{definition}
Let $T$ be a theory over the language $\mathcal{L}_{\PV}$. A \emph{total search problem of $T$} is a quantifier-free formula $A(\bar{x}, \bar{y})$ such that $T \vdash \forall \bar{x} \exists \bar{y} A(\bar{x}, \bar{y})$. A total search problem is called an \emph{$\NP$-search problem}, if there are sequences of polynomials $\bar{r}$ such that $\PV \vdash A(\bar{x}, \bar{y}) \to |\bar{y}| \leq \bar{r}(|\bar{x}|)$, where $|\bar{y}| \leq \bar{r}(|\bar{x}|)$ is an abbreviation for $\bigwedge_i (|y_i| \leq r_i(|\bar{x}|))$. We denote the class of all these total search (resp., $\NP$-search) problems of $T$ by $\TSP(T)$ (resp. $\TFNP(T)$).
\end{definition}

\begin{definition}\label{t30-13}
Let $\alpha$ be an ordinal, $\mathfrak{O}$ be its ptime representation, $A(\bar{x}, \bar{y})$ be a quantifier-free formula in $\mathcal{L}_{\PV}$ and $\beta \in \mathcal{O}$. By an \emph{$\LS(\preceq_{\beta})$-program for $A(\bar{x}, \bar{y})$}, we mean the following data:
an initial sequence of $\mathcal{L}_{\PV}$-terms $\bar{i}(\bar{x})$,
a quantifier-free $\mathcal{L}_{\PV}$-formula $G(\gamma, \bar{x}, \bar{z})$,
a sequence of $\mathcal{L}_{\PV}$-terms $\bar{N}(\gamma, \bar{x}, \bar{z})$,
an $\mathcal{L}_{\PV}$-term $q(\gamma, \bar{x}, \bar{z})$,
a sequence of $\mathcal{L}_{\PV}$-terms $\bar{p}(\bar{x}, \bar{z})$,
such that:
\begin{description}
\item[$\bullet$]
$\PV \vdash G(\beta, \bar{x}, \bar{i}(\bar{x}))$,
\item[$\bullet$]
$\PV \vdash \gamma \neq 0 \to q(\gamma, \bar{x}, \bar{z}) \prec \gamma$,
\item[$\bullet$]
$\PV \vdash \gamma \neq 0 \to [G(\gamma, \bar{x}, \bar{z}) \to G(q(\gamma, \bar{x}, \bar{z}), \bar{x}, \bar{N}(\gamma, \bar{x}, \bar{z}))]$,
\item[$\bullet$]
$\PV \vdash G(0, \bar{x}, \bar{z}) \to A(\bar{x}, \bar{p}(\bar{x}, \bar{z})) $.
\end{description}
By $\LS(\preceq_{\beta})$, we mean the class of all formulas $A(\bar{x}, \bar{y})$ for which there exists a $\LS(\preceq_{\beta})$-program. By $\PLS(\preceq_{\beta})$, we mean the class $\LS(\preceq_{\beta}) \cap \TFNP(Th(\mathbb{N}))$.
\end{definition}

Membership $A(\bar{x}, \bar{y}) \in \LS(\preceq_{\beta})$ implies $\forall \bar{x} \exists \bar{y}A(\bar{x}, \bar{y})$ and the $\LS(\preceq_{\beta})$-program actually provides an algorithm to compute $\bar{y}$ from $\bar{x}$. To see this, denote $G(\gamma, \bar{x}, \bar{z})$ by $G_{\gamma}$. The algorithm starts at the level $\beta$ with an initial value $\bar{i}(\bar{x})$ satisfying the property $G_{\beta}$. Then, using the feasible function $q$, it finds a lower level to go to and uses the modification $\bar{N}$ to update any value with the property $G_{\gamma}$ to a value satisfying the property $G_{q(\gamma)}$. Finally, reaching the zeroth level, the algorithm uses $\bar{p}$ to compute $\bar{y}$ satisfying $A$ from any value with the property $G_0$. 

The next theorem uses $\LS(\preceq_{\beta})$-programs ($\PLS(\preceq_{\beta})$-programs) to witness the total search ($\NP$-search) problems of $\PA+\bigcup_{\beta \in \mathcal{O}}\TI(\prec_{\beta})$. The idea is using Herbrand's theorem, Theorem \ref{Herbrand}, applied on $\PV$ to push the data extraction of Corollary \ref{MainCor} a bit further to reach an ordinal local search program for total search problems.

\begin{theorem}\label{OrdFlow}
Let $\alpha$ be an ordinal with the ptime representation $\mathfrak{O}$. Then
    $\TSP(\PA+\bigcup_{\beta \in \mathcal{O}}\TI(\prec_{\beta}))= \bigcup_{\beta \in \mathcal{O}} \LS(\preceq_{\beta})$ and 
    $\TFNP(\PA+\bigcup_{\beta \in \mathcal{O}}\TI(\prec_{\beta}))= \bigcup_{\beta \in \mathcal{O}} \PLS(\preceq_{\beta})$.
\end{theorem}
\begin{proof}
We only prove the first equality. The second is just a consequence. For the first direction, assume that $A(\bar{x}, \bar{y})$ has a $\LS(\preceq_{\beta})$-program. Set $H(\gamma, \bar{x})=\forall \bar{z} \neg G(\gamma, \bar{x}, \bar{z}) \wedge \forall \bar{y} \neg A(\bar{x}, \bar{y})$ and note that $H \in \forall_1$. We claim that $(H(\gamma, \bar{x}), \beta)$ is an $\alpha$-flow from $\forall \bar{y} \neg A(\bar{x}, \bar{y})$ to $\bot$. First, as $\PV \vdash G(0, \bar{x}, \bar{z}) \to A(\bar{x}, \bar{p}(\bar{x}, \bar{z})) $, we have $\PV \vdash \forall \bar{y} \neg A(\bar{x}, \bar{y}) \to \forall \bar{z} \neg G(0, \bar{x}, \bar{z})$ and hence $\PV \vdash \forall \bar{y} \neg A(\bar{x}, \bar{y}) \leftrightarrow H(0, \bar{x})$. Secondly, as $\PV \vdash G(\beta, \bar{x}, \bar{i}(\bar{x}))$, we reach $\PV \vdash \forall \bar{z} \neg G(\beta, \bar{x}, \bar{z}) \leftrightarrow \bot$ and hence $\PV \vdash \bot \leftrightarrow H(\beta, \bar{x})$. Finally, using $\PV \vdash \gamma \neq 0 \to q(\gamma, \bar{x}, \bar{z}) \prec \gamma$,
and
\[
\PV \vdash \gamma \neq 0 \to [G(\gamma, \bar{x}, \bar{z}) \to G(q(\gamma, \bar{x}, \bar{z}), \bar{x}, \bar{N}(\gamma, \bar{x}, \bar{z}))],
\]
it is easy to see that 
\[
\PV \vdash \forall \; 1 \preceq \delta \preceq \beta \; [\neg G(q(\delta, \bar{x}, \bar{z}), \bar{x}, \bar{N}(\delta, \bar{x}, \bar{z})) \rightarrow \neg G(\delta, \bar{x}, \bar{z})]
\]
and hence we reach 
\[
\PV \vdash \forall \; 1 \preceq \delta \preceq \beta \; [\forall \gamma \prec \delta \; \forall \bar{z} \neg G(\gamma, \bar{x}, \bar{z}) \rightarrow \forall \bar{z} \neg G(\delta, \bar{x}, \bar{z})].
\]
 The latter implies $
\PV \vdash \forall \; 1 \preceq \delta \preceq \beta \; [\forall \gamma \prec \delta \; H(\gamma, \bar{x}) \rightarrow H(\delta, \bar{x})]$. Therefore, $(H(\gamma, \bar{x}), \beta)$ is an $\alpha$-flow from $\forall \bar{y} \neg A(\bar{x}, \bar{y})$ to $\bot$.
Hence, $\PA+\bigcup_{\beta \in \mathcal{O}}\TI(\prec_{\beta}) \vdash \forall \bar{y} \neg A(\bar{x}, \bar{y}) \to \bot$, by Corollary \ref{MainCor} and thus, we reach $\PA+\bigcup_{\beta \in \mathcal{O}}\TI(\prec_{\beta}) \vdash \forall \bar{x} \exists \bar{y} \, A(\bar{x}, \bar{y})$. For the converse, assume that $\PA+\bigcup_{\beta \in \mathcal{O}}\TI(\prec_{\beta}) \vdash \forall \bar{x} \exists \bar{y} A(\bar{x}, \bar{y})$, where $A(\bar{x}, \bar{y}) \in \mathcal{L}_{\PV}$ is quantifier-free. As $\PA+\bigcup_{\beta \in \mathcal{O}}\TI(\prec_{\beta}) \vdash \forall \bar{y} \neg A(\bar{x}, \bar{y}) \to \bot$, by Corollary \ref{MainCor}, $\forall \bar{y} \neg A(\bar{x}, \bar{y}) \rhd_{\alpha} \bot$. Hence, there exist $H(\gamma, \bar{x}) \in \forall_1$ and $\beta \in \mathcal{O}$ such that 
$\PV \vdash \forall \bar{y} \neg A(\bar{x}, \bar{y}) \leftrightarrow H(0, \bar{x})$, $\PV \vdash H(\beta, \bar{x}) \leftrightarrow \bot$ and
\[
\PV \vdash \forall \; 1 \preceq \delta \preceq \beta \; [\forall \gamma \prec \delta \; H(\gamma, \bar{x}) \rightarrow H(\delta, \bar{x})].
\]
As $H \in \forall_1$, there exists a quantifier-free formula $I(\gamma, \bar{x}, \bar{z})$ such that $H(\gamma, \bar{x})$ and $\forall \bar{z} I(\gamma, \bar{x}, \bar{z})$ are equivalent over $\PV$. On the other hand, as the implications are provable in $\PV$, we can witness the existential quantifiers by ptime functions. Hence, there are $\mathcal{L}_{\PV}$-terms $\bar{Y}(\bar{x}, \bar{z})$, $\bar{Z}(\gamma, \bar{x}, \bar{z})$, $\Delta(\gamma, \bar{x}, \bar{z})$ and $\bar{W}(\bar{x})$ such that
\begin{itemize}
\item[$\bullet$]
$\PV \vdash \neg A(\bar{x}, \bar{Y}(\bar{x}, \bar{z})) \rightarrow I(0, \bar{x}, \bar{z})$,
\item[$\bullet$]
$\PV \vdash  I(\beta, \bar{x}, \bar{W}(\bar{x})) \rightarrow \bot$,
\item[$\bullet$]
$\PV \vdash \forall 1 \preceq \delta \preceq \beta \; [[(\Delta(\delta, \bar{x}, \bar{z}) \prec \delta \rightarrow I(\Delta(\delta, \bar{x}, \bar{z}), \bar{x}, \bar{Z}(\delta, \bar{x}, \bar{z}))] \rightarrow  I(\delta, \bar{x}, \bar{z})]$. 
\end{itemize}
Define $G(\delta, \bar{x}, \bar{z})=\neg I(\delta, \bar{x}, \bar{z}) \wedge (\delta \preceq \beta)$,
\[
q(\delta, \bar{x}, \bar{z})=
\begin{cases}
\Delta(\delta, \bar{x}, \bar{z}) & \neg I(\delta, \bar{x}, \bar{z}) \wedge (\delta \preceq \beta) \\
0 & \text{otherwise}
\end{cases} 
\]
$\bar{i}(\bar{x})= \bar{W}(\bar{x})$ and $\bar{p}(\bar{x}, \bar{z})=\bar{Y}(\bar{x}, \bar{z})$. It is easy to see that this new data is an $\LS(\preceq_{\beta})$-program for $A(\bar{x}, \bar{y})$.
\end{proof}

Applying Theorem \ref{OrdFlow} to $\alpha=\epsilon_0$, we reach the following Corollary, originally proved in \cite{Be}.

\begin{corollary}\label{ConcreteCor}
Let $\mathfrak{O}_{0}$ be the ptime representation of the ordinal $\epsilon_0$ introduced in Subsection \ref{Notation}. Then
    $\TSP(\PA)= \bigcup_{\beta \in \mathcal{O}_{0}} \LS(\preceq_{\beta})$ and 
    $\TFNP(\PA)= \bigcup_{\beta \in \mathcal{O}_{0}} \PLS(\preceq_{\beta})$.
\end{corollary}

%\begin{remark}
%Note that the only machinery we used to develop the theory of flows is a ptime notation system for the ordinals less than $\epsilon_0$ in a way that the four operations $\{+, \dotminus, \cdot, d\}$ are ptime and their basic properties are provable in $\PV$. Therefore, the witnessing technique can be safely generalized to any theory with a proof-theoretic ordinal that has such a notation system.
%\end{remark}

\section{$k$-Flows and Bounded Arithmetic} \label{kFlowsAndArithmetic}
In this section, we will modify the method developed for the strong theories of arithmetic in Section \ref{OrdinalFlowsandArithmetic} to also cover the bounded and hence weaker theories of arithmetic. The structure of the present section is similar to that of Section \ref{OrdinalFlowsandArithmetic}. After recalling the usual sequent calculi for the theories $S^k_2$ and $T^k_2$ in Subsection \ref{SequntCalculiForBounded}, the next subsection, Subsection \ref{SubsectionKFlows} will be devoted to investigate a suitable version of a flow for bounded arithmetic called a $k$-flow. Roughly speaking, a $k$-flow is an exponentially long uniform sequence of $\PV$-provable implications between $\mathcal{L}_{\PV}$-formulas in the class $\hat{\Pi}^b_k$. After proving some basic properties of $k$-flows, we will conclude the subsection by proving a witnessing theorem, transforming the proofs of the implications between $\hat{\Pi}^b_k$-formulas in $S^k_2$ and $T^k_2$ to some types of $k$-flows. Finally, in Subsection \ref{SubsectionReduction}, we will introduce the appropriate notion of a local search program to witness the $\PV$-provable implications further and find a complete witnessing for the theories $S^k_2$ and $T^k_2$.
\subsection{Sequent Calculi for Bounded Arithmetic} \label{SequntCalculiForBounded}

To recall the usual sequent calculi for $S^k_2$ and $T^k_2$, introduced in \cite{BussThesis}, first consider the following rules:
\begin{flushleft}
  		\textbf{Bounded Quantifier Rules:}
\end{flushleft}
\begin{center}
  	\begin{tabular}{c c}
  		\AxiomC{$\Gamma, A(s)  \Rightarrow \Delta $}
  		\RightLabel{{\footnotesize $L\forall^{\leq} $}} 
  		\UnaryInfC{$ \Gamma, s \leq t, \forall y \leq t \; A(y) \Rightarrow \Delta $}
  		\DisplayProof
	  		&
	   	\AxiomC{$\Gamma, z \leq t  \Rightarrow A(z), \Delta $}
  		\RightLabel{{\footnotesize $R\forall^{\leq} $}} 
  		\UnaryInfC{$ \Gamma \Rightarrow \forall y \leq t \; A(y), \Delta $}
  		\DisplayProof
	   		\\[3 ex]
   		\AxiomC{$\Gamma, z \leq t, A(z) \Rightarrow \Delta $}
  		\RightLabel{{\footnotesize $L\exists^{\leq} $}} 
  		\UnaryInfC{$ \Gamma, \exists y \leq t \; A(y) \Rightarrow \Delta $}
  		\DisplayProof
	   		&
   		\AxiomC{$\Gamma \Rightarrow A(s), \Delta  $}
  		\RightLabel{{\footnotesize $R\exists^{\leq} $}} 
  		\UnaryInfC{$ \Gamma, s \leq t, \Rightarrow \exists y \leq t \; A(y), \Delta$}
  		\DisplayProof
   			\\[3 ex]
	\end{tabular}
\end{center}
\begin{flushleft}
  		\textbf{Induction Rules:}
\end{flushleft}

\begin{center}
    	\begin{tabular}{c c}
\AxiomC{$\Gamma, A(\lfloor \frac{z}{2} \rfloor) \Rightarrow A(z), \Delta$}
\RightLabel{\footnotesize$\PInd_k$} 
		\UnaryInfC{$\Gamma, A(0) \Rightarrow A(t), \Delta$}
		\DisplayProof
  &
\AxiomC{$\Gamma, A(z) \Rightarrow A(z+1), \Delta$}
\RightLabel{\footnotesize$\Ind_k$} 
		\UnaryInfC{$\Gamma, A(0) \Rightarrow A(t), \Delta$}
		\DisplayProof
	\end{tabular}
\end{center}
In the rules $(R\forall^{\leq})$ and $(L\exists^{\leq})$ as well as in the induction rules, the variable $z$ should not appear in the consequence of the rule. Moreover, in the induction rules $(\PInd_k)$ and $(\Ind_k)$, the index $k$ means that the formula $A(z)$ is restricted to the class $\hat{\Pi}^b_k$.\\
The system $\mathbf{LS^k_2}$ (resp. $\mathbf{LT^k_2}$) for $S^k_2$ (resp. $T^k_2$) is defined as the system $\mathbf{LPV}$ plus the bounded quantifier rules and the rule $(\PInd_k)$ (resp. $(\Ind_k)$). 
For some technical reasons, we prefer to work with the alternative systems where the cut and the induction rules are weakened. Define the system $\mathbf{wLS^k_2}$ (resp. $\mathbf{wLT^k_2}$) similar to $\mathbf{LS^k_2}$ (resp. $\mathbf{LT^k_2}$) with the difference that in the former the quantifier rules in $\mathbf{LPV}$ are \emph{omitted} and the cut and the induction rule $(\PInd_k)$ (resp. $(\Ind_k)$) are \emph{replaced} by the weak cut and the weak induction rule $(\wPInd_k)$ (resp. $(\wInd_k)$) depicted below:
\begin{center}
	\begin{tabular}{c c}
	    \AxiomC{$\Gamma \Rightarrow A$}
	    \AxiomC{$A \Rightarrow \Delta$}
	    \RightLabel{\footnotesize$wCut$} 
		\BinaryInfC{$\Gamma \Rightarrow \Delta$}
		\DisplayProof
	\end{tabular}
\end{center}	
\begin{center}
	\begin{tabular}{c c}
 \AxiomC{$\Gamma, A(\lfloor\frac{z}{2}\rfloor) \Rightarrow A(z)$}
	    \RightLabel{\footnotesize$\wPInd_k$} 
		\UnaryInfC{$\Gamma, A(0) \Rightarrow A(s)$}
		\DisplayProof
  \hspace{7pt}
  &
	    \AxiomC{$\Gamma, A(z) \Rightarrow A(z+1)$}
	    \RightLabel{\footnotesize$\wInd_k$} 
		\UnaryInfC{$\Gamma, A(0) \Rightarrow A(s)$}
		\DisplayProof
	
	\end{tabular}
\end{center}	
In the weak induction rules, we have the similar constraints as before, namely that $A \in \hat{\Pi}^b_k$ and $z$ does not appear in the consequence of the rules. Note that the only point modified in the weak induction rules is the missing context $\Delta$. 

The following theorem ensures that the system $\mathbf{wLS^k_2}$ (resp. $\mathbf{wLT^k_2}$) is complete for the sequents of $\hat{\Pi}^b_k$-formulas. Notice that the lemma does not claim the full completeness as the system  $\mathbf{wLS^k_2}$ (resp. $\mathbf{wLT^k_2}$) is clearly weak to introduce any unbounded quantifier. 
\begin{lemma}\label{kCutConsequence}
For any $\Gamma \cup \Delta \subseteq \hat{\Pi}^b_k$:
\begin{description}
    \item[$\bullet$]
If $S^k_2 \vdash \bigwedge \Gamma \rightarrow \bigvee \Delta$, then $\Gamma \Rightarrow \Delta$ has a $\mathbf{wLS^k_2}$-proof only consisting of $\hat{\Pi}^b_k$-formulas.
     \item[$\bullet$]
If $T^k_2 \vdash \bigwedge \Gamma \rightarrow \bigvee \Delta$, then $\Gamma \Rightarrow \Delta$ has a $\mathbf{wLT^k_2}$-proof only consisting of $\hat{\Pi}^b_k$-formulas.
\end{description}
\end{lemma}
\begin{proof}
It is a well-known consequence of the cut reduction theorem for $\mathbf{LS^k_2}$ (resp. $\mathbf{LT^k_2}$) that if $\bigwedge \Gamma \to \bigvee \Delta$ is provable in $S^k_2$ (resp. $T^k_2$), it has a proof in $\mathbf{LS^k_2}$ (resp. $\mathbf{LT^k_2}$) only consisting of $\hat{\Pi}^b_k$-formulas and only using bounded quantifier rules instead of the usual unbounded quantifier rules in $\mathbf{LPV}$ \cite{BussThesis,JanBook}. Therefore, the only thing remained to prove is simulating the cut and the induction rules over $\hat{\Pi}^b_k$-formulas by their weak versions applied over the same family of formulas. This simulation is almost identical to the one presented in the proof of Lemma \ref{CutConsequence} and hence will be skipped here.
\end{proof}

\subsection{$k$-Flows} \label{SubsectionKFlows}
In this subsection, we will first introduce a $k$-flow as a  uniform term-length sequence of $\PV$-provable implications between $\hat{\Pi}^b_k$-formulas. Then, we will develop a high-level calculus for $k$-flows to witness the provability in theories $S^k_2$ and $T^k_2$.
\begin{definition}\label{kDefFlow}
Let $A(\bar{x}), B(\bar{x}) \in \hat{\Pi}^b_k$ be two $\mathcal{L}_{\PV}$-formulas and $t(\bar{x})$ be an $\mathcal{L}_{\PV}$-term. A \emph{$k$-flow} from $A(\bar{x})$ to $B(\bar{x})$ with the length $t(\bar{x})$ is a pair $(H(u, \bar{x}), t(\bar{x}))$, where $H(u, \bar{x}) \in \hat{\Pi}^b_k$ and:
\begin{description}
\item[$\bullet$]
$\PV \vdash H(0, \bar{x}) \leftrightarrow A(\bar{x})$.
\item[$\bullet$]
$\PV \vdash  H(t(\bar{x}), \bar{x}) \leftrightarrow B(\bar{x})$.
\item[$\bullet$]
$\PV \vdash \forall u < t(\bar{x}) \; [H(u, \bar{x}) \rightarrow H(u+1, \bar{x})]$.
\end{description}
A $k$-flow is called \emph{polynomial} if $t(\bar{x})=q(|\bar{x}|)$, for some polynomial $q$, where by equality, we mean the syntactical equality between the terms.
If there exists a $k$-flow from $A(\bar{x})$ to $B(\bar{x})$ with the length $t(\bar{x})$, we write $A(\bar{x}) \rhd^{t(\bar{x})}_{k} B(\bar{x})$. If we intend to emphasize on the existence of the $k$-flow regardless of its length, we write $A(\bar{x}) \rhd_k B(\bar{x})$ and if the $k$-flow is polynomial $A(\bar{x}) \rhd^p_{k} B(\bar{x})$. Moreover, if $\Gamma \cup \Delta \subseteq \hat{\Pi}^b_k$, by $\Gamma \rhd_k \Delta$ (resp. $\Gamma \rhd^p_k \Delta$), we mean $\bigwedge \Gamma \rhd_k \bigvee \Delta$ (resp. $\bigwedge \Gamma \rhd^p_k \bigvee \Delta$).
\end{definition}

Similar to the situation with the ordinal flows, it is also reasonable to provide a high-level calculus to work with the $k$-flows. The following series of lemmas realize this goal. 

\begin{lemma}\label{Padding} (Padding)
Let $A(\bar{x}), B(\bar{x}) \in \hat{\Pi}^b_k$ and $t(\bar{x}), s(\bar{x})$ be two $\mathcal{L}_\PV$-terms such that $\PV \vdash t(\bar{x}) \leq s(\bar{x})$. If $A(\bar{x}) \rhd_k^{t(\bar{x})} B(\bar{x})$, then $A(\bar{x}) \rhd_k^{s(\bar{x})} B(\bar{x})$. Therefore, without loss of generality, we can always assume that the length $t(\bar{x})$ of a $k$-flow is $\PV$-monotone, i.e., $\PV \vdash \bigwedge_{i=1}^n (x_i \leq y_i) \to t(\bar{x}) \leq t(\bar{y})$.
\end{lemma}
\begin{proof}
Let $(H(u, \bar{x}),t(\bar{x}))$ be a $k$-flow from $A(\bar{x})$ to $B(\bar{x})$. Then,  define 
\[
H'(u, \bar{x})=
\begin{cases}
H(u, \bar{x}) &  u \leq t(\bar{x})\\
B(\bar{x}) & u > t(\bar{x})
\end{cases} 
\]
Notice that $H'(u, \bar{x}) \in \hat{\Pi}^b_k$.
It is easy to prove that $(H'(u, \bar{x}), s(\bar{x}))$ is a $k$-flow from $A(\bar{x})$ to $B(\bar{x})$. The only thing worth emphasizing is the role of 
the assumption $\PV \vdash t(\bar{x}) \leq s(\bar{x})$ in the proof. This assumption together with the definition of $H'(u, \bar{x})$ shows
$\PV \vdash  H'(s(\bar{x}), \bar{x}) \leftrightarrow B(\bar{x})$ which is one of the conditions of being a $k$-flow. This observation completes the proof of the first part of the claim. 
For its second part, note that for any term $t(\bar{x})$, there exists a polynomial $q$ such that $\PV \vdash t(\bar{x}) \leq 2^{q(|\bar{x}|)}$ \cite{BussThesis,JanBook}. As $2^{q(|\bar{x}|)}$ is $\PV$-monotone, it is enough to use the first part to extend a $k$-flow with the length $t(\bar{x})$ to a $k$-flow with the length $2^{q(|\bar{x}|)}$. For polynomial $k$-flows, as the length $t(\bar{x})$ is in the form $q(|\bar{x}|)$, for some polynomial $q$, it is already $\PV$-monotone and hence there is nothing to prove.
\end{proof}

\begin{lemma}\label{kImplicationToFlow} Let $A(\bar{x}), B(\bar{x}), C(\bar{x}) \in \hat{\Pi}^b_k$. Then:
\begin{description}
 \item[$(i)$]
If $\PV \vdash A(\bar{x}) \to B(\bar{x})$, then $A(\bar{x}) \rhd^p_k B(\bar{x})$.
 \item[$(ii)$]
If $A(\bar{x}) \rhd_k B(\bar{x}) $, then $A(\bar{x}) \circ C(\bar{x}) \rhd_k B(\bar{x}) \circ C(\bar{x}) $, for any $\circ \in \{\wedge, \vee\}$. A similar claim also holds for $\rhd^p_k$.
\end{description}
\end{lemma}
\begin{proof}
The proof is similar to that of Lemma \ref{ImplicationToFlow}.
%For $(i)$, use $p=1$ and $H(u, \bar{x})=(u=0 \to A(\bar{x})) \wedge (u=1 \to B(\bar{x}))$. It is easy to see that $\PV \vdash H(0, \bar{x}) \leftrightarrow A(\bar{x})$ and $\PV \vdash H(1, \bar{x}) \leftrightarrow B(\bar{x})$. As $\PV \vdash A(\bar{x}) \to B(\bar{x})$, we can see that $(H(u, \bar{x}), p)$ is a flow from $A(\bar{x})$ to $B(\bar{x})$.\\
%For $(ii)$, we only prove the conjunction case. The disjunction case is similar. As $A(\bar{x}) \rhd_k B(\bar{x})$, by Definition \ref{DefFlow}, there exist a polynomial $p$ and a formula $H(u, \bar{x}) \in \hat{\Pi}^b_k$ as in Definition \ref{DefFlow}. Set $I(u, \bar{x})=H(u, \bar{x}) \wedge C(\bar{x})$. It is clear that $I(u, \bar{x}) \in \hat{\Pi}^b_k$. It is also easy to see that the pair $(I(u, \bar{x}), p)$ is a $k$-flow from $A(\bar{x}) \wedge C(\bar{x})$ to $B(\bar{x}) \wedge C(\bar{x})$, as the $\PV$-provability of $\forall u < p(|\bar{x}|) \; [H(u, \bar{x}) \rightarrow H(u+1, \bar{x})]$ implies the $\PV$-provability of $\forall u < p(|\bar{x}|) \; [H(u, \bar{x}) \wedge C(\bar{x}) \rightarrow H(u+1, \bar{x}) \wedge C(\bar{x})]$.
\end{proof}

\begin{lemma}\label{Bounded}(Bounded variables)
Let $A(\bar{x}, y), B(\bar{x}, y) \in \hat{\Pi}^b_k$ be two $\mathcal{L}_{\PV}$-formulas and $s(\bar{x})$ be an $\mathcal{L}_{\PV}$-term (not depending on $y$). If $A(\bar{x}, y) \rhd_k B(\bar{x}, y)$, then there exists a formula $I(u, y, \bar{x}) \in \hat{\Pi}^b_k$ and an $\mathcal{L}_{\PV}$-term $r(\bar{x})$ (not depending on $y$) such that:
\begin{description}
\item[$\bullet$]
$\PV \vdash I(0, y, \bar{x}) \leftrightarrow A(\bar{x}, y)$.
\item[$\bullet$]
$\PV \vdash \forall y \leq s(\bar{x}) [I(r(\bar{x}), y, \bar{x}) \leftrightarrow B(\bar{x}, y)]$.
\item[$\bullet$]
$\PV \vdash I(u, y, \bar{x}) \rightarrow I(u+1, y, \bar{x})$.
\item[$\bullet$]
$\PV \vdash r(\bar{x}) \geq 1$.
\end{description}
If we also have $A(\bar{x}, y) \rhd^p_k B(\bar{x}, y)$, then the term $r(\bar{x})$ can be chosen in the form $q(|\bar{x}|)$, for some polynomial $q$.
\end{lemma}
\begin{proof}
Assume $(H(u, y, \bar{x}), t(y, \bar{x}))$ is a $k$-flow from $A(\bar{x}, y)$ to $B(\bar{x}, y)$. Using Lemma \ref{Padding}, we can assume that $t(y, \bar{x})$ is $\PV$-monotone and $\PV \vdash t(y, \bar{x}) \geq 1$. Define
\[
I(u, y, \bar{x})=
\begin{cases}
H(u, y, \bar{x}) &  u \leq t(y, \bar{x})\\
B(y, \bar{x}) & u > t(y,\bar{x})
\end{cases} 
\]
and notice that $I(u, y, \bar{x}) \in \hat{\Pi}^b_k$.
Recall from the basic facts in bounded arithmetic that for the term $s(\bar{x})$, there is a polynomial $q_s$ such that $\PV \vdash |s(\bar{x})| \leq q_s(|\bar{x}|)$ \cite{BussThesis,JanBook}. Define $r(\bar{x})=t(2^{q_s(|\bar{x}|)}, \bar{x})$ and note that
$\PV \vdash y \leq s(\bar{x}) \to t(y, \bar{x}) \leq r(\bar{x})$, as $t(y, \bar{x})$ is $\PV$-monotone and $\PV \vdash r(\bar{x}) \geq 1$.
We claim that $I(u, y, \bar{x})$ and $r(\bar{x})$ work.
The first and the third claims in the statement of the lemma are the trivial consequences of the fact that $(H(u, y, \bar{x}), t(y, \bar{x}))$ is a $k$-flow from $A(\bar{x}, y)$ to $B(\bar{x}, y)$. For the second, notice that as $\PV \vdash y \leq s(\bar{x}) \to t(y, \bar{x}) \leq r(\bar{x})$, we can use the definition of $I(y, \bar{x})$ to see that the formula $I(r(\bar{x}), \bar{x})$ is $\PV$-equivalent to $B(y, \bar{x})$.  \\
For the polynomial case, if $(H(u, y, \bar{x}), t(y, \bar{x}))$ is a polynomial $k$-flow from $A(\bar{x}, y)$ to $B(\bar{x}, y)$, then there is a polynomial $q_t$ such that $t(y, \bar{x})=q_t(|y|, |\bar{x}|)$. Therefore, $r(\bar{x})=q_t(q_s(|x|)+1, |\bar{x}|)$ which implies that $r(\bar{x})$ is in the form $q_r(|\bar{x}|)$, for some poynomial $q_r$.
\end{proof}

\begin{lemma}\label{kGlu}
Let $\Gamma(\bar{x}) \cup \{A(\bar{x}), B(\bar{x}), C(\bar{x}), D(y, \bar{x})\} \subseteq \hat{\Pi}^b_k$. Then:
\begin{description}
\item[$(i)$](weak gluing)
If $A(\bar{x}) \rhd_k B(\bar{x})$ and $ B(\bar{x}) \rhd_k C(\bar{x})$ then $A(\bar{x}) \rhd_k C(\bar{x})$. A similar claim also holds for $\rhd^p_k$.
\item[$(ii)$](polynomial strong gluing)
If $\Gamma(\bar{x}), D(\lfloor\frac{y}{2}\rfloor, \bar{x}) \rhd^p_k D(y, \bar{x})$, then we have $\Gamma(\bar{x}), D(0, \bar{x}) \rhd^p_k  D(s(\bar{x}), \bar{x})$, for any $\mathcal{L}_{\PV}$-term $s(\bar{x})$.
\item[$(iii)$](strong gluing)
If $ \Gamma(\bar{x}), D(y, \bar{x}) \rhd_k D(y+1, \bar{x})$, then $\Gamma(\bar{x}), D(0, \bar{x}) \rhd_k  D(s(\bar{x}), \bar{x})$, for any $\mathcal{L}_{\PV}$-term $s(\bar{x})$.
\end{description}
\end{lemma}
\begin{proof}
For $(i)$, as $A(\bar{x}) \rhd_k B(\bar{x})$ and $ B(\bar{x}) \rhd_k C(\bar{x})$, there exist $k$-flows $(H(u, \bar{x}), t(\bar{x}))$ and $(H'(u, \bar{x}), t'(\bar{x}))$, from $A(\bar{x})$ to $B(\bar{x})$ and from $B(\bar{x})$ to $C(\bar{x})$, respectively. Set $t''(\bar{x})=t(\bar{x})+t'(\bar{x})+1$ and 
\[
H''(u, \bar{x})=
\begin{cases}
H(u, \bar{x}) &  u \leq t(\bar{x})\\
H'(u \dotminus (t(\bar{x}) + 1), \bar{x}) & u > t(\bar{x})
\end{cases} 
\]
Notice that $H''(u, \bar{x})$ is clearly a $\hat{\Pi}^b_k$-formula. We claim that $(H''(u, \bar{x}), t''(\bar{x}))$ is a $k$-flow from $A(\bar{x})$ to $C(\bar{x})$ as depicted in the following figure, (for simplicity, in the figure, we sometimes drop the free variables $\bar{x}$):
% https://q.uiver.app/?q=WzAsMTYsWzAsMCwiQShcXGJhcnt4fSkiXSxbMiwwLCJCKFxcYmFye3h9KSJdLFs1LDAsIkMoXFxiYXJ7eH0pIl0sWzAsMSwiSCgwKSJdLFsyLDEsIkgodCkiXSxbMSwxLCJcXGNkb3RzIl0sWzMsMiwiSCcoMCkiXSxbNSwyLCJIJyh0JykiXSxbNCwyLCJcXGNkb3RzIl0sWzMsMCwiQihcXGJhcnt4fSkiXSxbMCwzLCJIJycoMCkiXSxbMiwzLCJIJycodCkiXSxbMywzLCJIJycodCsxKSJdLFs1LDMsIkgnJyh0K3QnKzEpIl0sWzEsMywiXFxjZG90cyJdLFs0LDMsIlxcY2RvdHMiXSxbMCwxXSxbNSw0XSxbOCw3XSxbOSwyXSxbMSw5XSxbMCwzLCJcXGVxdWl2IiwzLHsic3R5bGUiOnsiYm9keSI6eyJuYW1lIjoibm9uZSJ9LCJoZWFkIjp7Im5hbWUiOiJub25lIn19fV0sWzEsNCwiXFxlcXVpdiIsMyx7InN0eWxlIjp7ImJvZHkiOnsibmFtZSI6Im5vbmUifSwiaGVhZCI6eyJuYW1lIjoibm9uZSJ9fX1dLFs5LDYsIlxcZXF1aXYiLDMseyJzdHlsZSI6eyJib2R5Ijp7Im5hbWUiOiJub25lIn0sImhlYWQiOnsibmFtZSI6Im5vbmUifX19XSxbMiw3LCJcXGVxdWl2IiwzLHsic3R5bGUiOnsiYm9keSI6eyJuYW1lIjoibm9uZSJ9LCJoZWFkIjp7Im5hbWUiOiJub25lIn19fV0sWzE0LDExXSxbMTEsMTJdLFsxNSwxM10sWzMsNV0sWzEwLDE0XSxbNiw4XSxbMTIsMTVdXQ==
\[\small
\begin{tikzcd}
	{A(\bar{x})} && {B(\bar{x})} & {B(\bar{x})} && {C(\bar{x})} \\
	{H(0)} & \cdots & {H(t)} \\
	&&& {H'(0)} & \cdots & {H'(t')} \\
	{H''(0)} & \cdots & {H''(t)} & {H''(t+1)} & \cdots & {H''(t+t'+1)}
	\arrow[from=1-1, to=1-3]
	\arrow[from=2-2, to=2-3]
	\arrow[from=3-5, to=3-6]
	\arrow[from=1-4, to=1-6]
	\arrow[from=1-3, to=1-4]
	\arrow["\equiv"{marking}, draw=none, from=1-1, to=2-1]
	\arrow["\equiv"{marking}, draw=none, from=1-3, to=2-3]
	\arrow["\equiv"{marking}, draw=none, from=1-4, to=3-4]
	\arrow["\equiv"{marking}, draw=none, from=1-6, to=3-6]
	\arrow[from=4-2, to=4-3]
	\arrow[from=4-3, to=4-4]
	\arrow[from=4-5, to=4-6]
	\arrow[from=2-1, to=2-2]
	\arrow[from=4-1, to=4-2]
	\arrow[from=3-4, to=3-5]
	\arrow[from=4-4, to=4-5]
\end{tikzcd}\]
First, it is trivial that $H''(0, \bar{x})$ is $\PV$-equivalent to $H(0, \bar{x})$ which is $\PV$-equivalent to $A(\bar{x})$. Similarly, $H''(t''(\bar{x}), \bar{x})$ is $\PV$-equivalent to $H'(t'(\bar{x}), \bar{x})$ which is $\PV$-equivalent to $C(\bar{x})$. To prove $\PV \vdash \forall u < t''(\bar{x}) \; [H''(u, \bar{x}) \to H''(u+1, \bar{x})]$, the cases $u < t(\bar{x})$ and $t(\bar{x}) < u < t''(\bar{x})$ are reduced to a similar claim for $H$ and $H'$. For $u=t(\bar{x})$, note that $H''(t(\bar{x}), \bar{x})$ is $\PV$-equivalent to $H(t(\bar{x}), \bar{x})$ and $H''(t(\bar{x})+1, \bar{x})$ is $\PV$-equivalent to $H'(0, \bar{x})$. As both formulas are $\PV$-equivalent to $B(\bar{x})$, the proof is complete. Finally, note that if the $k$-flows $(H(u, \bar{x}), t(\bar{x}))$ and $(H'(u, \bar{x}), t'(\bar{x}))$ are polynomial, there are polynomials $q$ and $q'$ such that $t(\bar{x})=q(|\bar{x}|)$ and $t'(\bar{x})=q'(|\bar{x}|)$. Hence, $t''(\bar{x})=q(|\bar{x}|)+q'(|\bar{x}|)+1$. Therefore, the $k$-flow $(H''(u, \bar{x}), t''(\bar{x}))$ is also polynomial.

For $(ii)$,
as $\Gamma(\bar{x}), D(\lfloor \frac{y}{2} \rfloor, \bar{x}) \rhd^p_k D(y, \bar{x})$, by Lemma \ref{kImplicationToFlow}, we have 
$\bigwedge \Gamma(\bar{x}) \wedge D(\lfloor \frac{y}{2} \rfloor, \bar{x}) \rhd^p_k \bigwedge \Gamma \wedge D(y, \bar{x})$. For simplicity, denote $\bigwedge \Gamma(\bar{x}) \wedge D(y, \bar{x})$ by $E(y, \bar{x})$. Therefore, we have $E(\lfloor \frac{y}{2} \rfloor, \bar{x}) \rhd^p_k E(y, \bar{x})$. First, we want to prove $E(0, \bar{x}) \rhd^p_k E(s(\bar{x}), \bar{x})$. Roughly speaking, the idea is gluing the polynomial $k$-flows from $E(\lfloor \frac{y}{2} \rfloor, \bar{x})$ to $E(y, \bar{x})$, one after another, starting from $y=s(\bar{x})$ till reaching $E(0, \bar{x})$:
% https://q.uiver.app/?q=WzAsNSxbNCwwLCJFKHMoXFxiYXJ7eH0pLCBcXGJhcnt4fSkiXSxbMywwLCJFKFxcbGZsb29yIFxcZnJhY3tzKFxcYmFye3h9KX17Mn0gXFxyZmxvb3IsIFxcYmFye3h9KSJdLFsxLDAsIlxcY2RvdHMiXSxbMiwwLCJFKFxcbGZsb29yIFxcZnJhY3tcXGxmbG9vciBcXGZyYWN7cyhcXGJhcnt4fSl9ezJ9IFxccmZsb29yfXsyfSBcXHJmbG9vciwgXFxiYXJ7eH0pIl0sWzAsMCwiRSgwLCBcXGJhcnt4fSkiXSxbMSwwXSxbMywxXSxbMiwzXSxbNCwyXV0=
\[\begin{tikzcd}
	{E(0, \bar{x})} & \cdots & {E(\lfloor \frac{\lfloor \frac{s(\bar{x})}{2} \rfloor}{2} \rfloor, \bar{x})} & {E(\lfloor \frac{s(\bar{x})}{2} \rfloor, \bar{x})} & {E(s(\bar{x}), \bar{x})}
	\arrow[from=1-4, to=1-5]
	\arrow[from=1-3, to=1-4]
	\arrow[from=1-2, to=1-3]
	\arrow[from=1-1, to=1-2]
\end{tikzcd}\]
Notice that the result of this gluing extends the length of the $k$-flow by $|s(\bar{x})|$ which is bounded by a polynomial and hence acceptable. More formally, using Lemma \ref{Bounded} for the formulas $E(\lfloor \frac{y}{2} \rfloor, \bar{x})$ and $E(y, \bar{x})$ and the term $2s(\bar{x})$ (the choice of $2s(\bar{x})$ instead of $s(\bar{x})$ is rather technical) and using the fact that $E(\lfloor \frac{y}{2} \rfloor, \bar{x}) \rhd^p_k E(y, \bar{x})$, we reach a pair $(H'(u, y, \bar{x}), t'(\bar{x}))$ such that:
\begin{description}
\item[$(1)$]
$\PV \vdash H'(0, y, \bar{x}) \leftrightarrow E(\lfloor \frac{y}{2} \rfloor,\bar{x})$,
\item[$(2)$]
$\PV \vdash \forall y \leq 2s(\bar{x}) [H'(t'(\bar{x}), y, \bar{x}) \leftrightarrow E(y, \bar{x})]$,
\item[$(3)$]
$\PV \vdash H'(u, y, \bar{x}) \rightarrow H'(u+1, y, \bar{x})$,
\item[$(4)$]
$\PV \vdash t'(\bar{x}) \geq 1$,
\end{description}
and $t'(\bar{x})=q_{t'}(|\bar{x}|)$, for some polynomial $q_{t'}$. Define the function $Y(z, \bar{x})$ as the result of $|s(\bar{x})|+1\dotminus z$ many iterations of the operation $n \mapsto \lfloor\frac{n}{2}\rfloor$ on $2s(\bar{x})$. Note that the function is clearly polynomial time computable. Therefore, we can define it recursively in $\PV$ and represent it by an $\mathcal{L}_{\PV}$-term. This term is $\PV$-provably bounded by $2s(\bar{x})$, i.e., $\PV \vdash Y(z, \bar{x}) \leq 2s(\bar{x})$ and we have $Y(0, \bar{x})=0$, $Y(|s(\bar{x})|, \bar{x})=s(\bar{x})$ and if $z \leq |s(\bar{x})|$, then $Y(z, \bar{x})=\lfloor \frac{Y(z+1, \bar{x})}{2} \rfloor$, all provable in $\PV$. Now, define
\[
I(u, \bar{x})=H'(u \dotminus t'(\bar{x})\lfloor \frac{u}{t'(\bar{x})} \rfloor, Y(\lfloor \frac{u}{t'(\bar{x})} \rfloor+1), \bar{x}).
\]
Note that $I(u, \bar{x})$ is well-defined as $t'(\bar{x})$ is greater than zero, provably in $\PV$. It is trivial that $I(u, \bar{x}) \in \hat{\Pi}^b_k$. Set $r(\bar{x})=t'(\bar{x})|s(\bar{x})|$. We claim that the pair $(I(u, \bar{x}), r(\bar{x}))$ is a $k$-flow from $E(0, \bar{x})$ to $E(s(\bar{x}),\bar{x})$ as depicted in the following figure. For simplicity, we drop the free variables $\bar{x}$ in the figure.
% https://q.uiver.app/?q=WzAsMTIsWzQsMCwiRShzKSJdLFs0LDEsIkgnKHQnLHMpIl0sWzMsMSwiXFxjZG90cyJdLFs0LDIsIkkofHN8dCcpIl0sWzEsMCwiRShcXGxmbG9vciBcXGZyYWN7c317Mn0gXFxyZmxvb3IpIl0sWzAsMiwiXFxjZG90cyJdLFsxLDEsIkgnKHQnLFxcbGZsb29yIFxcZnJhY3tzfXsyfSBcXHJmbG9vcikgXFxlcXVpdiBIJygwLHMpIl0sWzEsMiwiSSgofHN8LTEpdCcpIl0sWzMsMiwiXFxjZG90cyJdLFsyLDEsIkgnKDEsIHMpIl0sWzIsMiwiSSgofHN8LTEpdCcrMSkiXSxbMCwwLCJcXGNkb3RzIl0sWzAsMSwiXFxlcXVpdiIsMyx7InN0eWxlIjp7ImJvZHkiOnsibmFtZSI6Im5vbmUifSwiaGVhZCI6eyJuYW1lIjoibm9uZSJ9fX1dLFsyLDFdLFsxLDMsIlxcZXF1aXYiLDMseyJzdHlsZSI6eyJib2R5Ijp7Im5hbWUiOiJub25lIn0sImhlYWQiOnsibmFtZSI6Im5vbmUifX19XSxbNCwwXSxbNSw3XSxbNiw3LCJcXGVxdWl2IiwzLHsic3R5bGUiOnsiYm9keSI6eyJuYW1lIjoibm9uZSJ9LCJoZWFkIjp7Im5hbWUiOiJub25lIn19fV0sWzgsM10sWzYsOV0sWzksMl0sWzcsMTBdLFsxMCw4XSxbMTEsNF0sWzQsNiwiXFxlcXVpdiIsMyx7InN0eWxlIjp7ImJvZHkiOnsibmFtZSI6Im5vbmUifSwiaGVhZCI6eyJuYW1lIjoibm9uZSJ9fX1dLFs5LDEwLCJcXGVxdWl2IiwzLHsic3R5bGUiOnsiYm9keSI6eyJuYW1lIjoibm9uZSJ9LCJoZWFkIjp7Im5hbWUiOiJub25lIn19fV1d
\[ \small
\begin{tikzcd}
	\cdots & {E(\lfloor \frac{s}{2} \rfloor)} &&& {E(s)} \\
	& {H'(t',\lfloor \frac{s}{2} \rfloor) \equiv H'(0,s)} & {H'(1, s)} & \cdots & {H'(t',s)} \\
	\cdots & {I(|\lfloor \frac{s}{2} \rfloor|t')} & {I(|\lfloor \frac{s}{2} \rfloor|t'+1)} & \cdots & {I(|s|t')}
	\arrow["\equiv"{marking}, draw=none, from=1-5, to=2-5]
	\arrow[from=2-4, to=2-5]
	\arrow["\equiv"{marking}, draw=none, from=2-5, to=3-5]
	\arrow[from=1-2, to=1-5]
	\arrow[from=3-1, to=3-2]
	\arrow["\equiv"{marking}, draw=none, from=2-2, to=3-2]
	\arrow[from=3-4, to=3-5]
	\arrow[from=2-2, to=2-3]
	\arrow[from=2-3, to=2-4]
	\arrow[from=3-2, to=3-3]
	\arrow[from=3-3, to=3-4]
	\arrow[from=1-1, to=1-2]
	\arrow["\equiv"{marking}, draw=none, from=1-2, to=2-2]
	\arrow["\equiv"{marking}, draw=none, from=2-3, to=3-3]
\end{tikzcd}\]
To prove, we first claim that 
\[
\PV \vdash \forall z \leq |s(\bar{x})| \, [I(t'(\bar{x})z, \bar{x}) \leftrightarrow E(Y(z, \bar{x}), \bar{x})] \quad \quad 
(*)
\]
The reason is that by definition, $I(t'(\bar{x})z, \bar{x})=H'(0,Y(z+1, \bar{x}), \bar{x})$ and the latter is $\PV$-equivalent to 
$E(\lfloor \frac{Y(z+1, \bar{x})}{2}\rfloor, \bar{x})$, by the property $(1)$ above.
Finally, since for any $z \leq |s(\bar{x})|$, we have $Y(z, \bar{x})=\lfloor \frac{Y(z+1, \bar{x})}{2} \rfloor$ provably in $\PV$, we reach the $\PV$-equivalence with $E(Y(z, \bar{x}), \bar{x})$.\\
Now, we prove that $(I(u, \bar{x}), r(\bar{x}))$ is a $k$-flow from $E(0, \bar{x})$ to $E(s(\bar{x}),\bar{x})$.
First, note that $I(0, \bar{x})$ is $\PV$-equivalent to $E(0, \bar{x})$, by substituting $z=0$ in $(*)$ and using the $\PV$-provable fact that $Y(0, \bar{x})=0$. Secondly, note that $I(r(\bar{x}), \bar{x})$ is $\PV$-equivalent to $E(s(\bar{x}), \bar{x})$, by substituting $z=|s(\bar{x})|$ in $(*)$ and using the $\PV$-provable fact that $Y(|s(\bar{x})|, \bar{x})=s(\bar{x})$.
Thirdly, to prove 
$\PV \vdash \forall u < t'(\bar{x}) \; [I(u, \bar{x}) \rightarrow I(u+1, \bar{x})]$,
there are two cases to consider: Either $u+1$ divides $t'(\bar{x})$ or not. In the latter case, we have $\lfloor \frac{u+1}{t'(\bar{x})} \rfloor=\lfloor \frac{u}{t'(\bar{x})} \rfloor$. By definition
$I(u, \bar{x})$ is $H'(u \dotminus t'(\bar{x})\lfloor \frac{u}{t'(\bar{x})} \rfloor, Y(\lfloor \frac{u}{t'(\bar{x})} \rfloor+1), \bar{x})$ while $I(u+1, \bar{x})$ is $H'(u+1 \dotminus t'(\bar{x})\lfloor \frac{u+1}{t'(\bar{x})} \rfloor, Y(\lfloor \frac{u+1}{t'(\bar{x})} \rfloor+1), \bar{x})$. Therefore, the former proves the latter by property $(3)$ above. 
For the first case, if
$t'(\bar{x}) | u+1$, then there exists $z \leq |s(\bar{x})|$ such that $u+1=t'(\bar{x})z$. Therefore, $I(u+1, \bar{x})$ is $I(t'(\bar{x})z, \bar{x})$ which is $\PV$-equivalent to $E(Y(z, \bar{x}), \bar{x})$ by $(*)$, and hence $\PV$-equivalent to 
$H'(t'(\bar{x}), Y(z, \bar{x}), \bar{x})$ by $(2)$, as $Y(z, \bar{x})$ is $\PV$-provably bounded by $2s(\bar{x})$. As $I(u, \bar{x})$ is  $H'(t'(\bar{x}) \dotminus 1, Y(z, \bar{x}), \bar{x})$ by definition, by $(3)$, the formula $I(u, \bar{x})$ implies $I(u+1, \bar{x})$ in $\PV$.\\
So far, we showed that
$(I(u, \bar{x}), r(\bar{x}))$ is a $k$-flow from $E(0, \bar{x})$ to $E(s(\bar{x}),\bar{x})$. Again, 
recall that for the term $s(\bar{x})$, there is a polynomial $q_s$ such that $\PV \vdash |s(\bar{x})| \leq q_s(|\bar{x}|)$ \cite{BussThesis,JanBook}. 
Hence, $\PV \vdash r(\bar{x}) \leq q_s(|\bar{x}|)q_{t'}(|\bar{x}|)$. Therefore, using Lemma \ref{Padding}, we can prove the existence of a $k$-flow with the length $q_s(|\bar{x}|)q_{t'}(|\bar{x}|)$ from 
$E(0, \bar{x})$ to $E(s(\bar{x}), \bar{x})$
which implies $E(0, \bar{x}) \rhd^p_k E(s(\bar{x}), \bar{x})$. Now, to complete the proof of $(ii)$, by the definition of $E(y, \bar{x})$, we have $\bigwedge \Gamma(\bar{x}) \wedge D(0, \bar{x}) \rhd^p_k \bigwedge \Gamma(\bar{x}) \wedge D(s(\bar{x}), \bar{x})$. As $\PV \vdash \bigwedge \Gamma(\bar{x}) \wedge D(s(\bar{x}), \bar{x}) \to D(s(\bar{x}), \bar{x}) $, by Lemma \ref{kImplicationToFlow}, we have $ \bigwedge \Gamma(\bar{x}) \wedge D(s(\bar{x}), \bar{x}) \rhd^p_k D(s(\bar{x}), \bar{x}) $. Hence, by the weak gluing, the part $(i)$ in the present lemma, we reach $\bigwedge \Gamma(\bar{x}) \wedge D(0, \bar{x}) \rhd^p_k D(s(\bar{x}), \bar{x})$.  \\
The proof of $(iii)$ is similar to that of $(ii)$ and even easier. In this case, one must again define $E(y, \bar{x})$ as $\bigwedge \Gamma(\bar{x}) \wedge D(y, \bar{x})$ and then glue the $k$-flows from $E(y, \bar{x})$ to $E(y+1, \bar{x})$, one after another, for all $0 \leq y < s(\bar{x})$.
\end{proof}
\begin{lemma}\label{kConAndDis}(Conjunction and Disjunction Rules)
Let $\Gamma \cup \Delta \cup \{A, B\} \subseteq \hat{\Pi}^b_k$. Then:
\begin{description}
\item[$(i)$]
If $\Gamma, A \rhd_k \Delta$ or $\Gamma, B \rhd_k \Delta$ then $\Gamma, A \wedge B \rhd_k \Delta$.
\item[$(ii)$]
If $\Gamma \rhd_k A, \Delta$ and $\Gamma \rhd_k B, \Delta$ then $\Gamma \rhd_k A \wedge B, \Delta$.
\item[$(iii)$]
If $\Gamma \rhd_k A, \Delta$ or $\Gamma \rhd_k B, \Delta$ then $\Gamma \rhd_k A \vee B, \Delta$.
\item[$(iv)$]
If $\Gamma, A \rhd_k \Delta$ and $\Gamma, B \rhd_k \Delta$ then $\Gamma, A \vee B \rhd_k \Delta$.
\end{description}
A similar claim also holds for $\rhd^p_k$.
\end{lemma}
\begin{proof}
The argument is identical to that of Lemma \ref{ConjAndDisj} claiming the same fact for the ordinal flows.
%For $(i)$ and $(iii)$, note that all the implications $[\bigwedge \Gamma \wedge (A \wedge B) \to \bigwedge \Gamma \wedge A$], [$\bigwedge \Gamma \wedge (A \wedge B) \to \bigwedge \Gamma \wedge B$], [$\bigvee \Delta \vee A \to \bigvee \Delta \vee (A \vee B)$], and [$\bigvee \Delta \vee B \to \bigvee \Delta \vee (A \vee B)$] are all $\PV$-provable. Therefore, using Lemma \ref{ImplicationToFlow} and Lemma \ref{Glu}, we reach  what we wanted.\\ 
%For $(ii)$, as $\Gamma \rhd_k A, \Delta$, we have $\bigwedge \Gamma \rhd_k (\bigvee \Delta \vee A)$. Hence, by conjunction application, Lemma \ref{}, we have $\bigwedge \Gamma \rhd_k \bigwedge \Gamma \wedge (\bigvee \Delta \vee A)$. Moreover, we have $\bigwedge \Gamma \rhd_k \bigvee \Delta \vee B$ which by conjunction application, Lemma \ref{}, implies $\bigwedge \Gamma \wedge (\bigvee \Delta \vee A) \rhd_k (\bigvee \Delta \vee B) \wedge (\bigvee \Delta \vee A) $. Hence, by weak gluing, Lemma \ref{Glu}, we have $\bigwedge \Gamma \rhd_k (\bigvee \Delta \vee B) \wedge (\bigvee \Delta \vee A)$. As $\PV \vdash (\bigvee \Delta \vee B) \wedge (\bigvee \Delta \vee A) \to (\bigvee \Delta \vee (A \wedge B))$, by Lemma \ref{ImplicationToFlow}, we have $(\bigvee \Delta \vee B) \wedge (\bigvee \Delta \vee A) \rhd_k (\bigvee \Delta \vee (A \wedge B))$. Hence, by weak gluing, Lemma \ref{Glu}, $\Gamma \rhd_k (A \wedge B), \Delta$. The case $(iv)$ is similar to $(ii)$.
\end{proof}

\begin{lemma}\label{Neg}(Negation Rules) If $\Gamma \cup \Delta \subseteq \hat{\Pi}^b_k$ and $A, \neg A \in \hat{\Pi}^b_k$, then:
\begin{description}
\item[$(i)$]
If $\Gamma, A \rhd_k \Delta$ then $\Gamma \rhd_k \neg A, \Delta$.
\item[$(ii)$]
If $\Gamma \rhd_k A, \Delta$ then $\Gamma, \neg A \rhd_k \Delta$.
\end{description}
A similar claim also holds for $\rhd^p_k$.
\end{lemma}
\begin{proof}
We only prove the claim for $\rhd_k$. The case for $\rhd_k^p$ is identical. For $(i)$, assume $\Gamma, A \rhd_k \Delta$ which means $\bigwedge \Gamma \wedge A \rhd_k \bigvee \Delta$. As $\neg A \in \hat{\Pi}^b_k$, by Lemma \ref{kImplicationToFlow}, we have $(\bigwedge \Gamma \wedge A) \vee \neg A \rhd_k \bigvee \Delta \vee \neg A$. Since 
$
\PV \vdash \bigwedge \Gamma \rightarrow (\bigwedge \Gamma \wedge A) \vee \neg A$,
by Lemma \ref{kImplicationToFlow}, we have $\bigwedge \Gamma \rhd_k (\bigwedge \Gamma \wedge A) \vee \neg A$. Hence, by
weak gluing, Lemma \ref{kGlu}, we have $\bigwedge \Gamma \rhd_k \bigvee \Delta \vee \neg A$. The proof for $(ii)$ is similar.
\end{proof}

\begin{lemma}\label{quantifier}(Bounded Universal Quantifier)
Let $A(\bar{x}), B(\bar{x}, y) \in \hat{\Pi}^b_k$ and $s(\bar{x})$ be an $\mathcal{L}_{\PV}$-term. If $A(\bar{x}), (y \leq s(\bar{x})) \rhd_k B(\bar{x}, y)$, then $A(\bar{x}) \rhd_k \forall y \leq s(\bar{x}) B(y, \bar{x})$. The same also holds for $\rhd_k^p$.
\end{lemma}
\begin{proof}
Again, we only prove the claim for $\rhd_k$. The proof for $\rhd^p_k$ is identical. Since $A(\bar{x}), (y \leq s(\bar{x})) \rhd_k B(\bar{x}, y)$ and $y \leq s(\bar{x})$ is quantifier-free, by Lemma \ref{Neg}, we have $A(\bar{x}) \rhd_k (y \leq s(\bar{x}) \to B(\bar{x}, y))$. Note that $(y \leq s(\bar{x}) \to B(\bar{x}, y))$ is defined as $\neg (y \leq s(\bar{x})) \vee B(\bar{x}, y)$, as the negation is not primitive in the language. Use Lemma \ref{Bounded} for the formulas $A(\bar{x})$ and $(y \leq s(\bar{x}) \to B(\bar{x}, y))$ and the term $s(\bar{x})$. Therefore, we have a formula $I(u, y, \bar{x}) \in \hat{\Pi}^b_k$ and a term $r(\bar{x})$ such that: 
\begin{description}
\item[$\bullet$]
$\PV \vdash I(0, y, \bar{x}) \leftrightarrow A(\bar{x})$.
\item[$\bullet$]
$\PV \vdash \forall y \leq s(\bar{x}) [I(r(\bar{x}), y, \bar{x}) \leftrightarrow (y \leq s(\bar{x}) \to B(\bar{x}, y))]$.
\item[$\bullet$]
$\PV \vdash I(u, y, \bar{x}) \rightarrow I(u+1, y, \bar{x})$.
\end{description}
It is easy to see that the pair $\big (\forall y \leq s(\bar{x}) I(u, y, \bar{x}), r(\bar{x})\big)$ is a $k$-flow from $A(\bar{x})$ to $\forall y \leq s(\bar{x}) B(\bar{x}, y)$. 
\end{proof}

Now we are ready to use $k$-flows to witness the provable implications between $\hat{\Pi}^b_k$-formulas in $S^k_2$ and $T^k_2$.
\begin{theorem}\label{kMain}(Soundness and Completeness)
Let $\Gamma(\bar{x}) \cup \Delta(\bar{x}) \subseteq \hat{\Pi}^b_k$. Then:
\begin{description}
\item[$(i)$]
$S^k_2 \vdash \bigwedge \Gamma(\bar{x}) \rightarrow \bigvee \Delta(\bar{x})$ iff $ \Gamma(\bar{x}) \rhd^p_k \Delta(\bar{x})$.
\item[$(ii)$]
$T^k_2 \vdash \bigwedge \Gamma(\bar{x}) \rightarrow \bigvee \Delta(\bar{x})$ iff $ \Gamma(\bar{x}) \rhd_k \Delta(\bar{x})$.
\end{description}
\end{theorem}
\begin{proof}
We only prove $(i)$. The proof of $(ii)$ is similar. First, we prove the easier completeness part. If $\Gamma(\bar{x}) \rhd^p_{k} \Delta(\bar{x})$, then by Definition \ref{kDefFlow}, there exist a polynomial $q$, and a formula $H(u, \bar{x}) \in \hat{\Pi}^b_k$ such that:
\begin{description}
\item[$\bullet$]
$\PV \vdash H(0, \bar{x}) \leftrightarrow \bigwedge \Gamma(\bar{x})$,
\item[$\bullet$]
$\PV \vdash H(q(|\bar{x}|), \bar{x}) \leftrightarrow \bigvee \Delta(\bar{x})$,
\item[$\bullet$]
$\PV \vdash \forall u < q(|\bar{x}|) \; [H(u, \bar{x}) \rightarrow H(u+1, \bar{x})] $.
\end{description}
Using Lemma \ref{Padding}, without loss of generality, we can also assume that $\PV \vdash q(|\bar{x}|) \geq 1$. As $\PV$ is a subtheory of $S^k_2$, we also have all the above provabilities for $S^k_2$. Hence,
$
S^k_2 \vdash \forall u < q(|\bar{x}|) \; [H(u, \bar{x}) \rightarrow H(u+1, \bar{x})].
$
Since $H(u, \bar{x}) \in \hat{\Pi}^k_2$, by the $\hat{\Pi}^b_k-\LInd$ axiom, we have,
$
S^k_2 \vdash H(0, \bar{x}) \rightarrow H(|2^{q(|\bar{x}|)\dotminus 1}|, \bar{x})$. As $\PV \vdash q(|\bar{x}|) \geq 1$, we have $\PV \vdash |2^{q(|\bar{x}|)\dotminus 1}|=|q(\bar{x})|$. Hence, $S^k_2 \vdash H(0, \bar{x}) \rightarrow H(q(|\bar{x}|), \bar{x})$.
Therefore, $S^k_2 \vdash \bigwedge \Gamma(\bar{x}) \rightarrow \bigvee \Delta(\bar{x})$.\\
For soundness, assume $S^k_2 \vdash \bigwedge \Gamma(\bar{x}) \rightarrow \bigvee \Delta(\bar{x})$. By Lemma \ref{kCutConsequence}, $\Gamma(\bar{x}) \Rightarrow \Delta(\bar{x})$ has a $\mathbf{wLS^k_2}$-proof only consisting of $\hat{\Pi}^b_k$-formulas.
By induction on this proof, we show that for any sequent $\Sigma \Rightarrow \Lambda$ in the proof, we have $\Sigma \rhd^p_k \Lambda$.
For the axioms, as they are provable in $\PV$, using Lemma \ref{kImplicationToFlow}, there is nothing to prove. The case of structural rules (except for the weak cut) is easy. Weak cut and $(\wPInd_k)$ are addressed in Lemma \ref{kGlu}. The conjunction and disjunction rules are proved in Lemma \ref{kConAndDis} and the rule $(R\forall^{\leq})$ is addressed in Lemma \ref{quantifier}. Therefore, there are only three cases to consider.
If the last rule is
\begin{center}
  	\begin{tabular}{c c}
  		\AxiomC{$\Sigma(\bar{x}, y), B(\bar{x}, s(\bar{x}, y)) \Rightarrow \Lambda(\bar{x}, y) $}
  		\RightLabel{{\footnotesize $L\forall^{\leq} $}} 
  		\UnaryInfC{$\Sigma(\bar{x}, y), s(\bar{x}, y) \leq t(\bar{x}, y), \forall y \leq t(\bar{x}, y) B(\bar{x}, y) \Rightarrow \Lambda(\bar{x}, y)$}
  		\DisplayProof
	\end{tabular}
\end{center}
by the induction hypothesis, we have
$\Sigma(\bar{x}, y), B(\bar{x}, s(\bar{x}, y)) \rhd^p_k \Lambda(\bar{x}, y) $.
Since 
\[
\bigwedge \Sigma(\bar{x}, y) \wedge (s(\bar{x}, y) \leq t(\bar{x}, y)) \wedge \forall y \leq t(\bar{x}, y) B(\bar{x}, y)
\]
implies $\bigwedge \Sigma(\bar{x}, y) \wedge B(\bar{x}, s(\bar{x}, y))$ in $\PV$,
by Lemma \ref{kImplicationToFlow} and weak gluing, Lemma \ref{kGlu}, we have
\[
\Sigma(\bar{x}, y), s(\bar{x}, y) \leq t(\bar{x}, y), \forall y \leq t(\bar{x}, y) B(\bar{x}, y) \rhd^p_k \Lambda(\bar{x}, y).
\]
The case for the rule $R\exists^{\leq} $ is similar to the previous case. Finally, if the last rule is
\begin{center}
  	\begin{tabular}{c c}
  		
   	\AxiomC{$\Sigma(\bar{x}), z \leq s(\bar{x}), B(\bar{x}, z) \Rightarrow \Lambda(\bar{x})$}
  	 \RightLabel{{\footnotesize $L\exists^{\leq} $}} 
  		\UnaryInfC{$\Sigma(\bar{x}), \exists y \leq s(\bar{x}) B(\bar{x}, y) \Rightarrow\Lambda(\bar{x})$}
  		\DisplayProof

	\end{tabular}
\end{center}
by the induction hypothesis, we have $\Sigma(\bar{x}), z \leq s(\bar{x}), B(\bar{x}, z) \rhd^p_k \Lambda(\bar{x})$. Since $\exists y \leq s(\bar{x}) B(\bar{x}, y) \in \hat{\Pi}^b_k$ and it starts with an existential quantifier, it must belong to $\hat{\Sigma}^b_{k-1}$. Hence, both $\neg B(\bar{x}, z)$ and $\neg \exists y \leq s(\bar{x}) B(\bar{x}, y)=\forall y \leq s(\bar{x}) \neg B(\bar{x}, y)$ are in $\hat{\Pi}^b_k$. Therefore, by Lemma \ref{Neg},
\[
\Sigma(\bar{x}), z \leq s(\bar{x}) \rhd^p_k \Lambda(\bar{x}), \neg B(\bar{x}, z).
\]
By using the fact that the names of the parameters are not important in $k$-flows and employing Lemma \ref{quantifier}, we have
$
\Sigma(\bar{x}) \rhd^p_k \Lambda(\bar{x}), \forall y \leq s(\bar{x}) \; \neg B(\bar{x}, y)$.
Finally again by Lemma \ref{Neg}, we reach
$
\Sigma(\bar{x}), \exists y \leq s(\bar{x}) B(\bar{x}, y) \rhd^p_k \Lambda(\bar{x})$.
\end{proof}

\subsection{Reductions and $\PLS_{(k, l)}$-programs} \label{SubsectionReduction}

In Subsection \ref{SubsectionKFlows}, we transformed the $S^k_2$-provable (resp. $T^k_2$-provable) implications between $\hat{\Pi}^b_k$-formulas into exponentially (resp., polynomially) long uniform sequences of $\PV$-provable implications between $\hat{\Pi}^b_k$-formulas. Having that characterization at hand, one can use the universality of $\PV$ to employ generalized Herbrand's theorem and push the characterization of Theorem \ref{kMain} even further to witness \emph{all} essentially existential quantifiers in the $\PV$-provable implications by polynomial-time computable functions. 
Instead of following this rather \emph{absolute} approach, in this subsection, we will employ a \emph{relative} approach to witness all the essentially existential quantifiers up to a given level $l \leq k$. The idea is simple. First, by moving the $\PV$-provable implications from $\PV$ to $\PV_{k-l+1}$, we will pretend that all the $\mathcal{L}_{\PV}$-formulas in $\hat{\Pi}^b_{k-l} \cup \hat{\Sigma}^b_{k-l}$ are quantifier-free in $\mathcal{L}_{\PV_{k-l+1}}$. Therefore, only $l$ many alternating quantifiers are left to peel off for which we use the generalized Herbrand's theorem. In choosing the right value for $l$, there is a clear trade-off between the complexity of the witnessing functions on the one hand and the complexity of the witnessing more alternating quantifiers, on the other. For the smaller values of $l$, the latter would be quite easy as evidenced by Theorem \ref{Herbrand}. However, the cost to pay is the higher complexity of the witnessing functions that now live in the higher level of the polynomial hierarchy, i.e., in the class $\Box^p_{k-l+1}$. For the higher values of $l$, the situation is reverse. For instance, if $l=k$, then all the witnessing functions are polynomial time as they live in $\PV_{k-k+1}=\PV$. However, the generalized Herbrand's theorem must witness $k$ many quantifier alternations that is combinatorially too complex to deal with.
In the present subsection, we will lean towards the lower values for $l$ and will only apply the relative approach to two instances of $l=1$ and $l=2$ to avoid the high witnessing complexity. However, it is worth emphasizing that the main base, i.e., Theorem \ref{kMain} is there and one can use it for any value of $l$ by employing the right instance of Herbrand's theorem. We only cover these two cases to show that how interesting the concrete consequences can be. For $l=1$, we will show that some well-known witnessing theorems in bounded arithmetic are just special cases of our witnessing theorem. For $l=2$, the witnessing results are all new. 

\subsubsection{The game interpretation}

%To have a better computational interpretation of witnessing the $\PV_k$-provable implications, it is useful to recall the game-theoretic reading of bounded formulas. 
Let $k \geq l \geq 1$ be two numbers, $G(\bar{x}, y_1, y_2, \ldots, y_l)$ be a quantifier-free $\mathcal{L}_{\PV_k}$-formula and $t(\bar{x})$ be an $\mathcal{L}_{\PV_k}$-term. We call the pair $(G(\bar{x}, y_1, y_2, \ldots, y_l), t(\bar{x}))$ a \emph{$(k, l)$-game} (a game, for short) and we interpret it as a uniform family of $l$-turn games between two players parameterized by the variables $\bar{x}$. To emphasize this parameter role, we sometimes write $G_{\bar{x}}(y_1, y_2, \ldots, y_l)$ for $G(\bar{x}, y_1, y_2, \ldots, y_l)$ and if the variables are clear from the context, we use the shorthand $(G_{\bar{x}}, t(\bar{x}))$ for an instance of the game and $(G, t)$ for the uniform family itself. Given the value $\bar{a}$ for $\bar{x}$, the game $G_{\bar{a}}$ starts with the first player, denoted by I, playing the number $b_1 \leq t(\bar{a})$ for $y_1$. Then, the second player, denoted by II, plays $b_2 \leq t(\bar{a})$ for $y_2$ and so on. The resulting tuple $(b_1, b_2, \ldots, b_l)$ is called a \emph{play} of the game. For a play, $(b_1, b_2, \ldots, b_l)$, if $G(\bar{a}, b_1, b_2, \ldots, b_l)$ holds, the \emph{first player} wins the game and otherwise, the second player is the winner. A play $(b_1, \ldots, b_l)$ is called a \emph{winning play} for the first (second) player, if it makes the first (second) player wins. It is an easy and well-known fact that the first player has a \emph{winning strategy} in $G_{\bar{a}}$ iff $\exists y_1 \leq t(\bar{a}) \forall y_2 \leq t(\bar{a}) \exists y_3 \leq t(\bar{a}) \ldots G(\bar{a}, y_1, \ldots, y_l)$ holds. As we are always interested in the first player in this subsection, by a winning play and a winning strategy, we always mean them for the first player. Having two $(k, l)$-games $(G(\bar{x}, y_1, y_2, \ldots, y_l), t(\bar{x}))$ and $(H(\bar{x}, z_1, z_2, \ldots, z_l), s(\bar{x}))$, a natural question to ask is the following. Let the existence of a winning strategy in $(G_{\bar{x}}, t(\bar{x}))$ implies the existence of a winning strategy in $(H_{\bar{x}}, s(\bar{x}))$, for any $\bar{x}$, i.e., the implication
\[
\exists y_1 \leq t(\bar{x}) \forall y_2 \leq t(\bar{x}) \ldots G(\bar{x}, y_1, \ldots, y_l) \to 
\]
\[
\exists z_1 \leq s(\bar{x}) \forall z_2 \leq s(\bar{x}) \ldots H(\bar{x}, z_1, \ldots, z_l), \quad \quad (\dagger)
\]
hold. Then, does it mean that we can find an \emph{explicit} way to use a winning strategy for $(G_{\bar{x}}, t(\bar{x}))$ to design a winning strategy for $(H_{\bar{x}}, s(\bar{x}))$? One can even sharpen the question by asking if having a proof of the implication $(\dagger)$
in the theory $\PV_k$ helps to provide an explicit and \emph{relatively simple} transformation between the winning strategies. Fortunately, as $\PV_k$ is a universal theory, the extraction of the explicit transformation between the winning strategies is possible and it is simply the content of Herbrand's theorem, Theorem \ref{Herbrand} (up to some small modifications). We will explain the details for the two case $l=1$ and $l=2$, below.

\subsubsection{The case $l=1$}
%A $(k,1)$-game has the form $(G(\bar{x}, y), t(\bar{x}))$, where the first player plays $y \leq t(\bar{x})$ in the game's one and only round. The first player has a winning strategy in this game iff $\exists y \leq t(\bar{x}) G(\bar{x}, y)$ holds. 
Let $(G(\bar{x}, y), t(\bar{x}))$ and $(H(\bar{x}, z), s(\bar{x}))$ be two $(k, 1)$-games.
The most trivial way to reduce the winning strategy of the latter to that of the former is via a function $f(\bar{x}, y)$ that maps any move 
$y \leq t(\bar{x})$ in $(G, t)$ to a move $z \leq s(\bar{x})$ in $(H, s)$ such that if the play $y$ is a winning play in $(G, t)$, then the play $z=f(\bar{x}, y)$ is a winning play in $(H, s)$. Moreover, as we expect the reduction to be simple and verifiable, we expect that everything happens inside a base theory, in our case $\PV_k$. More formally:
\begin{definition}\label{1Reduction}
Let $(G(\bar{x}, y), t(\bar{x}))$ and $(H(\bar{x}, z), s(\bar{x}))$ be two $(k, 1)$-games. A \emph{$(k, 1)$-reduction} from $(H(\bar{x}, z), s(\bar{x}))$ to $(G(\bar{x}, y), t(\bar{x}))$ is an $\mathcal{L}_{\PV_k}$-term $f(\bar{x}, y)$ such that:
\begin{itemize}
    \item[$\bullet$]
$\PV_k \vdash \forall y \leq t(\bar{x}) [f(\bar{x}, y) \leq s(\bar{x})]$,
    \item[$\bullet$]
$\PV_k \vdash \forall y \leq t(\bar{x}) [G(\bar{x}, y) \to H(\bar{x}, f(\bar{x}, y))]
$. 
\end{itemize}
\end{definition}
Naturally, we expect a connection between the provability of 
\[
\exists y \leq t(\bar{x}) G(\bar{x}, y) \to \exists z \leq s(\bar{x}) H(\bar{x}, z)
\]
in $\PV_k$ and the existence of a $(k, 1)$-reduction. This is the content of the following modification of Herbrand's theorem.
\begin{theorem}\label{1ImpToRed}
For any two $(k, 1)$-games $(G(\bar{x}, y), t(\bar{x}))$ and $(H(\bar{x}, z), s(\bar{x}))$, the following are equivalent:
\begin{itemize}
    \item[$\bullet$]
$\PV_k \vdash \exists y \leq t(\bar{x}) G(\bar{x}, y) \to \exists z \leq s(\bar{x}) H(\bar{x}, z)$
    \item[$\bullet$]
There is a $(k, 1)$-reduction from  $(H(\bar{x}, z), s(\bar{x}))$ to $(G(\bar{x}, y), t(\bar{x}))$.
\end{itemize}
\end{theorem}
\begin{proof}
One direction is trivial. For the other, assume 
\[
\PV_k \vdash \exists y \leq t(\bar{x}) G(\bar{x}, y) \to \exists z \leq s(\bar{x}) H(\bar{x}, z).
\]
Therefore, $\PV_k \vdash \forall y \exists z\, [(y \leq t(\bar{x}) \wedge G(\bar{x}, y)) \to (z \leq s(\bar{x}) \wedge H(\bar{x}, z))]$. By Herbrand's theorem, Theorem \ref{Herbrand}, there exists an $\mathcal{L}_{\PV_k}$-term $g(\bar{x}, y)$ such that
\[
\PV_k \vdash  [y \leq t(\bar{x}) \wedge G(\bar{x}, y)] \to [g(\bar{x}, y) \leq s(\bar{x}) \wedge H(\bar{x}, g(\bar{x}, y))].
\]
Define
\[
f(\bar{x}, y)=
\begin{cases}
g(\bar{x}, y) &  g(\bar{x}, y) \leq s(\bar{x})\\
0 & g(\bar{x}, y) > s(\bar{x})
\end{cases} 
\]
It is easy to represent $f(\bar{x}, y)$ as an $\mathcal{L}_{\PV_k}$-term. By definition, it is clear that  $\PV_k \vdash \forall y \leq t(\bar{x})  [f(\bar{x}, y) \leq s(\bar{x})]$. Moreover, it is easy to see that $\PV_k \vdash \forall y \leq t(\bar{x}) [G(\bar{x}, y) \to H(\bar{x}, f(\bar{x}, y))]$.
\end{proof}

As explained in the opening of this subsection, for $l=1$, the combination of witnessing by $k$-flows, moving from $\PV$ to $\PV_k$ and using Theorem \ref{1ImpToRed} provides an explicit witnessing theorem for theories $S^k_2$ and $T^k_2$. This is what we will come back to in Corollary \ref{k1PLSWitnessing}. However, as the combination has a natural form itself, it is worth defining it directly.

\begin{definition}\label{k1PLS}
Let $A(\bar{x}, y) \in \hat{\Pi}^b_{k-1}$ be an $\mathcal{L}_{\PV}$-formula and $t(\bar{x})$ and $r(\bar{x})$ be two $\mathcal{L}_{\PV}$-terms. By a \emph{$\PLS_{(k, 1)}$-program} for $(A(\bar{x}, y), r(\bar{x}))$ with the length $t(\bar{x})$, we mean the following data:
an initial $\mathcal{L}_{\PV_k}$-term $i(\bar{x})$,
a quantifier-free $\mathcal{L}_{\PV_k}$-formula $G(\bar{x}, u, z)$,
an $\mathcal{L}_{\PV_k}$-term $N(\bar{x}, u, z)$ and 
an $\mathcal{L}_{\PV_k}$-term $p(\bar{x}, z)$, such that:
\begin{description}
\item[$\bullet$]
$\PV_k \vdash i(\bar{x}) \leq s(\bar{x})$,
\item[$\bullet$]
$\PV_k \vdash G(\bar{x}, 0,  i(\bar{x}))$,
\item[$\bullet$]
$\PV_k \vdash \forall z \leq s(\bar{x}) [N(\bar{x}, u, z) \leq s(\bar{x})]$,
\item[$\bullet$]
$\PV_k \vdash \forall z \leq s(\bar{x}) \, [G(\bar{x}, u, z) \to G(\bar{x}, u+1, N(\bar{x}, u, z))]$,
\item[$\bullet$]
$\PV_k \vdash \forall z \leq 
 s(\bar{x}) [p(\bar{x}, z) \leq r(\bar{x})]$,
\item[$\bullet$]
$\PV_k \vdash \forall z \leq 
 s(\bar{x})  [G(\bar{x}, t(\bar{x}), z) \to A(\bar{x}, p(\bar{x}, z))] $.
\end{description}
By $\PLS_{(k, 1)}$, we mean the class of all the pairs $(A(\bar{x}, y), r(\bar{x}))$ for which there exists a $\PLS_{(k, 1)}$-program. By $\PLS_{(k, 1)}^p$, we mean the class of all the pairs $(A(\bar{x}, y), r(\bar{x}))$ for which there exists a $\PLS_{(k, 1)}$-program with a polynomial length, i.e., $t(\bar{x})=q(|\bar{x}|)$, for some polynomial $q$.
\end{definition}

It is easy to see that if $(A(\bar{x}, y), r(\bar{x})) \in \PLS_{(k, 1)}$ then $\forall \bar{x} \exists y \leq r(\bar{x}) A(\bar{x}, y)$ holds and the $\PLS_{(k, 1)}$-program actually provides an algorithm to compute $y \leq r(\bar{x})$ from $\bar{x}$. Denoting $G(\bar{x}, u, z)$ by $G^u$, the algorithm starts at the zeroth level with an initial value $i(\bar{x})$ bounded by $s(\bar{x})$ satisfying the property $G^0$. Then, using the modification $N$, it goes from one level to the next updating any value $z \leq s(\bar{x})$ with the property $G^u$ to a value satisfying the property $G^{u+1}$. Note that the modification always respects the bound $s(\bar{x})$. Finally, reaching the level $t(\bar{x})$, the algorithm uses $p$ to compute $y \leq r(\bar{x})$ satisfying $A$ from any value $z \leq s(\bar{x})$ with the property $G^{t(\bar{x})}$. 

There are two points to emphasize here. First, the case $k=1$, where the predicate $G(\bar{x}, y, z)$ and all the functions $i(\bar{x})$, $N(\bar{x}, u, z)$ and $p(\bar{x}, z)$ are polynomial time computable is just another presentation of the well-known polynomial local search problems, ($\PLS$ for short), see \cite{PLS,JanBook}. Therefore, one can simply read $\PLS_{(k, 1)}$-programs as a generalization of $\PLS$ from polynomial time to the $k$-th level of the polynomial hierarchy, where the predicate $G(\bar{x}, u, z)$ and the functions $i(\bar{x})$, $N(\bar{x}, u, z)$ and $p(\bar{x}, z)$ are all allowed to be on the $k$-th level of the hierarchy. It is also worth mentioning that our $\PLS_{(k, 1)}$-programs are similar to but weaker than $\Pi^b_k-\PLS$ problems with $\Pi^b_l$-goals defined in \cite{BBPLS}, where the functions $i(\bar{x})$, $N(\bar{x}, u, z)$ and $p(\bar{x}, z)$ (and not the predicate $G(\bar{x}, u, z)$) must be polynomial-time computable and everything must be provable in $S^1_2$ rather than in $\PV_k$. The second point is about the $\PLS^p_{(k, 1)}$-programs with a polynomial length. For these programs, the algorithm we just provided can efficiently (relative to the level of the polynomial hierarchy, of course) compute the value of $y$ as it only needs to iterate the modification function $N$ for polynomially many times. In other words, we can pack the whole algorithm in one single $\mathcal{L}_{\PV_k}$-term as a formalized version of a $\Box^p_k$-function that computes $y$. We will come back to this observation in Corollary \ref{FunctionsOfS}, where we reprove a well-known witnessing theorem for $S^k_2$, characterizing the $\hat{\Sigma}^b_k$-definable functions of $S^k_2$ as the ones in the $k$-th level of the polynomial hierarchy.

\begin{remark}\label{RemarkOn1Reduction}
Employing the game interpretation we explained before, a $\PLS_{(k, 1)}$ program for $(A(\bar{x}, y), r(\bar{x}))$ with the length $t(\bar{x})$ is nothing but the following three $(k, 1)$-reductions:
\begin{itemize}
\item[$\bullet$]
$i(\bar{x})$ as a $(k, 1)$-reduction from $(G(\bar{x}, 0, z), s(\bar{x}))$ to $(\top,  s(\bar{x}))$.
\item[$\bullet$]
$N(\bar{x}, u, z)$ as a $(k, 1)$-reduction from the game $(G(\bar{x}, u+1, z),
s(\bar{x}))$ to the game $(G(\bar{x}, u, z), s(\bar{x}))$. Notice that $u$ is also a parameter here.
\item[$\bullet$]
$p(\bar{x}, z)$ as a $(k, 1)$-reduction from the game $(A(\bar{x}, y), r(\bar{x}))$ to the game $(G(\bar{x}, t(\bar{x}), z), s(\bar{x}))$. 
\end{itemize}
Notice that the formula $A(\bar{x}, y)$ is not quantifier-free in $\mathcal{L}_{\PV_k}$ and hence we cannot read the pair $(A(\bar{x}, y), r(\bar{x}))$ as a $(k, 1)$-game. However, as $A(\bar{x}, y)$ is in $\hat{\Pi}^b_{k-1}$, it is $\PV_k$-equivalent to a quantifier-free formula and hence we can pretend that it is quantifier-free. 
Having that observation, we can use Theorem \ref{1ImpToRed} to see that there is a $\PLS_{(k, 1)}$ program for $(A(\bar{x}, y), r(\bar{x}))$ with the length $t(\bar{x})$ iff there exist a quantifier-free $\mathcal{L}_{\PV_k}$-formula $G(\bar{x}, u, z)$ and an $\mathcal{L}_{\PV}$-term $s(\bar{x})$ such that:
\begin{itemize}
\item[$\bullet$]
$\PV_{k} \vdash \exists z \leq s(\bar{x}) G(\bar{x}, 0, z)$.
\item[$\bullet$]
$\PV_{k} \vdash \exists z \leq s(\bar{x}) G(\bar{x}, u, z) \rightarrow \exists z \leq s(\bar{x}) G(\bar{x}, u+1, z)$.
\item[$\bullet$]
$\PV_{k} \vdash  \exists z \leq s(\bar{x}) G(\bar{x}, t(\bar{x}), z) \to \exists y \leq r(\bar{x}) A(\bar{x}, y)$.
\end{itemize}
\end{remark}

The next Corollary uses $\PLS_{(k, 1)}$-programs (resp. $\PLS^p_{(k, 1)}$-programs) to witness the theorems of $S^k_2$ (resp. $T^k_2$) as promised before. 

\begin{corollary}\label{k1PLSWitnessing}
Let $k \geq 1$, $A(\bar{x}, y) \in \hat{\Pi}^b_{k-1}$ and $r(\bar{x})$ be an $\mathcal{L}_{\PV}$-term:
\begin{description}
\item[$(i)$]
$S^{k}_2 \vdash \forall \bar{x} \exists y \leq r(\bar{x}) A(\bar{x}, y)$ iff $(A(\bar{x}, y), r(\bar{x})) \in \PLS^p_{(k, 1)}$. 
\item[$(ii)$]
$T^{k}_2 \vdash \forall \bar{x} \exists y \leq r(\bar{x}) A(\bar{x}, y)$ iff $(A(\bar{x}, y), r(\bar{x})) \in \PLS_{(k, 1)}$.
\end{description}
\end{corollary}
\begin{proof}
We only prove $(i)$. The proof of $(ii)$ is similar. For the right to left direction, if there exists a $\PLS_{(k, 1)}$-program for $(A(\bar{x}, y), r(\bar{x}))$ with the length $q(|\bar{x}|)$, for some polynomial $q$, using Remark \ref{RemarkOn1Reduction}, there are quantifier-free $\mathcal{L}_{\PV_k}$-formula $G(\bar{x}, u, z)$ and an $\mathcal{L}_{\PV}$-term $s(\bar{x})$ such that:
\begin{itemize}
\item[$\bullet$]
$\PV_{k} \vdash \exists z \leq s(\bar{x}) G(\bar{x}, 0, z)$.
\item[$\bullet$]
$\PV_{k} \vdash \exists z \leq s(\bar{x}) G(\bar{x}, u, z) \rightarrow \exists z \leq s(\bar{x}) G(\bar{x}, u+1, z)$.
\item[$\bullet$]
$\PV_{k} \vdash  \exists z \leq s(\bar{x}) G(\bar{x}, q(|\bar{x}|), z) \to \exists y \leq r(\bar{x}) A(\bar{x}, y)$.
\end{itemize}
Since $\PV_k$ is interpretable in $S^k_2$, mapping all quantifier-free $\mathcal{L}_{\PV_k}$-formulas to $\mathcal{L}_{\PV}$-formulas in $\hat{\Sigma}^b_k$, we can pretend that $G(\bar{x}, u, z) \in \hat{\Sigma}^b_k$ and all the above formulas are also provable in $S^k_2$.
Finally, since the theory $S^k_2$ has the axiom $\hat{\Sigma}^b_k-\LInd$ and $\exists z \leq s(\bar{x}) G(\bar{x}, u, z) \in \hat{\Sigma}^b_k$, we have $S^k_2 \vdash \exists y \leq r(\bar{x}) A(\bar{x}, y)$. 
For the other direction, assume $S^{k}_2 \vdash \forall \bar{x} \exists y \leq r(\bar{x}) A(\bar{x}, y)$. Hence, $S^{k}_2 \vdash \forall y \leq r(\bar{x}) \; \neg A(\bar{x}, y) \rightarrow \bot$. By Theorem \ref{kMain}, $\forall y \leq r(\bar{x}) \; \neg A(\bar{x}, y) \rhd^p_k \bot$.
Therefore,
there exist a polynomial $q$ and a formula $H(u, \bar{x}) \in \hat{\Pi}^b_{k}$ such that:
\begin{itemize}
\item[$\bullet$]
$\PV \vdash H(0, \bar{x}) \leftrightarrow [\forall y \leq r(\bar{x}) \; \neg A(\bar{x}, \bar{y})]$.
\item[$\bullet$]
$\PV \vdash H(q(|\bar{x}|), \bar{x}) \leftrightarrow \bot$.
\item[$\bullet$]
$\PV \vdash \forall u < q(|\bar{x}|) \; [H(u, \bar{x}) \rightarrow H(u+1, \bar{x})]$.
\end{itemize}
Define $H'(u, \bar{x})$ as $[(u \leq q(|\bar{x}|)) \to H(u, \bar{x})]$. It is easy to see that
\begin{itemize}
\item[$\bullet$]
$\PV \vdash [\forall y \leq r(\bar{x}) \; \neg A(\bar{x}, \bar{y})] \to  H'(0, \bar{x})$.
\item[$\bullet$]
$\PV \vdash H'(q(|\bar{x}|), \bar{x}) \to \bot$.
\item[$\bullet$]
$\PV \vdash H'(u, \bar{x}) \rightarrow H'(u+1, \bar{x})$.
\end{itemize}
As $\PV$ has the pairing function, it can encode finite many bounded variables as one bounded variable. Hence, without loss of generality, we can assume that $H'$ is in the prenex bounded form starting with one universal quantifier on $z$, i.e., $H'(u, \bar{x})=\forall z \leq s'(\bar{x}, u) \; I(\bar{x}, u, z)$, where $s'$ is $\PV$-monotone and $I \in \hat{\Sigma}^b_{k-1}$. Define $s(\bar{x})$ as $s'(q(|\bar{x}|), \bar{x})$. Then, it is easy to see that
\[
\PV \vdash H'(u, \bar{x}) \leftrightarrow [\forall z \leq s(\bar{x}) [(z \leq s'(u, \bar{x})) \wedge (u \leq q(|\bar{x}|)) \to I(\bar{x}, u, z)]].
\]
Hence, without loss of generality, we can assume that $H'$ is in the form $\forall z \leq s(\bar{x}) J(\bar{x}, u, z)$, where 
$J$ is in $\hat{\Sigma}^b_{k-1}$. Since $\PV$ is a subtheory of $\PV_{k}$ and in $\PV_{k}$ any formula in $\hat{\Sigma}^b_{k-1}$ is equivalent to a quantifier-free formula, we can assume that $J$ is quantifier-free in the language of $\PV_{k}$ and we have 
\begin{itemize}
\item[$\bullet$]
$\PV_{k} \vdash  [\forall y \leq r(\bar{x}) \; \neg A(\bar{x}, y)] \to \forall z \leq s(\bar{x}) J(\bar{x}, 0, z)$.
\item[$\bullet$]
$\PV_{k} \vdash \forall z \leq s(\bar{x}) J(\bar{x}, q(|\bar{x}|), z) \to \bot$.
\item[$\bullet$]
$\PV_{k} \vdash \forall z \leq s(\bar{x}) J(\bar{x}, u, z) \rightarrow \forall z \leq s(\bar{x}) J(\bar{x}, u+1, z)$.
\end{itemize}
Define $G(\bar{x}, u, z)$ as $\neg J(\bar{x}, q(|\bar{x}|) \dotminus u, z)$ and note that it is a quantifier-free $\mathcal{L}_{\PV_k}$-formula. Therefore, we have
\begin{itemize}
\item[$\bullet$]
$\PV_{k} \vdash \exists z \leq s(\bar{x}) G(\bar{x}, q(|\bar{x}|), z) \to  \exists y \leq r(\bar{x}) A(\bar{x}, y) $.
\item[$\bullet$]
$\PV_{k} \vdash \exists z \leq s(\bar{x}) G(\bar{x}, 0, z)$.
\item[$\bullet$]
$\PV_{k} \vdash \exists z \leq s(\bar{x}) G(\bar{x}, u, z) \rightarrow \exists z \leq s(\bar{x}) G(\bar{x}, u+1, z)$.
\end{itemize}
Finally, it is enough to use Remark \ref{RemarkOn1Reduction} to get a $\PLS_{(k, 1)}$-program for the pair $(A(\bar{x}, y), r(\bar{x}))$ with the length $q(|\bar{x}|)$. Hence, $(A(\bar{x}, y), r(\bar{x})) \in \PLS^p_{(k, 1)}$.
\end{proof}

Note that the second part in Corollary \ref{k1PLSWitnessing}, when applied on $k=1$, reproves the well-known characterization of the $T^1_2$-provable formulas of the form $\forall \bar{x} \exists y \leq r(\bar{x}) A(\bar{x}, y)$, where $A(\bar{x}, y) \in \hat{\Pi}^b_0$, in terms of the usual $\PLS$ problems \cite{PLS,JanBook}. Our result, however, seems a bit weaker than the one proved in \cite{PLS,JanBook}, as in the latter $y$ is not assumed to be bounded and $A(\bar{x}, y)$ can be in $\hat{\Sigma}^b_1$ rather than in our lower class of $\hat{\Pi}^b_0$. However, proving the stronger form from the one we provided is just a standard technique. First, notice that the presence of $r(\bar{x})$ is no restriction, thanks to Parikh theorem. Secondly, to reduce the complexity of $A(\bar{x}, y)$, it is enough to write $A(\bar{x}, y)$ in the form $\exists \bar{z} \leq \bar{x}(\bar{x}) B(\bar{x}, y, \bar{z})$, where $B(\bar{x}, y, \bar{z}) \in \hat{\Pi}^b_0$. Then, using the pairing function available in $\PV$, we can make $y$ and all the variables $\bar{z}$ into one bounded variable $w \leq t(\bar{x})$. Now, we can apply Corollary \ref{k1PLSWitnessing} to compute $w$ by a $\PLS_{(k, 1)}$-program. With this technique, the $\PLS_{(k, 1)}$-program not only computes the intended variable $y$, but it also finds a value for the variables $\bar{z}$. To retrieve our formula $A(\bar{x}, y)$, we can simply keep the computation for $y$ and forget the other values for $\bar{z}$ by reintroducing their existential quantifiers. 
Having this observation about the usual $\PLS$, one can read Corollary \ref{k1PLSWitnessing} as a generalization of the mentioned characterization for $T^1_2$ to cover both $T^k_2$ and $S^k_2$, for any $k \geq 1$. However,  the latter case can be strengthened even further as the polynomial $\PLS_{(k, 1)}$-program provided in Corollary \ref{k1PLSWitnessing} can be simplified to one single $\mathcal{L}_{\PV_k}$-term. This reproves the following well-known witnessing theorem for $S^k_2$ \cite{BussThesis,JanBook}.  
\begin{corollary}\label{FunctionsOfS}
The provably $\hat{\Sigma}^b_k$-definable functions of $S_2^k$ are in $\Box_{k}^p$. Even better, if $S^k_2 \vdash \forall \bar{x} \exists y A(\bar{x}, y)$, where $A(\bar{x}, y) \in \hat{\Sigma}^b_k$, then there exists a function $f \in \Box_k^p$ represented as an $\mathcal{L}_{\PV_k}$-term such that $\PV_k \vdash \forall \bar{x} \, A(\bar{x}, f(\bar{x}))$.
\end{corollary}
\begin{proof}
Following the technique we just described, without loss of generality, we can assume that $A(\bar{x}, y)$ has no existential quantifier in its front and hence it is actually in $\hat{\Pi}^{b}_{k-1}$. By Parikh theorem, there exists an $\mathcal{L}_{\PV}$-term $r(\bar{x})$ such that
$S^k_2 \vdash \forall \bar{x} \exists y \leq r(\bar{x}) A(\bar{x}, y)$.
By Corollary \ref{k1PLSWitnessing}, there exists a $\PLS_{(k, 1)}$-program for $(A(\bar{x}, y), r(\bar{x}))$ with the length $q(|\bar{x}|)$, for a polynomial $q$. Let $G(\bar{x}, u, z)$, $i(\bar{x})$, $N(\bar{x}, u, z)$ and $p(\bar{x}, z)$ be the data of the $\PLS_{(k, 1)}$-program. By recursion on notation on $w$, define the function $M(w, \bar{x})$ as
\[
\begin{cases}
M(0, \bar{x})=i(\bar{x}) &  \\
M(w, \bar{x})=N(\bar{x}, |w| \dotminus 1, M(\lfloor \frac{w}{2} \rfloor, \bar{x})) & w > 0\\
\end{cases} 
\]
Recall that the $\mathcal{L}_{\PV_k}$-terms are closed under bounded recursion on notation.
As both $i$ and $N$ are $\mathcal{L}_{\PV_k}$-terms and $i(\bar{x})$ is bounded by $s(\bar{x})$ and $N(\bar{x}, u, z)$ maps any $z \leq s(\bar{x})$ to something below $s(\bar{x})$, we can make sure that the function $M(w, \bar{x})$ is also representable as an $\mathcal{L}_{\PV_k}$-term. Now, define $I(\bar{x}, w, z)$ as $G(\bar{x}, |w|, z)$. Using the properties of the $\PLS_{(k, 1)}$-program, it is clear that
\begin{description}
\item[$\bullet$]
$\PV_k \vdash I(\bar{x}, 0,  i(\bar{x}))$,
\item[$\bullet$]
$\PV_k \vdash \forall z \leq s(\bar{x}) \forall w > 0 \, [I(\bar{x}, \lfloor \frac{w}{2} \rfloor, z) \to I(\bar{x}, w, N(\bar{x}, |w| \dotminus 1, z))]$.
\end{description}
Therefore, as $\PV_k \vdash M(w, \bar{x}) \leq s(\bar{x})$, by using the axiom $\PInd$ on the quantifier-free formula $I(\bar{x}, w, M(w, \bar{x}))$, we can prove
$\PV_k \vdash I(\bar{x}, w, M(w, \bar{x}))$. Substituting $w=\lfloor \frac{2^{q(|\bar{x}|)}}{2} \rfloor$, we reach
\[
\PV_k \vdash I(\bar{x}, |\lfloor \frac{2^{q(|\bar{x}|)}}{2} \rfloor|, M(\lfloor \frac{2^{q(|\bar{x}|)}}{2} \rfloor, \bar{x})).
\]
Using the fact that $\PV_k \vdash |\lfloor \frac{2^{q(|\bar{x}|)}}{2} \rfloor|=q(|\bar{x}|)$, we have
\[
\PV_k \vdash G(\bar{x}, q(|\bar{x}|), M(\lfloor \frac{2^{q(|\bar{x}|)}}{2} \rfloor, \bar{x})).
\]
Therefore, as $p(\bar{x}, z)$ is an $\mathcal{L}_{\PV_k}$-term
and it has the property
\[
\PV_k \vdash \forall z \leq s(\bar{x})[G(\bar{x}, q(|\bar{x}|), z) \to A(\bar{x}, p(\bar{z}, z))],
\]
we can define $f(\bar{x})=p(\bar{x}, M(\lfloor \frac{2^{q(|\bar{x}|)}}{2} \rfloor, \bar{x}))$ as an $\mathcal{L}_{\PV_k}$-term. Therefore, $\PV_k \vdash \forall \bar{x} A(\bar{x}, f(\bar{x}))$. 
\end{proof}

\subsubsection{The case $l=2$}
%Let $k \geq 2$ be a fixed number. A $(k, 2)$-game is any pair in the form $(G(\bar{x}, y, z), t(\bar{x}))$. The game starts with the first player playing $y \leq t(\bar{x})$ and proceeds by the second player playing $z \leq t(\bar{x})$. The first player has a winning strategy iff $\exists y \leq t(\bar{x}) \forall z \leq t(\bar{x}) G(\bar{x}, y, z)$ holds.
Let $(G(\bar{x}, y, z), t(\bar{x}))$ and $(H(\bar{x}, v, w), s(\bar{x}))$ be two $(k, 2)$-games. To use a winning strategy for $(G(\bar{x}, y, z), t(\bar{x}))$ to design one for $(H(\bar{x}, v, w), s(\bar{x}))$, the most trivial way is using two functions $f(\bar{x}, y)$ and $g(\bar{x}, y, w)$, where $f(\bar{x}, y)$ reads the first move $y \leq t(\bar{x})$ in $(G, t)$ and computes a first move $v \leq s(\bar{x})$ in $(H, s)$. Then, $g(\bar{x}, y, w)$ reads the second move $w \leq s(\bar{x})$ in $(H, s)$ and computes a second move $z \leq t(\bar{x})$ in $(G, t)$. These computations must be in a way that if the play $(y, z)$ is winning in $(G, t)$, the play $(v, w)$ is winning in $(H, t)$. Expecting the whole reduction process to be simple relative to $\PV_k$, we have:
\begin{definition}
Let $(G(\bar{x}, y, z), t(\bar{x}))$ and $(H(\bar{x}, v, w), s(\bar{x}))$ be two $(k, 2)$-games. By a \emph{deterministic $(k, 2)$-reduction} from $(H, s)$ to $(G, t)$, we mean two $\mathcal{L}_{\PV_k}$-terms $f(\bar{x}, y)$ and $g(\bar{x}, y, w)$ such that:
\begin{itemize}
    \item[$\bullet$]
     $\PV_k \vdash \forall y \leq t(\bar{x}) [f(\bar{x}, y) \leq s(\bar{x})]$.
      \item[$\bullet$]
     $\PV_k \vdash \forall w \leq s(\bar{x}) \forall y \leq t(\bar{x}) [g(\bar{x}, y, w) \leq t(\bar{x})]$.
     \item[$\bullet$]
     $\PV_k \vdash \forall w \leq s(\bar{x}) \forall y \leq t(\bar{x}) [G(\bar{x}, y, g(\bar{x}, y, w)) \to H(\bar{x}, f(\bar{x}, y), w)]$.
\end{itemize}
\end{definition}
In a similar fashion to what we had
n the previous subsubsection, we expect an equivalence between the provability of 
\[
\exists y \leq t(\bar{x}) \forall z \leq t(\bar{x}) G(\bar{x}, y, z) \to 
\exists v \leq s(\bar{x}) \forall w \leq s(\bar{x}) H(\bar{x}, v, w)
\]
in $\PV_k$ and the existence of a deterministic $(k, 2)$-reduction from $(H, s)$ to $(G, t)$. Unfortunately, this expected equivalence does not exist, unless a hardness conjecture in complexity theory fails. Let us first explain this conjecture.\\

Let $U, V \subseteq \mathbb{N}$ be two disjoint $\mathbf{NP}$-sets. We call a polynomial time computable $S \subseteq \mathbb{N}$ a \emph{separator} for $U$ and $V$, if $U \subseteq S$ and $S \cap V=\varnothing$. The hardness conjecture we want to use states that there are two disjoint $\mathbf{NP}$-sets $U$ and $V$ that have no separator. 

\begin{example} \label{Impossibility}
Let $U$ and $V$ be two disjoint $\mathbf{NP}$-sets that have no separator and represent them by the $\mathcal{L}_{\PV}$-formulas $\exists y \leq s_B(x) B(x, y)$ and $\exists y \leq s_C(x) C(x, y)$, respectively, where
$B(x, y)$ and $C(x, z)$ are two quantifier-free $\mathcal{L}_{\PV}$-formulas and $s_B(x)$ and $s_C(x)$ are two $\mathcal{L}_{\PV}$-terms.  First, notice that without loss of generality, we can always assume that $s_B(x)=s_C(x)$ and $\forall x (s_B(x) >0)$ holds in the standard model. The reason is that we can replace $\exists y \leq s_B(x) B(x, y)$ by 
\[
\exists y \leq s_B(x)+s_C(x)+1 \,[y \leq s_B(x) \wedge B(x, y)]
\]
and similarly for $\exists z \leq s_C(x) C(x, z)$. From now on, denote both $s_B(x)$ and $s_C(x)$, by the common name $s(x)$. Moreover, notice that as $U \cap V=\varnothing$, the formula 
$\forall y \leq s(x) \neg B(x, y) \vee \forall z \leq s(x) \neg C(x, z)$ is true, for any value for $x$. Now,
let 
\[
A(x, w, y, z)=(w=0 \to \neg B(x, y)) \wedge (w \neq 0 \to \neg C(x, z)).
\]
It is clear that the formula
\[
\exists w_0w_1 \leq s(x) \forall y_0y_1z_0z_1 \leq s(x) \; [A(x, w_0, y_0, z_0) \vee A(x, w_1, y_1, z_1)]
\]
logically implies 
\[
\exists w \leq s(x) \forall yz \leq s(x) \; A(x, w, y, z)
\]
and hence the implication is provable in $\PV$. 
Unfortunately, in both formulas, some of the quantifier blocks have more than one bounded quantifiers, and hence, the formulas cannot be read as $(k, 2)$-games. However, using the pairing function and its projections available in $\PV$, it is not hard to change the formulas to $\PV$-equivalent formulas in the right form. We will avoid applying this change here as it makes everything unnecessarily complicated. Instead, we keep working with the original formulas as the one and the only exception in this paper. However, let us emphasize that whatever we claim in this example can be rewritten in a precise way using the mentioned encoding. 
Having said that, in the rest of this example, we pretend that we are working with the two $(k, 2)$-games $(A(x, w, y, z), s(x))$ and $(A(x, w_0, y_0, z_0) \vee A(x, w_1, y_1, z_1), s(x))$ and we show that there is no deterministic $(k, 2)$-reduction from the $(k, 2)$-game $(A(x, w, y, z), s(x))$ to the $(k, 2)$-game $(A(x, w_0, y_0, z_0) \vee A(x, w_1, y_1, z_1), s(x))$. For the sake of contradiction, assume that there are polynomial time computable functions $f(x, w_0, w_1)$, $g_0(x, w_0, w_1, y, z)$, $g_1(x, w_0, w_1, y, z)$, $h_0(x, w_0, w_1, y, z)$ and finally $h_1(x, w_0, w_1, y, z)$, all represented as $\mathcal{L}_{\PV}$-terms such that they read $w_0$, $w_1$, $y$ and $z$ below $s(x)$ and compute $w$, $y_0$, $y_1$, $z_0$ and $z_1$ all below $s(x)$, respectively, satisfying the property
\[
\PV \vdash \forall w_0w_1yz \leq s [(A(x, w_0, g_0, h_0) \vee A(x, w_1, g_1, h_1)) \to A(x, f, y, z)].
\]
(The arguments of the functions are omitted, for simplicity). Therefore, the formula
\[
\forall w_0w_1yz \leq s [(A(x, w_0, g_0, h_0) \vee A(x, w_1, g_1, h_1)) \to A(x, f, y, z)]
\]
is true in the standard model.
Substitute $w_0=0$ and $w_1=1$ and notice that the condition $s(x) \geq 1$ allows such a substitution.
We see that $A(x, 0, g_0, h_0)$ is equivalent to $\neg B(x, g_0)$ and $A(x, 1, g_1, h_1)$ is equivalent to $\neg C(x, h_1)$. Therefore, the following formula is true:
\[
\forall yz \leq s(x) [(\neg B(x, g_0) \vee \neg C(x, h_1)) \to A(x, f, y, z)].
\]
Therefore, we have 
\[
[\forall y \leq s(x) \neg B(x, y) \vee \forall z \leq s(x) \neg C(x, z)] \to \forall yz \leq s(x) A(x, f, y, z).
\]
Recall that as $U \cap V=\varnothing$, we know $\forall y \leq s(x) \neg B(x, y) \vee \forall z \leq s(x) \neg C(x, z)$ is true. Therefore, we reach $\forall yz \leq s(x) A(x, f, y, z)$. Now, note that if $f(x, 0, 1)=0$, the formula $A(x, f, y, z)$ is equivalent to $\neg B(x, y)$ and if $f(x, 0, 1) \neq 0$, it is equivalent to $\neg C(x, z)$. Therefore,
if $f(x, 0, 1)=0$, we have $\forall y \leq s(x) \neg B(x, y)$ and if $f(x, 0, 1) \neq 0$, we have $\forall z \leq s(x) \neg C(x, z)$. We claim that the set $S=\{x \in \mathbb{N} \mid f(x) \neq 0\}$ is a separator. First, note that $S$ is polynomial computable as $f(x, w_0, w_1)$ is a polynomial-time computable function. Secondly, it is clear that $S$ is disjoint from $V$. To show that it includes $U$, assume $x \in U$ and $f(x, 0, 1)=0$. Then, $\forall y \leq s(x) \neg B(x, y)$ which means that $x \notin U$. Therefore, we found a separator which is impossible. Hence, the claimed deterministic $(k, 2)$-reduction does not exist.
\end{example} 
As we observed in Example \ref{Impossibility}, deterministic $(k, 2)$-reductions are not even powerful enough to capture the pure logical implications between the existence of the winning strategies. To solve the problem, in the following, we strengthen the notion by relaxing the determinism in the definition.
\begin{definition}\label{2-reduction}
Let $(G(\bar{x}, y, z), t(\bar{x}))$ and $(H(\bar{x}, v, w), s(\bar{x}))$ be two $(k, 2)$-games. By a \emph{$(k, 2)$-reduction} from $(H, s)$ to $(G, t)$, we mean a finite sequence of $\mathcal{L}_{\PV_k}$-terms $f_0(\bar{x}, y)$, $f_1(\bar{x}, y, w_0)$, ..., $f_m(\bar{x}, y, w_0, \ldots, w_{m-1})$ together with an $\mathcal{L}_{\PV_k}$-term $g(\bar{x}, y, w_0, \ldots, w_{m})$ such that all the following are provable in $\PV_k$:
\begin{description}
    \item[$\bullet$]
     $\forall \bar{w} \leq s(\bar{x}) \forall y \leq t(\bar{x}) [f_i(\bar{x}, y, w_0, \ldots, w_{i-1}) \leq s(\bar{x})]$, for any $0 \leq i \leq m$.
      \item[$\bullet$]
     $\forall \bar{w} \leq s(\bar{x}) \forall y \leq t(\bar{x}) [g(\bar{x}, y, w_0, \ldots, w_{m}) \leq t(\bar{x})]$.
     \item[$\bullet$]
     $
     \forall \bar{w} \leq s(\bar{x}) \forall y \leq t(\bar{x}) [G(\bar{x}, y, g(\bar{x}, y, w_0, \ldots, w_{m})) \to \bar{H}(\bar{x}, y, w_0, \ldots, w_m)]$, where $\bar{H}(\bar{x}, y, w_0, \ldots, w_m)$ is $\bigvee_{i=0}^m H(\bar{x}, f_i(\bar{x}, y, w_0, \ldots, w_{i-1}), w_i)$.
\end{description}
\end{definition}

\begin{remark}
Here is a computational interpretation of a $(k, 2)$-reduction as a non-deterministic version of the deterministic $(k,2)$-reductions we had before. A $(k, 2)$-reduction starts with reading the first move $y \leq t(\bar{x})$ in $(G, t)$ and uses $f_0$ to transform it to a first move $v_0 \leq s(\bar{x})$ in $(H, s)$. Then, as before it reads the second move $w_0 \leq s(\bar{x})$ in $(H, s)$. However, instead of using it to find a second move in $(G, t)$, it uses $f_1$ to come up with another possible first move $v_1 \leq s(\bar{x})$ in $(H, s)$ and asks about its second move $w_1$. It keeps repeating this procedure to finally after $m+1$ many enquiries, it uses $g$ to compute the second move in $(G, t)$. These computations are in a way that if the produced play for $(G, t)$ is winning, then \emph{one} of the produced plays for $(H, s)$ is winning. 
\end{remark}
The following theorem slightly modifies Herbrand's theorem, Theorem \ref{Herbrand}, to connect the $\PV_k$-provability of the implication between the existence of the strategies and the existence of $(k, 2)$-reductions.
\begin{theorem}\label{k2Herb}
Let $(G(\bar{x}, y, z), t(\bar{x}))$ and $(H(\bar{x}, v, w), s(\bar{x}))$ be two $(k, 2)$-games. Then 
\[
\PV_k \vdash \exists y \leq t(\bar{x}) \forall z \leq t(\bar{x}) G(\bar{x}, y, z) \to \exists v \leq s(\bar{x}) \forall w \leq s(\bar{x}) H(\bar{x}, v, w)
\]
iff there exists a $(k, 2)$-reduction from $(H, s)$ to $(G, t)$.
\end{theorem}
\begin{proof}
One direction is clear. For the other, assume 
\[
\PV_k \vdash \exists y \leq t(\bar{x}) \forall z \leq t(\bar{x}) G(\bar{x}, y, z) \to \exists v \leq s(\bar{x}) \forall w \leq s(\bar{x}) H(\bar{x}, v, w).
\]
Define $\Tilde{G}(\bar{x}, y, z)$ as $[y \leq t(\bar{x}) \wedge (z \leq t(\bar{x}) \to G(\bar{x}, y, z))]$ and $\Tilde{H}(\bar{x}, v, w)$ as $[v \leq s(\bar{x}) \wedge (w \leq s(\bar{x}) \to H(\bar{x}, v, w))]$. Now, by moving the quantifiers, we have
\[
\PV_k \vdash \forall y \exists vz \forall w [\Tilde{G}(\bar{x}, y, z) \to \Tilde{H}(\bar{x}, v, w)].
\]
Using the pairing function available in $\PV$, we can make two variables $v$ and $z$ into one variable, apply Herbrand's theorem, Theorem \ref{Herbrand} and then retrieve $y$ and $z$ again, by projections. Therefore, there are $\mathcal{L}_{\PV_k}$-terms $g_0(\bar{x}, y)$,  $h_0(\bar{x}, y)$, $g_1(\bar{x}, y, w_0)$, $h_1(\bar{x}, y, w_0)$, ..., $g_m(\bar{x}, y, w_0, \ldots, w_{m-1})$ and $h_m(\bar{x}, y, w_0, \ldots, w_{m-1})$ such that
\[
\PV_k \vdash \bigvee_{i=0}^{m} [\Tilde{G}(\bar{x}, y, g_i(\bar{x}, y, w_0, \ldots, w_{i-1})) \to  \Tilde{H}(\bar{x}, h_i(\bar{x}, y, w_0, \ldots, w_{i-1}), w_i)].
\]
Define $g'(\bar{x}, y, w_0, \ldots, w_{m})$ by cases: if $\tilde{G}(\bar{x}, y, g_0(\bar{x}, y))$ is false, define $g'$ as $g_0(\bar{x}, y)$; if $\tilde{G}(\bar{x}, y, g_0(\bar{x}, y))$ is true but $\tilde{G}(\bar{x}, y, g_1(\bar{x}, y, w_0))$ is false, define $g'$ as $g_1(\bar{x}, y, w_0)$; if both $\tilde{G}(\bar{x}, y, g_0(\bar{x}, y))$ and $\tilde{G}(\bar{x}, y, g_1(\bar{x}, y, w_0))$ are true but $\tilde{G}(\bar{x}, y, g_2(\bar{x}, y, w_0, w_1))$ is false, define $g'$ as $g_2(\bar{x}, y, w_0, w_1)$ and so on. Finally, if all of $G(\bar{x}, y, g_i(\bar{x}, y, w_0, \ldots, w_{i-1}))$'s are true, define $g'$ as $0$: 
\[
g'(\bar{x}, y, w_0, \ldots, w_{m})=
\begin{cases}
g_0(\bar{x}, y) &  \neg \tilde{G}(\bar{x}, y, g_0(\bar{x}, y))\\
g_1(\bar{x}, y, w_0) & \tilde{G}(\bar{x}, y, g_0(\bar{x}, y)), \neg \tilde{G}(\bar{x}, y, g_1(\bar{x}, y, w_0))\\
$\ldots$ & $\ldots$\\
0  & o.w.
\end{cases} 
\]
Note that $g'$ is defined in a way that unless $\bigwedge_{i=0}^{m} \Tilde{G}(\bar{x}, y, g_i(\bar{x}, y, w_0, \ldots, w_{i-1}))$ is true, we always have $\neg \Tilde{G}(\bar{x}, y, g'(\bar{x}, y, w_0, \ldots, w_{m}))$.
Therefore, it is easy to see that
\[
\PV_k \vdash [\Tilde{G}(\bar{x}, y, g'(\bar{x}, y, w_0, \ldots, w_{m})) \to  \bigwedge_{i=0}^{m} \Tilde{G}(\bar{x}, y, g_i(\bar{x}, y, w_0, \ldots, w_{i-1}))],
\]
and hence
\[
\PV_k \vdash [\Tilde{G}(\bar{x}, y, g'(\bar{x}, y, w_0, \ldots, w_{m})) \to  \bigvee_{i=0}^{m}\Tilde{H}(\bar{x}, h_i(\bar{x}, y, w_0, \ldots, w_{i-1}), w_i)].
\]
Define
\[
f_i(\bar{x}, y, w_0, \ldots, w_{i-1})=
\begin{cases}
h_i(\bar{x}, y, w_0, \ldots, w_{i-1}) &  h_i(\bar{x}, y, w_0, \ldots, w_{i-1}) \leq s(\bar{x})\\
0  & h_i(\bar{x}, y, w_0, \ldots, w_{i-1}) > s(\bar{x})
\end{cases} 
\]
for any $0 \leq i \leq m$ and set
\[
g(\bar{x}, y, w_0, \ldots, w_{m})=
\begin{cases}
g'(\bar{x}, y, w_0, \ldots, w_{m}) &  g'(\bar{x}, y, w_0, \ldots, w_{m}) \leq t(\bar{x})\\
0  & g'(\bar{x}, y, w_0, \ldots, w_{m}) > t(\bar{x})
\end{cases} 
\]
It is clear that
$\PV_k \vdash f_i(\bar{x}, y, w_0, \ldots, w_{i-1}) \leq s(\bar{x})$, for any $0 \leq i \leq m$ and $\PV_k \vdash g(\bar{x}, y, w_0, \ldots, w_{m}) \leq t(\bar{x})$. It is also clear that 
\[
\forall y \leq t(\bar{x})[G(\bar{x}, y,  g(\bar{x}, y, w_0, \ldots, w_{m})) \to \Tilde{G}(\bar{x}, y,  g'(\bar{x}, y, w_0, \ldots, w_{m}))]
\]
and 
\[
\forall \bar{w} \leq s(\bar{x}) [\tilde{H}(\bar{x},  h_i(\bar{x}, y, w_0, \ldots, w_{i-1}), w_i) \to 
\]
\[
H(\bar{x}, y,  f_i(\bar{x}, y, w_0, \ldots, w_{i-1}), w_i)]
\]
are provable in $\PV_k$.
Therefore, we reach the implication
$\forall \bar{w} \leq s(\bar{x}) \forall y \leq t(\bar{x}) [G(\bar{x}, y, g(\bar{x}, y, w_0, \ldots, w_{m})) \to \bigvee_{i=0}^m H(\bar{x}, f_i(\bar{x}, y, w_0, \ldots, w_{i-1}), w_i)]$ in $\PV_k$.
\end{proof}

%\begin{remark}\label{t2-17}
%The Example \ref{t2-16} shows that pure logical deductions are far beyond the power of low level deterministic reductions. In other words, it is possible to prove $B$ by $A$ just by some elementary methods of logic but it does not mean that $B$ can be deterministically reducible to $A$. Let us explain where the problem is. At the first glance, it seems that all logical rules are completely syntactical and amenable to low complexity reductions. It is correct everywhere except for one logical rule: The contraction rule which is more or less responsible for all kinds of computational explosions like the explosion of the lengths of the proofs after the elimination of cuts. Notice that the reason that we have the equivalence in the Example \ref{t2-15} is this contraction rule and it is easy to see that this rule is the source of non-determinism and hence interactions. Therefore, it seems natural to use non-deterministic reductions to simulate computationally what is going on in the realm of proofs.   
%\end{remark}

\begin{definition}\label{k2PLS}
Let $k \geq 2$ be a natural number, $A(\bar{x}, y, z) \in \hat{\Sigma}^b_{k-1}$ be an $\mathcal{L}_{\PV_k}$-formula and $t(\bar{x})$ and $r(\bar{x})$ be two $\mathcal{L}_\PV$-terms. By a \emph{$\PLS_{(k,2)}$-program} for the pair $(A(\bar{x}, y, z), r(\bar{x}))$, we mean a $(k, 2)$-game $(G(\bar{x}, u, v, w), s(\bar{x}))$ (read $u$ as a parameter) and
\begin{itemize}
\item[$\bullet$]
an initial sequence $i(\bar{x}, w)$ of $\mathcal{L}_{\PV_k}$-terms as a $(k, 2)$-reduction from the game $(G(\bar{x}, 0, v, w), s(\bar{x}))$ to $(\top, s(\bar{x}))$,
\item[$\bullet$]
a sequence $N(\bar{x}, u, v, w)$ of $\mathcal{L}_{\PV_k}$-terms as a $(k, 2)$-reduction from the game 
$(G(\bar{x}, u+1, v, w), s(\bar{x}))$ to 
$(G(\bar{x}, u, v, w), s(\bar{x}))$,
\item[$\bullet$]
a sequence $p(\bar{x}, v, z)$ of $\mathcal{L}_{\PV_k}$-terms as a $(k, 2)$-reduction from the game $(A(\bar{x}, y, z), r(\bar{x}))$ to $(G(\bar{x}, t(\bar{x}), v, w), s(\bar{x}))$. Here, we pretend that $A(\bar{x}, y, z)$ is a quantifier-free $\mathcal{L}_{\PV_k}$-formula.
\end{itemize}
By $\PLS_{(k, 2)}$, we mean the class of all the pairs $(A(\bar{x}, y, z), r(\bar{x}))$ for which there exists a $\PLS_{(k, 2)}$-program. By $\PLS_{(k, 2)}^p$, we mean the class of all the pairs $(A(\bar{x}, y, z), r(\bar{x}))$ for which there exists a $\PLS_{(k, 2)}$-program with polynomial length, i.e., $t(\bar{x})=q(|\bar{x}|)$, for some polynomial $q$.
\end{definition}

One can read a (polynomial) $\PLS_{(k, 2)}$-program as (a polynomially) an exponentially long sequence of reductions between $2$-turn games, starting with an explicit winning strategy for the first game, where all the functions and predicates live in the $k$-th level of the polynomial hierarchy verified in $\PV_k$.\\

Similar to what we had in the last subsubsection, we can finally witness provability in $T^k_2$ (resp. $S^k_2$) by (resp. polynomial) $\PLS_{(k, 2)}$-programs.
\begin{corollary}\label{k2PLSWitnessing}
Let $k \geq 2$, $A(\bar{x}, y, z) \in \hat{\Sigma}^b_{k-2}$ and $r(\bar{x})$ be an $\mathcal{L}_{\PV}$-term:
\begin{description}
\item[$(i)$]
$S^{k}_2 \vdash \forall \bar{x} \exists y \leq r(\bar{x}) \forall z \leq r(x) A(\bar{x}, y, z)$ iff $(A(\bar{x}, y, z), r(\bar{x})) \in \PLS^p_{(k-1, 2)}$. 
\item[$(ii)$]
$T^{k}_2 \vdash \forall \bar{x} \exists y \leq r(\bar{x}) \forall z \leq r(\bar{x}) A(\bar{x}, y, z)$ iff $(A(\bar{x}, y, z), r(\bar{x})) \in \PLS_{(k-1,2)}$.
\end{description}
\end{corollary}
\begin{proof}
The proof is similar to that of Corollary \ref{k1PLSWitnessing}. Therefore, we only explain the main ingredients for $(i)$. For the right to left, assume that there is a $\PLS_{(k, 2)}$-program for $(A(\bar{x}, y, z), r(\bar{x}))$ with the length $q(|\bar{x}|)$, for some polynomial $q$.
We use Theorem \ref{k2Herb} to transform the existence of the reductions in the $\PLS_{(k, 2)}$-program to the following provable implications:
\begin{itemize}
\item[$\bullet$]
$\PV_{k-1} \vdash \exists v \leq s(\bar{x}) \forall w \leq s(\bar{x}) G(\bar{x}, 0, v, w)$.
\item[$\bullet$]
$\PV_{k-1} \vdash \exists v \leq s(\bar{x}) \forall w \leq s(\bar{x}) G(\bar{x}, u, v, w) \rightarrow \exists v \leq s(\bar{x}) \forall w \leq s(\bar{x}) G(\bar{x}, u+1, v, w)$.
\item[$\bullet$]
$\PV_{k-1} \vdash  \exists v \leq s(\bar{x}) \forall w \leq s(\bar{x}) G(\bar{x}, q(|\bar{x}|), v, w) \to \exists y \leq r(\bar{x}) \forall z \leq r(\bar{x}) A(\bar{x}, y, z)$.
\end{itemize}
As any quantifier-free $\mathcal{L}_{\PV_{k-1}}$-formula can be interpreted as an $\mathcal{L}_{\PV}$-formula in $\hat{\Pi}^b_{k-1}$ and $\PV_{k-1}$ can be interpreted in $S^{k-1}_2$, we can pretend that all the above implications are provable in $S^k_2$ and $G \in \hat{\Pi}^b_{k-1}$. Therefore, we can assume that $\exists v \leq s(\bar{x}) \forall w \leq s(\bar{x}) G(\bar{x}, u, v, w) \in \hat{\Sigma}^b_k$. Using $\LInd$ in $S^k_2$ on the formula $\exists v \leq s(\bar{x}) \forall w \leq s(\bar{x}) G(\bar{x}, u, v, w)$, we reach $S^k_2 \vdash \exists y \leq r(\bar{x}) \forall z \leq r(\bar{x}) A(\bar{x}, y, z)$. Conversely, we assume that $S^k_2 \vdash \exists y \leq r(\bar{x}) \forall z \leq r(\bar{x}) A(\bar{x}, y, z)$. Hence, $\forall y \leq r(\bar{x}) \exists z \leq r(\bar{x}) \neg A(\bar{x}, y, z) \to \bot$ is provable in $S^k_2$. Since $A \in \hat{\Sigma}^b_{k-2}$, the formula $\forall y \leq r(\bar{x}) \exists z \leq r(\bar{x}) \neg A(\bar{x}, y, z) $ is in $\hat{\Pi}^b_{k}$. Hence, by Theorem \ref{kMain}, we have $\forall y \leq r(\bar{x}) \exists z \leq r(\bar{x}) \neg A(\bar{x}, y, z) \rhd^p_k \bot$. Call the $k$-flow $(H(u, \bar{x}), t(\bar{x}))$, where $t(\bar{x})=q(|\bar{x}|)$, for some polynomial $q$. Without loss of generality, write $H(u, \bar{x})$ in the form $\forall v \leq s(\bar{x}) \exists w \leq s(\bar{x}) J(\bar{x}, u, v, w)$, where $J \in \hat{\Pi}^b_{k-2}$. As $k \geq 2$, the theory $\PV$ is a subtheory of $\PV_{k-1}$. Therefore, moving the implications in the definition of the $k$-flow from $\PV$ to $\PV_{k-1}$, we have:
\begin{itemize}
\item[$\bullet$]
$\PV_{k-1} \vdash  [\forall y \leq r(\bar{x}) \exists z \leq r(\bar{x}) \; \neg A(\bar{x}, y, z)] \to \forall v \leq s(\bar{x}) \exists w \leq s(\bar{x}) J(\bar{x}, 0, v, w)$.
\item[$\bullet$]
$\PV_{k-1} \vdash \forall v \leq s(\bar{x}) \exists w \leq s(\bar{x}) J(\bar{x}, q(|\bar{x}|), v, w) \to \bot$.
\item[$\bullet$]
$\PV_{k-1} \vdash \forall v \leq s(\bar{x}) \exists w \leq s(\bar{x}) J(\bar{x}, u, v, w) \rightarrow \forall v \leq s(\bar{x}) \exists w \leq s(\bar{x}) J(\bar{x}, u+1, v, w)$.
\end{itemize}
As $J \in \hat{\Pi}^b_{k-2}$, in $\PV_{k-1}$, we can pretend that $J$ is a quantifier-free $\mathcal{L}_{\PV_{k-1}}$-formula. Define $G(\bar{x}, u, v, w)$ as $\neg J(\bar{x}, q(|\bar{x}|) \dotminus u, v, w)$. Therefore, we have:
\begin{itemize}
\item[$\bullet$]
$\PV_{k-1} \vdash \exists v \leq s(\bar{x}) \forall w \leq s(\bar{x}) G(\bar{x}, q(|\bar{x}|), v, w) \to  \exists y \leq r(\bar{x}) \forall z \leq r(\bar{x}) A(\bar{x}, y, z) $.
\item[$\bullet$]
$\PV_{k-1} \vdash \exists v \leq s(\bar{x}) \forall w \leq s(\bar{x}) G(\bar{x}, 0, v, w)$.
\item[$\bullet$]
$\PV_{k-1} \vdash \exists v \leq s(\bar{x}) \forall w \leq s(\bar{x}) G(\bar{x}, u, v, w) \rightarrow \exists v \leq s(\bar{x}) \forall w \leq s(\bar{x}) G(\bar{x}, u+1, v, w)$.
\end{itemize}
Finally, it is enough to use Theorem \ref{k2Herb} to get a $\PLS_{(k-1, 2)}$-program for the pair $(A(\bar{x}, y, z), r(\bar{x}))$ with the length $q(|\bar{x}|)$. Hence, $(A(\bar{x}, y, z), r(\bar{x})) \in \PLS^p_{(k-1, 2)}$.
\end{proof}

It is worth putting Corollary \ref{k2PLSWitnessing} for the concrete case $k=2$ into plain words. Here, the corollary characterizes the $T^2_2$-provability (resp. $S^2_2$-provability) of a formula in the form $\forall \bar{x} \exists y \leq r(\bar{x}) \forall z \leq r(x) A(\bar{x}, y, z)$, where $A$ is a polynomial time computable predicate represented as a quantifier-free $\mathcal{L}_{\PV}$-formula by the existence of an exponentially (resp. polynomially) long sequence of polynomial time reductions between polynomial time games starting on an explicit polynomial time winning strategy in the first game.\\

\noindent \textbf{Acknowledgements:}
We are grateful to Pavel Pudl\'{a}k and Neil Thapen to put the connection between the total search problems and the theories of arithmetic into our attention. We are also thankful for the fruitful discussions we had. The support by the FWF project P 33548 is also gratefully acknowledged.

\bibliographystyle{splncs04}
\bibliography{newbib}

%\bibliographystyle{splncs04}
%\bibliography{mybibliography}

\end{document}